\documentclass[11pt,twoside]{report}
\usepackage[usecolor,DM]{uaThesis}
\usepackage{a4wide}
\usepackage{amsmath,amssymb,amsthm}
\usepackage{makeidx}

\usepackage{pagesInBib}

\usepackage{fancyhdr}

\newtheorem{theorem}{Theorem}
\newtheorem{cor}[theorem]{Corollary}
\newtheorem{lemma}[theorem]{Lemma}
\newtheorem{prop}[theorem]{Proposition}

\theoremstyle{definition}
\newtheorem{definition}[theorem]{Definition}
\newtheorem{example}[theorem]{Example}
\theoremstyle{remark}
\newtheorem{remark}[theorem]{Remark}

\DeclareMathOperator{\Dom}{Dom}

\def\ThesisYear{2010}

\makeindex

\begin{document}

\fancyhf{}

\fancyhead[LO]{\slshape \rightmark}
\fancyhead[RE]{\slshape\leftmark} \fancyfoot[C]{\thepage}


\iffalse
  \TitlePage
    \DRAFT 
    \GRID  
    \HEADER{\BAR\FIG{\epsfig{file=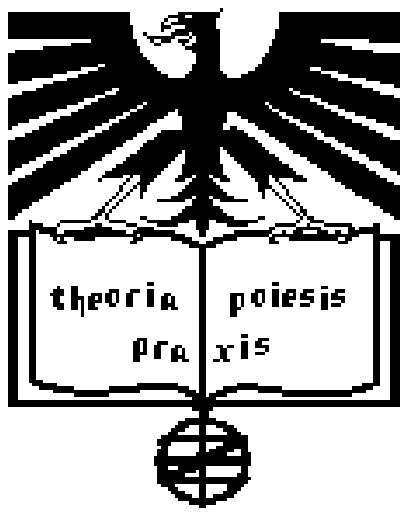,height=60mm}}}   
           {\ThesisYear}
    \TITLE{Rui A. C. Ferreira}
          {Como escrever uma tese bonita e cheia de resultados importantes}
  \EndTitlePage

\else
  \TitlePage
    \HEADER{\BAR\FIG{\begin{minipage}{80mm} 
            ``When I have clarified and exhausted a subject, then I turn away from it, in order to go into darkness again.''
             \begin{flushright}
              --- Carl Friedrich Gauss
             \end{flushright}
            \end{minipage}}}
           {\ThesisYear}
    \TITLE{Rui A. C. Ferreira}
          {Calculus of Variations on Time Scales\\ and Discrete Fractional Calculus}
  \EndTitlePage
\fi

\titlepage\ \endtitlepage 

\TitlePage
  \HEADER{}{\ThesisYear}
  \TITLE{Rui A. C. Ferreira}
        {C\'{a}lculo das Varia\c{c}\~{o}es em Escalas Temporais e C\'{a}lculo Fraccion\'{a}rio Discreto}
  \vspace*{15mm}
  \TEXT{}
       {Disserta\c c\~ao apresentada \`a Universidade de Aveiro para cumprimento dos requisitos
        necess\'arios \`a obten\c c\~ao do grau de Doutor em Matem\'{a}tica, re\-a\-li\-za\-da sob a orienta\c c\~ao
        cient\'\i fica de Delfim Fernando Marado Torres, Professor Associado do Departamento de Matem\'{a}tica da Universidade de Aveiro, e Martin Bohner,
        Professor Catedr\'{a}tico do Departamento de Matem\'{a}tica e Estat\'{\i}stica da Universidade de Ci\^{e}ncia e Tecnologia de Missouri, EUA.}
\EndTitlePage

\titlepage\ \endtitlepage 

\TitlePage
  \vspace*{55mm}
  \TEXT{\textbf{o j\'uri~/~the jury\newline}}
       {}
  \TEXT{presidente~/~president}
       {\textbf{Ant\'{o}nio Manuel Melo de Sousa Pereira}\newline {\small
        Professor Catedr\'atico da Universidade de Aveiro (por delega\c c\~ao do Reitor da
        Universidade de Aveiro)}}
  \vspace*{5mm}
  \TEXT{vogais~/~examiners committee}
  {\textbf{Martin Bohner}\newline {\small
        Professor Catedr\'atico da Universidade de Ci\^{e}ncia e Tecnologia de Missouri (co-orientador)}}
\vspace*{5mm}
  \TEXT{}
       {\textbf{Ant\'{o}nio Costa Ornelas Gon\c{c}alves}\newline {\small
        Professor Associado com Agrega\c{c}\~{a}o da Escola de Ci\^{e}ncias e Tecnologia da Universidade de \'{E}vora}}
  \vspace*{5mm}
  \TEXT{}
       {\textbf{Delfim Fernando Marado Torres}\newline {\small
        Professor Associado da Universidade de Aveiro (orientador)}}
  \vspace*{5mm}
  \TEXT{}
       {\textbf{Maria Lu\'{\i}sa Ribeiro dos Santos Morgado}\newline {\small
        Professora Auxiliar da Escola de Ci\^{e}ncias e Tecnologia da Universidade Tr\'{a}s-os-Montes e Alto Douro}}
        \vspace*{5mm}
  \TEXT{}
       {\textbf{Nat\'{a}lia da Costa Martins}\newline {\small
        Professora Auxiliar da Universidade de Aveiro}}
\EndTitlePage

\titlepage\ \endtitlepage 

\TitlePage
  \vspace*{55mm}
  \TEXT{\textbf{agradecimentos~/\newline acknowledgements}}
       {I am deeply grateful to my advisors, Delfim F. M. Torres
       and Martin Bohner, for their inestimable support during
       this thesis elaboration. I must emphasize the indisputable scientific
       support that I was provided, being a privilege to work with
       them. The hospitality of Martin Bohner
       during my stay in Rolla is also gratefully acknowledged.\\
       I also thank \emph{The Portuguese Foundation for Science and
       Technology} (FCT) for the financial support through the PhD fellowship SFRH/BD/39816/2007,
       without which this work wouldn't be possible.\\
       My next thank goes to the people with whom I interacted at the University of Aveiro, Missouri University of Science and Technology
       and the Lusophone University of Humanities and Technologies, that in some way contributed to this work.\\
      To all my friends I just want to say: "Thank you for all your encouragement
      words!".\\
       Finally, and as always, I wish to thank and dedicate this dissertation to my family.}

\vfil

\TEXT{}{\begin{center}
\includegraphics[scale=0.5]{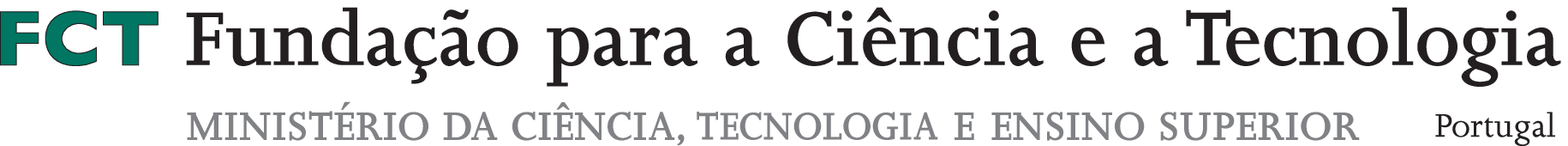}
\end{center}}

\EndTitlePage

\titlepage\ \endtitlepage 

\TitlePage
  \vspace*{55mm}
  \TEXT{\textbf{Resumo}}
       {Estudamos problemas do c\'{a}lculo das varia\c{c}\~{o}es e controlo \'{o}ptimo no
contexto das escalas temporais. Especificamente, obtemos
condi\c{c}\~{o}es necess\'{a}rias de optimalidade do tipo de
Euler--Lagrange tanto para lagrangianos
dependendo de derivadas delta de ordem superior como para problemas isoperim\'{e}tricos.
Desenvolvemos tamb\'{e}m alguns m\'{e}todos directos que permitem
resolver determinadas classes de problemas variacionais atrav\'{e}s de desigualdades
em escalas temporais. No \'{u}ltimo cap\'{\i}tulo apresentamos
operadores de diferen\c{c}a fraccion\'{a}rios e propomos um novo
c\'{a}lculo das varia\c{c}\~{o}es fraccion\'{a}rio em tempo discreto.
Obtemos as correspondentes condi\c{c}\~{o}es necess\'{a}rias
de Euler--Lagrange e Legendre, ilustrando depois a teoria
com alguns exemplos.\newline\newline
\textbf{Palavras chave:} c\'{a}lculo das varia\c{c}\~{o}es, condi\c{c}\~{o}es
necess\'{a}rias de optimalidade do tipo de Euler--Lagrange e Legendre,
controlo \'{o}ptimo, derivada delta, derivada diamond,
derivada fraccion\'{a}ria discreta, desigualdades integrais, escalas temporais.}
  \TEXT{}{}
  \vspace{5cm}

\EndTitlePage

\titlepage\ \endtitlepage 

\TitlePage
  \vspace*{55mm}
  \TEXT{\textbf{Abstract}}
       {We study problems of the calculus of variations and optimal control within the framework
       of time scales. Specifically, we obtain Euler--Lagrange type equations for both Lagrangians
       depending on higher order delta derivatives and isoperimetric problems. We also develop some direct methods
       to solve certain classes of variational problems via dynamic inequalities. In the last chapter we introduce
       fractional difference operators and propose a new discrete-time fractional calculus of variations.
       Corresponding Euler--Lagrange and Legendre necessary optimality conditions are derived
       and some illustrative examples provided.\newline\newline
}

\EndTitlePage

\titlepage\ \endtitlepage 





\pagenumbering{roman}
\addcontentsline{toc}{chapter}{Contents}
\tableofcontents

\cleardoublepage

\pagenumbering{arabic}

\chapter*{Introduction\markboth{INTRODUCTION}{}}\addcontentsline{toc}{chapter}{Introduction}

This work started when the author first met his advisor Delfim F.
M. Torres in July, 2006. In a preliminary conversation Delfim F.
M. Torres introduced to the author the Time Scales calculus.
\begin{quote}
``A major task of mathematics today is to harmonize the continuous
and the discrete, to include them in one comprehensive
mathematics, and to eliminate obscurity from both."
\end{quote}
E. T. Bell wrote this in 1937 and Stefan Hilger accomplished it in
his PhD thesis \cite{HilgerThesis} in 1988. Stefan Hilger defined
a time scale as a nonempty closed subset of the real numbers and
on such sets he defined a (delta) $\Delta$-derivative and a
$\Delta$-integral. These delta derivative and integral are the
classical derivative and Riemann integral when the time scale is
the set of the real numbers, and become the forward difference and
a sum when the time scale is the set of the integer numbers (cf.
the definitions in Chapter \ref{chap1}). This observation
highlights one of the purposes of Hilger's work: unification of
the differential and difference calculus. But at the same time,
since a time scale can be any other subset of the real numbers, an
extension was also achieved (see Remark \ref{rem1} pointing out
this fact). As written in \cite{livro}, the time scales calculus
has a tremendous potential for applications. For example, it can
model insect populations that are continuous while in season (and
may follow a difference scheme with variable step-size), die out
in (say) winter, while their eggs are incubating or dormant, and
then hatch in a new season, giving rise to a nonoverlapping
population. Motivated by its theoretical and practical usefulness
many researchers have dedicated their time to the development of
the time scales calculus in the last twenty years (cf.
\cite{livro,livro2} and references therein).

Since Delfim F. M. Torres works mainly in Calculus of Variations
and Optimal Control and in 2006 only three papers
\cite{econo,CD:Bohner:2004,zeidan}, one of them by Martin Bohner
\cite{CD:Bohner:2004}, within this topic were available in the
literature, we decided to invite Martin Bohner, whose knowledge in
time scales calculus is widely known, to be also an author's
advisor and we began to make an attempt to develop the theory.
This was indeed the main motivation for this work.

The two research papers \cite{CD:Bohner:2004} and \cite{zeidan}
were published in 2004. In the first one, necessary conditions for
weak local minima of the basic problem of the calculus of
variations were established, among them the Euler--Lagrange
equation, the Legendre necessary condition, the strengthened
Legendre condition, and the Jacobi condition, while in the second
one, a calculus of variations problem with variable endpoints was
considered in the space of piecewise rd-continuously
$\Delta$-differentiable functions. For this problem, the
Euler–-Lagrange equation, the transversality condition, and the
accessory problem were derived as necessary conditions for weak
local optimality. Moreover, assuming the coercivity of the second
variation, a corresponding second order sufficiency criterion was
established. The paper \cite{econo} was published in 2006 and
there, were also obtained necessary optimality conditions for the
basic problem of the calculus of variations but this time for the
class of functions having continuous (nabla) $\nabla$-derivative
(cf. Definition \ref{defnabla}). The authors of \cite{econo} find
this derivative most adequate to use in some problems arising in
economics.

In a first step towards the development of the theory we decided
to study a variant of the basic problem of the calculus of
variations, namely, the problem in which the Lagrangian depends
on higher order $\Delta$-derivatives. What seemed to be a natural
consequence of the basic problem (as it is in the differential
case) it turned out not to be! We only achieved partial results
for this problem, being the time scales restricted to the ones
with forward operator given by $\sigma(t)=a_1t+a_0$ for some
$a_1\in\mathbb{R}^+$ and $a_0\in\mathbb{R}$, $t \in
[a,\rho(b)]_\mathbb{T}$ (see Section \ref{higherorder} for the
discussions). Another subject that we studied was the
isoperimetric problem and here we were able to prove the
corresponding Euler--Lagrange equation on a general time scale
(cf. Theorem \ref{T1}). We were also interested in using
direct methods to solve some calculus of variations and optimal
control problems. This is particularly useful since even simple
classes of problems lead to dynamic Euler--Lagrange equations for
which methods to compute explicit solutions are very hard or not
known. In order to accomplish this, some integral
inequalities\footnote{We have also applied some integral
inequalities to solve some initial value problems on time scales
(cf. Chapter \ref{sec:Prel}).} needed to be proved on a general
time scale. This is shown in Chapter \ref{chap2}.

Another subject that interests us and that we have been working in
is the development of the discrete fractional calculus
\cite{Miller} and, in particular, of the discrete fractional
calculus of variations. Recently, two research papers
\cite{Atici0,Atici} appeared in the literature and therein it is
suggested that in the future a general theory of a
\emph{fractional time scales calculus} can exist (see the book
\cite{Miller1} for an introduction to fractional calculus and its
applications). We have made already some progresses in this field
(specifically for the time scale $h\mathbb{Z}, h>0$), namely in
the calculus of variations, and we present our results in Chapter
\ref{disfrac}.

This work is divided in two major parts. The first part has three
chapters in which we provide some preliminaries about time scales
calculus, calculus of variations, and integral inequalities. The
second part is divided in four chapters and in each we present our
original work. Specifically, in Section \ref{higherorder}, the
Euler--Lagrange equation for problems of the calculus of
variations depending on the $\Delta$-derivatives up to order
$r\in\mathbb{N}$ of a function is obtained. In Section \ref{iso},
we prove a necessary optimality condition for isoperimetric
problems. In Sections \ref{ineq} and \ref{sec:app:CV}, some
classes of variational problems are solved directly with the help
of some integral inequalities. In Section \ref{sec:mainResults}
some extensions of Gronwall's type inequalities are obtained while
in Section \ref{ineconstan}, using topological methods, we prove
that certain integrodynamic equations have $C^1[\mathbb{T}]$
solutions. In Section \ref{diamond}, we prove
H$\ddot{\mbox{o}}$lder, Cauchy--Schwarz, Minkowski and Jensen's
type inequalities in the more general setting of
$\Diamond_\alpha$-integrals. In Section \ref {duasvar}, we obtain
some Gronwall--Bellman--Bihari type inequalities for functions
depending on two time scales variables. Chapter \ref{disfrac} is
devoted to a preliminary study of the calculus of variations using
discrete fractional derivatives. In Section \ref{sec0}, we obtain
some results that in turn will be used in Section \ref{sec1} to
prove the main results, namely, an Euler--Lagrange type equation
(cf. Theorem~\ref{teorem0}) and a Legendre's necessary optimality
condition (cf. Theorem~\ref{teorem1}).

Finally, we write our conclusions in Chapter~\ref{conclusao} as
well as some projects for the future development of some subjects
herein studied.

\clearpage{\thispagestyle{empty}\cleardoublepage}

\part{Synthesis}

\clearpage{\thispagestyle{empty}\cleardoublepage}

\chapter{Time Scales Calculus: definitions and basic
results}\label{chap1}

A nonempty closed subset of $\mathbb{R}$ is called a \emph{time
scale}\index{Time Scale} and is denoted by $\mathbb{T}$. The two
most basic examples of time scales are the sets $\mathbb{R}$ and
$\mathbb{Z}$. It is clear that closed subsets of $\mathbb{R}$ may
not be connected, e.g., $[0,1]\cup[2,3]$. The following operators
are massively used within the theory:

\begin{definition}
The mapping $\sigma:\mathbb{T}\rightarrow\mathbb{T}$, defined by
$\sigma(t)=\inf{\{s\in\mathbb{T}:s>t\}}$ with
$\inf\emptyset=\sup\mathbb{T}$ (\textrm{i.e.}, $\sigma(M)=M$ if
$\mathbb{T}$ has a maximum $M$) is called \emph{forward jump
operator}\index{Forward jump operator}. Accordingly we define the
\emph{backward jump operator}\index{Backward jump operator}
$\rho:\mathbb{T}\rightarrow\mathbb{T}$ by
$\rho(t)=\sup{\{s\in\mathbb{T}:s<t\}}$ with
$\sup\emptyset=\inf\mathbb{T}$ (\textrm{i.e.}, $\rho(m)=m$ if
$\mathbb{T}$ has a minimum $m$). The symbol $\emptyset$ denotes
the empty set.
\end{definition}

The following classification of points is used within the theory:
A point $t\in\mathbb{T}$ is called
\emph{right-dense}\index{Points!right-dense},
\emph{right-scattered}\index{Points!right-scattered},
\emph{left-dense}\index{Points!left-dense} and
\emph{left-scattered}\index{Points!left-scattered} if
$\sigma(t)=t$, $\sigma(t)>t$, $\rho(t)=t$ and $\rho(t)<t$,
respectively.

Let us define the sets $\mathbb{T}^{\kappa^n}$, inductively:
$\mathbb{T}^{\kappa^1}=\mathbb{T}^\kappa=\{t\in\mathbb{T}:\mbox{$t$ non-maximal or
left-dense}\}$ and
$\mathbb{T}^{\kappa^n}$=$(\mathbb{T}^{\kappa^{n-1}})^\kappa$,
$n\geq 2$. Also, $\mathbb{T}_\kappa=\{t\in\mathbb{T}:\mbox{$t$
non-minimal or right-dense}\}$ and
$\mathbb{T}_{\kappa^n}$
$=(\mathbb{T}_{\kappa^{n-1}})_\kappa$,
$n\geq 2$. We denote $\mathbb{T}^{\kappa}\cap\mathbb{T}_{\kappa}$
by $\mathbb{T}^{\kappa}_{\kappa}$.

The functions $\mu,\nu:\mathbb{T}\rightarrow[0,\infty)$ are
defined by $\mu(t)=\sigma(t)-t$ and $\nu(t)=t-\rho(t)$.

\begin{example}
If $\mathbb{T}=\mathbb{R}$, then $\sigma(t)=\rho(t)=t$ and
$\mu(t)=\nu(t)=0$. If $\mathbb{T}=[0,1]\cup[2,3]$, then
\[ \sigma(t) = \left\{ \begin{array}{ll}
t & \mbox{if $t\in[0,1)\cup[2,3]$};\\
2 & \mbox{if $t=1$},\end{array} \right. \] while,
\[ \rho(t) = \left\{ \begin{array}{ll}
t & \mbox{if $t\in[0,1]\cup(2,3]$};\\
1 & \mbox{if $t=2$}.\end{array} \right. \] Also, we have that,
\[ \mu(t) = \left\{ \begin{array}{ll}
0 & \mbox{if $t\in[0,1)\cup[2,3]$};\\
1 & \mbox{if $t=1$},\end{array} \right. \] and,
\[ \nu(t) = \left\{ \begin{array}{ll}
0 & \mbox{if $t\in[0,1]\cup(2,3]$};\\
1 & \mbox{if $t=2$}.\end{array} \right. \]
\end{example}
For two points $a,b\in\mathbb{T}$, the time scales interval is
defined by
$$[a,b]_{\mathbb{T}}=\{t\in\mathbb{T}:a\leq t\leq b\}.$$

We now state the two definitions of differentiability on time
scales. Throughout we will frequently write
$f^\sigma(t)=f(\sigma(t))$ and $f^\rho(t)=f(\rho(t))$.

\begin{definition}\label{def2}
We say that a function $f:\mathbb{T}\rightarrow\mathbb{R}$ is
\emph{$\Delta$-differentiable}\index{Function!$\Delta$-differentiable}
at $t\in\mathbb{T}^\kappa$ if there is a number $f^{\Delta}(t)$
such that for all $\varepsilon>0$ there exists a neighborhood $U$
of $t$ (\textrm{i.e.}, $U=(t-\delta,t+\delta)\cap\mathbb{T}$ for
some $\delta>0$) such that
$$|f^\sigma(t)-f(s)-f^{\Delta}(t)(\sigma(t)-s)|
\leq\varepsilon|\sigma(t)-s|,\quad\mbox{ for all $s\in U$}.$$ We
call $f^{\Delta}(t)$ the \emph{$\Delta$-derivative} of $f$ at $t$.
\end{definition}
\begin{remark}
We note that if the number $f^\Delta(t)$ of Definition \ref{def2}
exists then it is unique (see \cite{Hilger90,Hilger97}, and also
the recent paper \cite{TSPORTU} written in Portuguese).
\end{remark}

The $\Delta$-derivative of order $n\in\mathbb{N}$ of a function
$f$ is defined by
$f^{\Delta^n}(t)=\left(f^{\Delta^{n-1}}(t)\right)^\Delta$,
$t\in\mathbb{T}^{\kappa^n}$, provided the right-hand side of the
equality exists, where $f^{\Delta^0}=f$.

\begin{definition}\label{defnabla}
We say that a function $f:\mathbb{T}\rightarrow\mathbb{R}$ is
\emph{$\nabla$-differentiable}\index{Function!$\nabla$-differentiable}
at $t\in\mathbb{T}_\kappa$ if there is a number $f^{\nabla}(t)$
such that for all $\varepsilon>0$ there exists a neighborhood $U$
of $t$ (\textrm{i.e.}, $U=(t-\delta,t+\delta)\cap\mathbb{T}$ for
some $\delta>0$) such that
$$|f^\rho(t)-f(s)-f^{\nabla}(t)(\rho(t)-s)|
\leq\varepsilon|\rho(t)-s|,\quad\mbox{ for all $s\in U$}.$$ We
call $f^{\nabla}(t)$ the \emph{$\nabla$-derivative} of $f$ at $t$.
\end{definition}

Some basic properties will now be given for the
$\Delta$-derivative. We refer the reader to the works presented in
\cite{livro} for results on $\nabla$-derivatives.

\begin{remark}
Our results are mainly proved using the $\Delta$-derivative. We
point out that it is immediate to get the analogous ones for the
$\nabla$-derivative. This is done in an elegant way using the
recent duality theory of C. Caputo \cite{caputo}.
\end{remark}

\begin{theorem}\cite[Theorem~1.16]{livro}
\label{teorema0} Assume $f:\mathbb{T}\rightarrow\mathbb{R}$ is a
function and let $t\in\mathbb{T}^\kappa$. Then we have the
following:
\begin{enumerate}
    \item If $f$ is $\Delta$-differentiable at $t$, then $f$ is
    continuous at $t$.
    \item If $f$ is continuous at $t$ and $t$ is right-scattered,
    then $f$ is $\Delta$-differentiable at $t$ with
    \begin{equation}\label{derdiscr}f^\Delta(t)=\frac{f^\sigma(t)-f(t)}{\mu(t)}.\end{equation}
    \item If $t$ is right-dense, then $f$ is $\Delta$-differentiable at
    $t$ if and only if the limit
    $$\lim_{s\rightarrow t}\frac{f(s)-f(t)}{s-t}$$
    exists as a finite number. In this case,
    $$f^\Delta(t)=\lim_{s\rightarrow t}\frac{f(s)-f(t)}{s-t}.$$
    \item If $f$ is $\Delta$-differentiable at $t$, then
    \begin{equation}\label{transfor}
    f^\sigma(t)=f(t)+\mu(t)f^\Delta(t).
    \end{equation}
\end{enumerate}
\end{theorem}
It is an immediate consequence of Theorem \ref{teorema0} that if
$\mathbb{T}=\mathbb{R}$, then the $\Delta$-derivative becomes the
classical one, i.e., $f^\Delta=f'$ while if
$\mathbb{T}=\mathbb{Z}$, the $\Delta$-derivative reduces to the
\emph{forward difference} $f^\Delta(t)=\Delta f(t)=f(t+1)-f(t)$.

\begin{theorem} \cite[Theorem~1.20]{livro}
Assume $f,g:\mathbb{T}\rightarrow\mathbb{R}$ are
$\Delta$-differentiable at $t\in\mathbb{T}^\kappa$. Then:
\begin{enumerate}
    \item The sum $f+g:\mathbb{T}\rightarrow\mathbb{R}$ is $\Delta$-differentiable at
    $t$ and $(f+g)^\Delta(t)=f^\Delta(t)+g^\Delta(t)$.
    \item For any number $\xi\in\mathbb{R}$, $\xi f:\mathbb{T}\rightarrow\mathbb{R}$ is $\Delta$-differentiable at
    $t$ and $(\xi f)^\Delta(t)=\xi f^\Delta(t)$.
    \item The product $fg:\mathbb{T}\rightarrow\mathbb{R}$ is $\Delta$-differentiable at
    $t$ and
    \begin{equation}\label{produto}
    (fg)^\Delta(t)=f^\Delta(t)g(t)+f^\sigma(t)g^\Delta(t)=f(t)g^\Delta(t)+f^\Delta(t)g^\sigma(t).
    \end{equation}
\end{enumerate}
\end{theorem}

The next lemma will be used later in this text, specifically, in
Section \ref{higherorder}. We give a proof here.
\begin{lemma}
\label{lema0} Let $t\in \mathbb{T}^\kappa$ ($t\neq\min\mathbb{T}$)
satisfy the property $\rho(t)=t<\sigma(t)$. Then, the jump
operator $\sigma$ is not continuous at $t$ and therefore not
$\Delta$-differentiable also.
\end{lemma}
\begin{proof}
We begin to prove that $\lim_{s\rightarrow t^-}\sigma(s)=t$. Let
$\varepsilon>0$ and take $\delta=\varepsilon$. Then for all
$s\in(t-\delta,t)$ we have
$|\sigma(s)-t|\leq|s-t|<\delta=\varepsilon$. Since $\sigma(t)>t$,
this implies that $\sigma$ is not continuous at $t$, hence not
$\Delta$-differentiable by 1. of Theorem \ref{teorema0}.
\end{proof}

Now we turn to \emph{$\Delta$-integration on time scales}.

\begin{definition}
A function $f:\mathbb{T}\rightarrow\mathbb{R}$ is called
\emph{rd-continuous}\index{Function!rd-continuous} if it is
continuous at right-dense points and if the left-sided limit
exists at left-dense points.
\end{definition}

\begin{remark}
We denote the set of all rd-continuous functions by
C$_{\textrm{rd}}$ or C$_{\textrm{rd}}(\mathbb{T})$, and the set of
all $\Delta$-differentiable functions with rd-continuous
derivative by C$_{\textrm{rd}}^1$ or
C$_{\textrm{rd}}^1(\mathbb{T})$.
\end{remark}

\begin{example}
Consider $\mathbb{T}=\bigcup_{k=0}^{\infty}[2k,2k+1]$. For this
time scale,
\[
\mu(t) = \left\{
\begin{array}{ll}
0 & \mbox{if $t \in\bigcup_{k=0}^\infty[2k,2k+1)$};\\
1 & \mbox{if $t \in\bigcup_{k=0}^\infty\{2k+1\}$}.
\end{array}
\right. \] Let us consider $t\in[0,1]\cap\mathbb{T}$. Then, we
have
\begin{equation*}
f^\Delta(t)=\lim_{s\rightarrow t}\frac{f(t)-f(s)}{t-s},\
t\in[0,1)\, ,
\end{equation*}
provided this limit exists, and
$$f^\Delta(1)=\frac{f(2)-f(1)}{2-1},$$
provided $f$ is continuous at $t=1$. Let $f$ be defined on
$\mathbb{T}$ by
\[
f(t) = \left\{
\begin{array}{ll}
t & \mbox{if $t\in[0,1)$};\\
2 & \mbox{if $t\geq 1$}.
\end{array}
\right.
\]
We observe that at $t=1$ $f$ is rd-continuous (since
$\lim_{t\rightarrow 1^-}f(t)=1$) but not continuous (since
$f(1)=2$).
\end{example}

\begin{definition}
A function $F:\mathbb{T}\rightarrow\mathbb{R}$ is called an
\emph{antiderivative}\index{Antiderivative} of
$f:\mathbb{T}^\kappa\rightarrow\mathbb{R}$ provided
$F^\Delta(t)=f(t)$ for all $t\in\mathbb{T}^\kappa$.
\end{definition}

\begin{theorem} \cite[Theorem~1.74]{livro}
Every rd-continuous function has an antiderivative.
\end{theorem}
Let $f:\mathbb{T}^\kappa:\rightarrow\mathbb{R}$ be a rd-continuous
function and let $F:\mathbb{T}\rightarrow\mathbb{R}$ be an
antiderivative of $f$. Then, the
\emph{$\Delta$-integral}\index{$\Delta$-integral} is defined by
$\int_s^r f(t)\Delta t=F(r)-F(s)$ for all $r,s\in\mathbb{T}$. It
satisfies
\begin{equation}
\label{sigma}
\int_t^{\sigma(t)}f(\tau)\Delta\tau=\mu(t)f(t),\quad
t\in\mathbb{T}^\kappa.
\end{equation}

\begin{theorem}\cite[Theorem~1.77]{livro}\label{teorema1}
Let $a,b,c\in\mathbb{T}$, $\xi\in\mathbb{R}$ and
$f,g\in$C$_{\textrm{rd}}(\mathbb{T}^\kappa)$. Then,
\begin{enumerate}
    \item $\int_a^b[f(t)+g(t)]\Delta t=\int_a^bf(t)\Delta t+\int_a^bg(t)\Delta
    t$.
    \item $\int_a^b(\xi f)(t)\Delta t=\xi\int_a^b f(t)\Delta t$.
    \item $\int_a^b f(t)\Delta t=-\int_b^a f(t)\Delta t$.
    \item $\int_a^bf(t)\Delta t=\int_a^cf(t)\Delta t+\int_c^bf(t)\Delta
    t$.
    \item $\int_a^af(t)\Delta t=0$.
    \item if $|f(t)|\leq g(t)$ on $[a,b]_{\mathbb{T}}^\kappa$,
    then
    $$\left|\int_a^bf(t)\Delta t\right|\leq \int_a^bg(t)\Delta t.$$
\end{enumerate}
\end{theorem}

One can easily prove \cite[Theorem~1.79]{livro} that, when
$\mathbb{T}=\mathbb{R}$ then $\int_a^bf(t)\Delta
t=\int_a^bf(t)dt$, being the right-hand side of the equality the
usual Riemann integral, and when $\mathbb{T}=\mathbb{Z}$ then
$\int_a^bf(t)\Delta t=\sum_{t=a}^{b-1}f(t)$.

\begin{remark}
It is an immediate consequence of the last item in Theorem
\ref{teorema1} that,
\begin{equation}\label{desbasic}
f(t)\leq g(t),\quad\forall t\in[a,b]_{\mathbb{T}}^\kappa\quad
\mbox{implies}\quad \int_a^b f(t)\Delta t\leq\int_a^b g(t)\Delta
t.
\end{equation}
\end{remark}

We now present the (often used in this text) integration by parts
formulas for the $\Delta$-integral:

\begin{equation}
 \int_{a}^{b}f^\sigma(t)g^{\Delta}(t)\Delta t
 =\left[(fg)(t)\right]_{t=a}^{t=b}-\int_{a}^{b}f^{\Delta}(t)g(t)\Delta
 t\label{partes1};
\end{equation}
\begin{equation}
\int_{a}^{b}f(t)g^{\Delta}(t)\Delta t
=\left[(fg)(t)\right]_{t=a}^{t=b}-\int_{a}^{b}f^{\Delta}(t)g^\sigma(t)\Delta
t.\label{partes2}
\end{equation}
\begin{remark}
For analogous results on $\nabla$-integrals the reader can
consult, e.g., \cite{livro2}.
\end{remark}

Some more definitions and results must be presented since they
will be used in the sequel. We start defining the polynomials on
time scales. Throughout the text we put
$\mathbb{N}_0=\{0,1,\ldots\}$.

\begin{definition}\label{polinomios}
The functions $g_k,h_k:\mathbb{T}^2\rightarrow\mathbb{R}$, defined
recursively by
$$g_0(t,s)=h_0(t,s)=1,\quad s,t\in\mathbb{T},$$
$$g_{k+1}(t,s)=\int_s^tg_k(\sigma(\tau),s)\Delta\tau,\quad
h_{k+1}(t,s)=\int_s^th_k(\tau,s)\Delta\tau,\quad
k\in\mathbb{N}_0,\quad s,t\in\mathbb{T},$$ are called the
\emph{polynomials on time scales}.
\end{definition}
Now we define the \emph{exponential function}. First we need the
concept of \emph{regressivity}.

\begin{definition}
A function $p:\mathbb{T}^\kappa\rightarrow\mathbb{R}$ is
\emph{regressive}\index{Function!regressive} provided
$$1+\mu(t)p(t)\neq 0$$
holds for all $t\in\mathbb{T}^\kappa$. We denote by $\mathcal{R}$
the set of all regressive and rd-continuous functions. The set of
all \emph{positively regressive} functions is defined by
$$\mathcal{R}^+=\{p\in\mathcal{R}:1+\mu(t)p(t)> 0,\ \mbox{for all}\ t\in\mathbb{T}^\kappa\}.$$
\end{definition}

\begin{theorem}\cite[Theorem~1.37]{livro2}
Suppose $p\in\mathcal{R}$ and fix $t_0\in\mathbb{T}$. Then the
initial value problem
\begin{equation}\label{exp}
y^\Delta=p(t)y,\ y(t_0)=1
\end{equation}
has a unique solution on $\mathbb{T}$.
\end{theorem}
\begin{definition}
Let $p\in\mathcal{R}$ and $t_0\in\mathbb{T}$. The exponential
function on time scales\index{Exponential function on time scales}
is defined by the solution of the IVP \eqref{exp} and is denoted
by $e_p(\cdot,t_0)$.
\end{definition}

\begin{remark}
If $\mathbb{T}=\mathbb{R}$, the exponential function is given by
$$e_p(t,t_0)=e^{\int_{t_0}^t p(\tau) d\tau},$$with
$t,t_0\in\mathbb{R}$ and $p:\mathbb{R}\rightarrow\mathbb{R}$ a
continuous function. For $\mathbb{T}=\mathbb{Z}$, the exponential
function is $$e_p(t,t_0)=\prod_{\tau=t_0}^t[1+p(\tau)],$$for
$t,t_0\in\mathbb{Z}$ with $t_0<t$ and
$p:\mathbb{Z}\rightarrow\mathbb{R}$ a sequence satisfying
$p(t)\neq -1$ for all $t\in\mathbb{Z}$. Further examples of
exponential functions can be found in \cite[Sect.~2.3]{livro}.
\end{remark}

The reader can find several properties of the exponential function
in \cite{livro}. Let us now present some results about linear
dynamic equations.

For $n\in \mathbb{N}_0$ and rd-continuous functions
$p_i:\mathbb{T}\rightarrow \mathbb{R}$, $1\leq i\leq n$, let us
consider the $n$th order linear dynamic equation
\begin{equation}
\label{linearDiffequa} Ly=0\, , \quad \text{ where }
Ly=y^{\Delta^n}+\sum_{i=1}^n p_iy^{\Delta^{n-i}} \, .
\end{equation}
A function $y:\mathbb{T}\rightarrow \mathbb{R}$ is said to be a
solution of equation (\ref{linearDiffequa}) on $\mathbb{T}$
provided $y$ is $n$ times delta differentiable on
$\mathbb{T}^{\kappa^n}$ and satisfies $Ly(t)=0$ for all
$t\in\mathbb{T}^{\kappa^n}$.

\begin{lemma}\cite[p.~239]{livro}
\label{8:88 Bohner} If $z=\left(z_1, \ldots,z_n\right) :
\mathbb{T} \rightarrow \mathbb{R}^n$ satisfies for all $t\in
\mathbb{T}^\kappa$
\begin{equation}\label{5.86}
z^\Delta=A(t)z(t),\qquad\mbox{where}\qquad A=\left(
                                               \begin{array}{ccccc}
                                                 0 & 1 & 0 & \ldots & 0 \\
                                                 \vdots & 0 & 1 & \ddots & \vdots \\
                                                 \vdots &  & \ddots & \ddots & 0 \\
                                                 0 & \ldots & \ldots & 0 & 1 \\
                                                 -p_n & \ldots & \ldots & -p_2 & -p_1 \\
                                               \end{array}
                                             \right)
\end{equation}
then $y=z_1$ is a solution of equation \eqref{linearDiffequa}.
Conversely, if $y$ solves \eqref{linearDiffequa} on $\mathbb{T}$,
then $z=\left(y, y^\Delta,\ldots, y^{\Delta^{n-1}}\right) :
\mathbb{T}\rightarrow \mathbb{R}$ satisfies \eqref{5.86} for all
$t\in\mathbb{T}^{\kappa^n}$.
\end{lemma}

\begin{definition}\cite[p.~239]{livro}
We say that equation (\ref{linearDiffequa}) is \emph{regressive}
provided $I + \mu(t) A(t)$ is invertible for all $t \in
\mathbb{T}^\kappa$, where $A$ is the matrix in \eqref{5.86}.
\end{definition}

\begin{definition}\cite[p.~243]{livro}
Let $y_k:\mathbb{T}\rightarrow\mathbb{R}$ be (m-1) times
$\Delta$-differentiable functions for all $1\leq k\leq m$. We then
define the \emph{Wronski determinant} $W=W(y_1,\ldots,y_m)$ of the
set $\{y_1,\ldots,y_m\}$ by $W(y_1,\ldots,y_m)=\det
V(y_1,\ldots,y_m)$, where
\begin{equation}
V(y_1,\ldots,y_m)=\left(
                                               \begin{array}{cccc}
                                                 y_1 & y_2 & \ldots & y_m \\
                                                 y_1^\Delta & y_2^\Delta & \ldots & y_m^\Delta \\
                                                 \vdots & \vdots &  & \vdots \\
                                                 y_1^{\Delta^{m-1}} & y_2^{\Delta^{m-1}} & \ldots &
                                                 y_m^{\Delta^{m-1}}
                                               \end{array}
                                             \right).
\end{equation}
\end{definition}
\begin{definition}
A set of solutions $\{y_1,\ldots,y_n\}$ of the regressive equation
\eqref{linearDiffequa} is called a \emph{fundamental system} for
\eqref{linearDiffequa} if there is
$t_0\in\mathbb{T}^{\kappa^{n-1}}$ such that $W(y_1,\ldots,y_n)\neq
0$.
\end{definition}

\begin{definition}\cite[p.~250]{livro}\label{cauchy}
We define the Cauchy function $y:\mathbb{T} \times
\mathbb{T}^{\kappa^n}\rightarrow \mathbb{R}$ for the linear
dynamic equation~(\ref{linearDiffequa}) to be, for each fixed
$s\in\mathbb{T}^{\kappa^n}$, the solution of the initial value
problem
\begin{equation}
\label{IVP} Ly=0,\quad
y^{\Delta^i}\left(\sigma(s),s\right)=0,\quad 0\leq i \leq
n-2,\quad y^{\Delta^{n-1}}\left(\sigma(s),s\right)=1\, .
\end{equation}
\end{definition}

\begin{theorem}\cite[p.~251]{livro}\label{eqsol}
Suppose $\{y_1,\ldots,y_n\}$ is a fundamental system of the
regressive equation~(\ref{linearDiffequa}). Let $f\in C_{rd}$.
Then the solution of the initial value problem
$$
Ly=f(t),\quad y^{\Delta^i}(t_0)=0,\quad 0\leq i\leq n-1 \, ,
$$
is given by $y(t)=\int_{t_0}^t y(t,s)f(s)\Delta s$, where $y(t,s)$
is the Cauchy function for \eqref{linearDiffequa}.
\end{theorem}

We end this section enunciating three theorems, namely, the
derivative under the integral sign, a chain rule\index{Chain rule
on time scales} and a mean value theorem on time scales\index{Mean
value theorem on time scales}.

\begin{theorem}\cite[Theorem~1.117]{livro}  \label{teork}
Let $t_0\in\mathbb{T}^\kappa$ and assume
$k:\mathbb{T}\times\mathbb{T}\rightarrow\mathbb{R}$ is continuous
at $(t,t)$, where $t\in\mathbb{T}^\kappa$ with $t>t_0$. In
addition, assume that $k(t,\cdot)$ is rd-continuous on
$[t_0,\sigma(t)]$. Suppose that for each $\varepsilon>0$ there
exists a neighborhood $U$ of $t$, independent of
$\tau\in[t_0,\sigma(t)]$, such that
$$|k(\sigma(t),\tau)-k(s,\tau)-k^{\Delta_1}(t,\tau)(\sigma(t)-s)|
\leq\varepsilon|\sigma(t)-s|\ \mbox{for all $s\in U$},$$ where
$k^{\Delta_1}$ denotes the delta derivative of $k$ with respect to
the first variable. Then,
$$g(t):=\int_{t_0}^t k(t,\tau)\Delta\tau\ \mbox{implies}\ g^\Delta(t)=\int_{t_0}^t k^{\Delta_1}(t,\tau)\Delta\tau+k(\sigma(t),t).$$
\end{theorem}

\begin{theorem}\cite[Theorem~1.90]{livro}
\label{teor1} Let $f:\mathbb{R}\rightarrow\mathbb{R}$ be
continuously differentiable and suppose
$g:\mathbb{T}\rightarrow\mathbb{R}$ is $\Delta$-differentiable on
$\mathbb{T}^\kappa$. Then, $f\circ g$ is $\Delta$-differentiable
and the formula
$$(f\circ g)^\Delta(t)=\left\{\int_0^1f'(g(t)+h\mu(t)g^\Delta(t))dh\right\}g^\Delta(t),\quad t\in\mathbb{T}^\kappa$$
holds.
\end{theorem}

\begin{theorem}\cite[Theorem~1.14]{livro2}\label{meanvaluethm}
Let $f$ be a continuous function on $[a,b]_\mathbb{T}$ that is
$\Delta$-differentiable on $[a,b)_\mathbb{T}$. Then there exist
$\xi,\tau\in[a,b)_\mathbb{T}$ such that
$$f^\Delta(\xi)\leq\frac{f(b)-f(a)}{b-a}\leq f^\Delta(\tau).$$
\end{theorem}

\clearpage{\thispagestyle{empty}\cleardoublepage}

\chapter{Calculus of Variations and Optimal Control}

The calculus of variations deals with finding extrema and, in this
sense, it can be considered a branch of optimization. The problems
and techniques in this branch, however, differ markedly from those
involving the extrema of functions of several variables owing to
the nature of the domain on the quantity to be optimized. The
calculus of variations is concerned with finding extrema for
functionals, i.e., for mappings from a set of functions to the
real numbers. The candidates in the competition for an extremum
are thus functions as opposed to vectors in $\mathbb{R}^n$, and
this furnishes the subject a distinct character. The functionals
are generally defined by definite integrals; the set of functions
are often defined by boundary conditions and smoothness
requirements, which arise in the formulation of the problem/model.
Let us take a look at the classical (basic) problem\index{Calculus
of Variations!classical problem} of the calculus of variations:
find a function $y\in C^1[a,b]$ such that
\begin{equation}\label{Prob0}
\mathcal{L}[y(\cdot)]=\int_{a}^{b}L(t,y(t),y'(t))d t
\longrightarrow \min ,\quad y(a)=y_a,\quad y(b)=y_b,
\end{equation}
with $a,b,y_a,y_b\in\mathbb{R}$ and $L(t,u,v)$ satisfying some
smoothness properties.

The enduring interest in the calculus of variations is in part due
to its applications. We now present a historical example of this.

\begin{example}[Brachystochrones]
The history of the calculus of variations essentially begins with
a problem posed by Johann Bernoulli (1696) as a challenge to the
mathematical comunity and in particular to his brother Jacob. The
problem is important in the history of the calculus of variations
because the method developed by Johann's pupil, Euler, to solve
this problem provided a sufficiently general framework to solve
other variational problems \cite{Brunt}.

The problem that Johann posed was to find the shape of a wire
along which a bead initially at rest slides under gravity from one
end to the other in minimal time. The endpoints of the wire are
specified and the motion of the bead is assumed frictionless. The
curve corresponding to the shape of the wire is called a
\emph{brachystochrone} or a curve of fastest descent.

The problem attracted the attention of various mathematicians
throughout the time including Huygens, L'H\^{o}pital, Leibniz, Newton,
Euler and Lagrange (see \cite{Brunt} and references cited therein
for more historical details).

To model Bernoulli's problem we use Cartesian coordinates with the
positive $y$-axis oriented in the direction of the gravitational
force. Let $(a,y_a)$ and $(b,y_b)$ denote the coordinates of the
initial and final positions of the bead, respectively. Here, we
require that $a<b$ and $y_a<y_b$. The problem consists of
determining, among the curves that have $(a,y_a)$ and $(b,y_b)$ as
endpoints, the curve on which the bead slides down from $(a,y_a)$
to $(b,y_b)$ in minimum time. The problem makes sense only for
continuous curves. We make the additional simplifying (but
reasonable) assumptions that the curve can be represented by a
function $y:[a,b]\rightarrow\mathbb{R}$ and that $y$ is at least
piecewise differentiable in the interval $[a,b]$. Now, the total
time it takes the bead to slide down a curve is given by
\begin{equation}\label{brach0}
T[y(\cdot)]=\int_0^l\frac{ds}{v(s)},
\end{equation}
where $l$ denotes the arclength of the curve, $s$ is the arclength
parameter, and $v$ is the velocity of the bead $s$ units down the
curve from $(a,b)$.

We now derive an expression for the velocity in terms of the
function $y$. We use the law of conservation of energy to achieve
this. At any position $(x,y(x))$ on the curve, the sum of the
potential and kinetic energies of the bead is a constant. Hence
\begin{equation}\label{brach1}
\frac{1}{2}mv^2(x)+mgy(x)=c,
\end{equation}
where $m$ is the mass of the bead, $v$ is the velocity of the bead
at $(x,y(x))$, and $c$ is a constant. Solving equation
\eqref{brach1} for $v$ gives
$$v(x)=\sqrt{\frac{2c}{m}-2gy(x)}.$$
Equality \eqref{brach0} becomes
$$T[y(\cdot)]=\int_a^b\frac{\sqrt{1+y'^2(x)}}{\sqrt{\frac{2c}{m}-2gy(x)}}dx.$$
We thus seek a function $y$ such that $T$ is minimum and
$y(a)=y_a$, $y(b)=y_b$.

It can be shown that the extrema for $T$ is a portion of the curve
called \emph{cycloid} (cf. Example 2.3.4 in \cite{Brunt}).
\end{example}

Let us return to the problem given in \eqref{Prob0} and write the
first and second order necessary optimality conditions for it,
i.e., the well-known Euler--Lagrange equation and Legendre's
necessary condition, respectively.

\begin{theorem}\label{teor2}(cf., e.g., \cite{Brunt})
Suppose that $L_{uu}$, $L_{uv}$ and $L_{vv}$ exist and are
continuous. Let $\hat{y}\in C^1[a,b]$ be a solution to the problem
given in \eqref{Prob0}. Then, necessarily
\begin{enumerate}
    \item
    $\frac{d}{dt}L_v(t,\hat{y}(t),\hat{y}'(t))=L_u(t,\hat{y}(t),\hat{y}'(t)),\quad
    t\in[a,b]\quad\mbox{\emph{(Euler--Lagrange equation\index{Euler--Lagrange equation!continuos case})}}$.
    \item $L_{vv}(t,\hat{y}(t),\hat{y}'(t))\geq 0,\quad t\in[a,b]\quad\mbox{\emph{(Legendre's
    condition\index{Legendre's necessary condition!continuous case})}}$.
\end{enumerate}
\end{theorem}

It can be constructed a discrete analogue of problem
\eqref{Prob0}: Let $a,b\in\mathbb{Z}$. Find a function $y$ defined
on the set $\{a,a+1,\ldots,b-1,b\}$ such that
\begin{equation}\label{Prob1}
\mathcal{L}[y(\cdot)]=\sum_{t=a}^{b-1}L(t,y(t+1),\Delta y(t))
\longrightarrow \min ,\quad y(a)=y_a,\quad y(b)=y_b,
\end{equation}
with $y_a,y_b\in\mathbb{R}$. For this problem we have the
following result.
\begin{theorem}\label{teoremadisc}(cf., e.g., \cite{discCV})
Suppose that $L_{uu}$, $L_{uv}$ and $L_{vv}$ exist and are
continuous. Let $\hat{y}$ be a solution to the problem given in
\eqref{Prob1}. Then, necessarily
\begin{enumerate}
    \item
    $\Delta L_v(t)=L_u(t),\quad
    t\in\{a,\ldots,b-2\}\quad\mbox{\emph{(discrete Euler--Lagrange equation\index{Euler--Lagrange equation!discrete case})}}$.
    \item $L_{vv}(t)+2L_{uv}(t)+L_{uu}(t)+L_{vv}(t+1)\geq 0$,
    $t\in\{a,\ldots,b-2\}$, (discrete Legendre's
    condition\index{Legendre's necessary condition!discrete case}),
\end{enumerate}
where $(t)=(t,\hat{y}(t+1),\Delta\hat{y}(t))$.
\end{theorem}
In 2004, M. Bohner wrote a paper \cite{CD:Bohner:2004} in which it
was introduced the calculus of variations on time scales. We now
present some of the results obtained.
\begin{definition}
A function $f$ defined on $[a,b]_\mathbb{T}\times\mathbb{R}$ is
called \emph{continuous in the second variable, uniformly in the
first variable}, if for each $\varepsilon>0$ there exists
$\delta>0$ such that $|x_1-x_2|<\delta$ implies
$|f(t,x_1)-f(t,x_2)|<\varepsilon$ for all $t\in[a,b]_\mathbb{T}$.
\end{definition}
\begin{lemma}[cf. Lemma 2.2 in \cite{CD:Bohner:2004}]
Suppose that $F(x)=\int_a^bf(t,x)\Delta t$ is well defined. If
$f_x$ is continuous in $x$, uniformly in $t$, then
$F'(x)=\int_a^bf_x(t,x)\Delta t$.
\end{lemma}

We now introduce the basic problem of the calculus of variations
on time scales: Let $a,b\in\mathbb{T}$ and
$L(t,u,v):[a,b]^\kappa_\mathbb{T}\times\mathbb{R}^2\rightarrow\mathbb{R}$
be a function. Find a function $y\in \textrm{C}_{\textrm{rd}}^1$
such that
\begin{equation}\label{Prob2}
\mathcal{L}[y(\cdot)]=\int_{a}^{b}L(t,y^\sigma(t),y^\Delta(t))\Delta
t \longrightarrow \min ,\quad y(a)=y_a,\quad y(b)=y_b,
\end{equation}
with $y_a,y_b\in\mathbb{R}$.

\begin{definition}
For $f\in\textrm{C}_{\textrm{rd}}^1$ we define the norm
$$\|f\|=\sup_{t\in[a,b]^\kappa_\mathbb{T}}|f^\sigma(t)|+\sup_{t\in[a,b]^\kappa_\mathbb{T}}|f^\Delta(t)|.$$
A function $\hat{y}\in \textrm{C}_{\textrm{rd}}^1$ with
$\hat{y}(a)=y_a$ and $\hat{y}(b)=y_b$ is called a (weak) local
minimum for problem \eqref{Prob2} provided there exists $\delta>0$
such that $\mathcal{L}(\hat{y})\leq\mathcal{L}(\hat{y})$ for all
$y\in \textrm{C}_{\textrm{rd}}^1$ with $y(a)=y_a$ and $y(b)=y_b$
and $\|y-\hat{y}\|<\delta$.
\end{definition}
\begin{definition}
A function $\eta\in\textrm{C}_{\textrm{rd}}^1$ is called an
\emph{admissible variation} provided $\eta\neq 0$ and
$\eta(a)=\eta(b)=0$.
\end{definition}
\begin{lemma}[cf. Lemma 3.4 in \cite{CD:Bohner:2004}]\label{lema1}
Let $y,\eta\in\textrm{C}_{\textrm{rd}}^1$ be arbitrary fixed
functions. Put
$f(t,\varepsilon)=L(t,y^\sigma(t)+\varepsilon\eta^\sigma(t),y^\Delta(t)+\varepsilon\eta^\Delta(t))$
and $\Phi(\varepsilon)=\mathcal{L}(y+\varepsilon\eta)$,
$\varepsilon\in\mathbb{R}$. If $f_\varepsilon$ and
$f_{\varepsilon\varepsilon}$ are continuous in $\varepsilon$,
uniformly in $t$, then
\begin{align}
\Phi'(\varepsilon)&=\int_a^b[L_u(t,y^\sigma(t),y^\Delta(t))\eta^\sigma(t)+L_v(t,y^\sigma(t),y^\Delta(t))\eta^\Delta(t)]\Delta t,\nonumber\\
\Phi''(\varepsilon)&=\int_a^b\{L_{uu}[y](t)(\eta^\sigma(t))^2+2\eta^\sigma(t)L_{uv}[y](t)\eta^\Delta(t)+L_{vv}[y](t)(\eta^\Delta(t))^2\}\Delta
t,\nonumber
\end{align}
where $[y](t)=(t,y^\sigma(t),y^\Delta(t))$.
\end{lemma}
The next lemma is the time scales version of the classical
Dubois--Reymond lemma\index{Dubois–-Reymond lemma on time scales}.
\begin{lemma}[cf. Lemma 4.1 in \cite{CD:Bohner:2004}]
\label{lem:DR} Let $g\in
C_{\textrm{rd}}([a,b]_\mathbb{T}^\kappa)$. Then,
$$\int_{a}^{b}g(t)\eta^\Delta(t)\Delta t=0,\quad
\mbox{for all $\eta\in C_{\textrm{rd}}^1([a,b]_\mathbb{T})$ with
$\eta(a)=\eta(b)=0$},$$ holds if and only if
$$g(t)=c,\quad\mbox{on $[a,b]^\kappa_\mathbb{T}$ for some
$c\in\mathbb{R}$}.$$
\end{lemma}

Nest theorem contains first and second order necessary optimality
conditions for problem defined by \eqref{Prob2}. Its proof can be
seen in \cite{CD:Bohner:2004}.

\begin{theorem}
Suppose that $L$ satisfies the assumption of Lemma \ref{lema1}. If
$\hat{y}\in\textrm{C}_{\textrm{rd}}^1$ is a (weak) local minimum
for problem given by \eqref{Prob2}, then necessarily
\begin{enumerate}
    \item
    $ L_v^\Delta[\hat{y}](t)=L_u[\hat{y}](t),\quad
    t\in[a,b]^{\kappa^2}_\mathbb{T}\quad\mbox{\emph{(time scales Euler--Lagrange equation\index{Euler--Lagrange equation!time scales case})}}$.
    \item $L_{vv}[\hat{y}](t)+\mu(t)\{2L_{uv}[\hat{y}](t)+\mu(t)L_{uu}[\hat{y}](t)+(\mu^\sigma(t))^\ast L_{vv}[\hat{y}](\sigma(t))\}\geq 0,\quad t\in[a,b]^{\kappa^2}_\mathbb{T}\quad\\\mbox{\emph{(time scales Legendre's
    condition\index{Legendre's necessary condition!time scales case})}}$,
\end{enumerate}
where $[y](t)=(t,y^\sigma(t),y^\Delta(t))$ and
$\alpha^\ast=\frac{1}{\alpha}$ if
$\alpha\in\mathbb{R}\backslash\{0\}$ and $0^\ast=0$.
\end{theorem}
\begin{example}
Consider the following problem
\begin{equation}
\mathcal{L}[y(\cdot)]=\int_{a}^{b}[y^\Delta(t)]^2\Delta t
\longrightarrow \min ,\quad y(a)=y_a,\quad y(b)=y_b.
\end{equation}
Its Euler--Lagrange equation is $y^{\Delta^2}(t)=0$ for all
$t\in[a,b]^{\kappa^2}_\mathbb{T}$. It follows that $y(t)=ct+d$,
and the constants $c$ and $d$ are obtained using the boundary
conditions $y(a)=y_a$ and $y(b)=y_b$. Note that Legendre's
condition is always satisfied in this problem since $\mu(t)\geq 0$
and $L_{uu}=L_{uv}=0$, $L_{vv}=2$.
\end{example}

\clearpage{\thispagestyle{empty}\cleardoublepage}

\chapter{Inequalities}

Inequalities have proven to be one of the most important tools for
the development of many branches of mathematics. Indeed, this
importance seems to have increased considerably during the last
century and the theory of inequalities may nowadays be regarded as
an independent branch of mathematics. A particular feature that
makes the study of this interesting topic so fascinating arises
from the numerous fields of applications, such as fixed point
theory and calculus of variations.

The integral inequalities of various types have been widely
studied in most subjects involving mathematical analysis. In
recent years, the application of integral inequalities has greatly
expanded and they are now used not only in mathematics but also in
the areas of physics, technology and biological sciences
\cite{Pch}. Moreover, many physical problems arising in a wide
variety of applications are governed by both ordinary and partial
difference equations \cite{book:DCV} and, since the integral
inequalities with explicit estimates are so important in the study
of properties (including the existence itself) of solutions of
differential and integral equations, their discrete analogues
should also be useful in the study of properties of solutions of
difference equations.

An early significant integral inequality and certainly a keystone
in the development of the theory of differential equations can be
stated as follows.
\begin{theorem}\label{thm8}
If $u\in C[a,b]$ is a nonnegative function and
$$u(t)\leq c+\int_a^tdu(s)ds,$$
for all $t\in[a,b]$ where $c,d$ are nonnegative constants, then
the function $u$ has the estimate
$$u(t)\leq c\exp(d(t-a)),\quad t\in[a,b].$$
\end{theorem}
The above theorem was discovered by Gronwall \cite{GronwallIne} in
1919 and is now known as \emph{Gronwall's inequality}. The
discrete version of Theorem \ref{thm8} seems to have appeared
first in the work of Mikeladze \cite{mike} in 1935.

In 1956, Bihari \cite{bihari} gave a nonlinear generalization of
Gronwall's inequality by studying the inequality
$$u(t)\leq c+\int_a^t dw(u(s))ds,$$
where $w$ satisfies some prescribed conditions, which is of
fundamental importance in the study of nonlinear problems and is
known as \emph{Bihari's inequality}.

Other fundamental inequalities as, the H$\ddot{\mbox{o}}$lder (in
particular, Cauchy--Schwarz) inequality, the Minkowski inequality
and the Jensen inequality caught the fancy of a number of research
workers. Let us state and then provide a simple example of
application (specifically to a calculus of variations problem) of
the Jensen inequality.
\begin{theorem}\emph{\cite{mit}}
\label{thm1} If $g \in C([a, b], (c, d))$ and $f \in C((c, d),
\mathbb{R} )$ is convex, then
$$
f \left( \frac{\int_{a}^{b}g(s) ds}{b-a}\right ) \leq
\frac{\int_{a}^{b}f(g(s)) ds}{b-a}.
$$
\end{theorem}
\begin{example}
Consider the following calculus of variations problem: Find the
minimum of the functional $\mathcal{L}$ defined by
$$\mathcal{L}[y(\cdot)]=\int_0^1 [y'(t)]^4 dt,$$
with the boundary conditions $y(0)=0$ and $y(1)=1$. Since
$f(x)=x^4$ is convex for all $x\in\mathbb{R}$ we have that
$$\mathcal{L}[y(\cdot)]\geq\left(\int_0^1 y'(t) dt\right)^4=1,$$
for all $y\in C^1[0,1]$. Now, the function $\hat{y}(t)=t$ satisfies
the boundary conditions and is such that
$\mathcal{L}[\hat{y}(t)]=1$. Therefore the functional
$\mathcal{L}$ achieves its minimum at $\hat{y}$.
\end{example}

At the time of the beginning of this work there was already work
done in the development of inequalities of the above mentioned
type, i.e., Jensen, Gronwall, Bihari, etc., within the time scales
setting (cf. \cite{inesurvey,Bihary,rev1:r,Gronwall,wong}). We now
state two of them being the others presented when required. The
first one is a comparison theorem.

\begin{theorem}[Theorem 5.4 of \cite{inesurvey}]\label{teo3}
Let $a\in\mathbb{T}$, $y\in C_\textrm{rd}(\mathbb{T})$ and $f\in
C_\textrm{rd}(\mathbb{T}^\kappa)$ and $p\in\mathcal{R}^+$. Then,
$$y^\Delta(t)\leq p(t)y(t)+f(t),\quad t\in\mathbb{T}^\kappa,$$
implies
$$y(t)\leq y(a)e_p(t,a)+\int_a^t e_p(t,\sigma(\tau))f(\tau)\Delta\tau,\quad t\in\mathbb{T}.$$
\end{theorem}

The next theorem presents Gronwall's inequality on time scales and
can be found in \cite[Theorem~5.6]{inesurvey}.
\begin{theorem}[Gronwall's inequality on time scales]\index{Gronwall's inequality on time scales}
\label{gronw} Let $t_0\in\mathbb{T}$. Suppose that
$u,a,b\in$C$_{\textrm{rd}}(\mathbb{T})$ and $b\in\mathcal{R}^+$,
$b\geq 0$. Then,
$$u(t)\leq a(t)+\int_{t_0}^t b(\tau)u(\tau)\Delta\tau\quad \mbox{for all}\quad t\in\mathbb{T}$$
implies
$$u(t)\leq a(t)+\int_{t_0}^t a(\tau)b(\tau)e_b(t,\sigma(\tau))\Delta\tau\quad \mbox{for all}\quad t\in\mathbb{T}.$$
\end{theorem}
If we consider the time scale $\mathbb{T}=h\mathbb{Z}$ we get,
from Theorem \ref{gronw}, a discrete version of the Gronwall
inequality (cf. \cite[Example~5.1]{inesurvey}).
\begin{cor}
If $c,d$ are two nonnegative constants, $a,b\in h\mathbb{Z}$ and
$u$ is a function defined on $[a,b]\cap h\mathbb{Z}$, then the
inequality
$$u(t)\leq c+\sum_{k=\frac{a}{h}}^{\frac{t}{h}-1}du(kh)h,\quad t\in [a,b]\cap h\mathbb{Z},$$
implies
$$u(t)\leq c+\sum_{k=\frac{a}{h}}^{\frac{t}{h}-1}cd(1+dh)^{\frac{t-h(k+1)}{h}}h,\quad t\in [a,b]\cap h\mathbb{Z}.$$
\end{cor}

\begin{remark}\label{rem1}
We note that many new inequalities were accomplished by the fact
that the proofs are done in a general time scale. For example, to
the best of our knowledge, no Gronwall's inequality was known for
the time scale
$\mathbb{T}=\bigcup_{k\in\mathbb{Z}}[k,k+\frac{1}{2}]$. For this
time scale, if $a,b\in\mathbb{T}$ ($a<b$), then the
$\Delta$-integral is (see \cite{Hilger90}):
$$\int_a^bf(t)\Delta t=\int_a^{[a]}f(t)dt+\sum_{k=[a]}^{[b]-1}\left[\int_k^{k+\frac{1}{2}}f(t)dt+\frac{1}{2}f(k+\frac{1}{2})\right]+\int_{[b]}^{b}f(t)dt,$$
where $[t]$ is the Gauss bracket.
\end{remark}

\clearpage{\thispagestyle{empty}\cleardoublepage}

\part{Original Work}

\clearpage{\thispagestyle{empty}\cleardoublepage}

\chapter{Calculus of Variations on Time Scales}

In this chapter we will present some of our achievements made in
the calculus of variations within the time scales setting. In
Section \ref{higherorder}, we present the Euler--Lagrange equation
for the problem of the calculus of variations depending on the
$\Delta$-derivative of order $r\in\mathbb{N}$ of a function. In
Section \ref{iso}, we prove a necessary optimality condition for
isoperimetric problems.

\section{Higher order $\Delta$-derivatives}\label{higherorder}

Working with functions that have more than the first
$\Delta$-derivative on a general time scale can cause problems.
Consider the simple example of $f(t)=t^2$. It is easy to see that
$f^\Delta$ exists and $f^\Delta(t)=t+\sigma(t)$. However, if for
example we take $\mathbb{T}=[0,1]\cup[2,3]$, then $f^{\Delta^2}$
doesn't exist at $t=1$ because $t+\sigma(t)$ is not continuous at
that point.

Here, we consider time scales such that:

\begin{description}
\item[(H)] $\sigma(t)=a_1t+a_0$ for some $a_1\in\mathbb{R}^+$ and
$a_0\in\mathbb{R}$, $t \in [a,\rho(b)]_\mathbb{T}$.
\end{description}

Under hypothesis (H) we have, among others, the differential
calculus ($\mathbb{T}=\mathbb{R}$, $a_1 = 1$, $a_0 = 0$), the
difference calculus ($\mathbb{T}=\mathbb{Z}$, $a_1 = a_0 = 1$) and
the quantum calculus ($\mathbb{T}=\{q^k:k\in\mathbb{N}_0\}$, with
$q>1$,  $a_1 = q$, $a_0 = 0$).

\begin{remark}
\label{remark0} From assumption (H) it follows by
Lemma~\ref{lema0} that it is not possible to have points which are
simultaneously left-dense and right-scattered. Also points that
are simultaneously left-scattered and right-dense do not occur,
since $\sigma$ is strictly increasing.
\end{remark}

\begin{lemma}[cf. \cite{Ferr3}]
Under hypothesis (H), if $f$ is a two times
$\Delta$-differentiable function, then the next formula holds:
\begin{equation}
\label{rui-1} f^{\sigma\Delta}(t)=a_1 f^{\Delta\sigma}(t),\quad
t\in\mathbb{T}^{\kappa^2}.
\end{equation}
\end{lemma}
\begin{proof}
We have
$f^{\sigma\Delta}(t)=\left[f(t)+\mu(t)f^\Delta(t)\right]^\Delta$
by formula (\ref{transfor}). By the hypothesis on $\sigma$,
function $\mu$ is $\Delta$-differentiable, hence
$\left[f(t)+\mu(t)f^\Delta(t)\right]^\Delta
=f^\Delta(t)+\mu^\Delta(t)f^{\Delta\sigma}(t)+\mu(t)f^{\Delta^2}(t)$
and applying again formula (\ref{transfor}) we obtain
$f^\Delta(t)+\mu^\Delta(t)f^{\Delta\sigma}(t)+\mu(t)f^{\Delta^2}(t)
=f^{\Delta\sigma}(t)+\mu^\Delta(t)f^{\Delta\sigma}(t)
=(1+\mu^\Delta(t))f^{\Delta\sigma}(t)$. Now we only need to
observe that $\mu^\Delta(t)=\sigma^\Delta(t)-1$ and the result
follows.
\end{proof}

We consider the Calculus of Variations problem in which the
Lagrangian depends on $\Delta$-derivatives up to order
$r\in\mathbb{N}$. Our extension for problem given in \eqref{Prob2}
is now enunciated:

\begin{equation}
\label{problema }
\begin{gathered}
\mathcal{L}[y(\cdot)]=\int_{a}^{\rho^{r-1}(b)}
L(t,y^{\sigma^r}(t),y^{\sigma^{r-1}\Delta}(t),\ldots,
y^{\sigma\Delta^{r-1}}(t),y^{\Delta^r}(t))\Delta t\longrightarrow\min, \\
y(a)=y_a,  \ \  y\left(\rho^{r-1}(b)\right)=y_b ,\\
\vdots\\ \tag{P} y^{\Delta^{r-1}}(a)=y_a^{r-1},\ \
y^{\Delta^{r-1}}\left(\rho^{r-1}(b)\right)=y_b^{r-1},
\end{gathered}
\end{equation}

Assume that the Lagrangian $L(t,u_0,u_1,\ldots,u_r)$ of problem
(\ref{problema }) has (standard) partial derivatives with respect
to $u_0,\ldots,u_r$, $r\geq 1$, and partial  $\Delta$-derivative
with respect to $t$ of order $r+1$. Let $y\in\mathrm{C}^{2r}$,
where
$$\mathrm{C}^{2r}=\left\{y:[a,b]_\mathbb{T}\rightarrow\mathbb{R}:
y^{\Delta^{2r}}\ \mbox{is continuous on}\
\mathbb{T}^{\kappa^{2r}}\right\}.$$

\begin{remark}
We will assume from now on and until the end of this section that
the time scale $\mathbb{T}$ has at least $2r+1$ points. Indeed, if
the time scale has only $2r$ points, then it can be written as
$\mathbb{T}=\{a,\sigma(a),\ldots,\sigma^{2r-1}(a)\}$. Noting that
$\rho^{r-1}(b)=\rho^{r-1}(\sigma^{2r-1}(a))=\sigma^r(a)$ and using
formula \eqref{derdiscr} we conclude that $y(t)$ for
$t\in\{a,\ldots,\sigma^{r-1}(a)\}$ can be determined using the
boundary conditions
$[y(a)=y_a,\ldots,y^{\Delta^{r-1}}(a)=y_a^{r-1}]$ and $y(t)$ for
$t\in\{\sigma^r(a),\ldots,\sigma^{2r-1}(a)\}$ can be determined
using the boundary conditions
$[y(\rho^{r-1}(b))=y_b,\ldots,y^{\Delta^{r-1}}(\rho^{r-1}(b))=y^{r-1}_b]$.
Therefore we would have nothing to minimize in problem
\eqref{problema }.
\end{remark}

\begin{definition}
We say that $y_\ast\in\mathrm{C}^{2r}$ is a \emph{weak local
minimum} for (\ref{problema }) provided there exists $\delta>0$
such that $\mathcal{L}(y_\ast)\leq\mathcal{L}(y)$ for all
$y\in\mathrm{C}^{2r}$ satisfying the constraints in (\ref{problema
}) and $\|y-y_\ast\|_{r,\infty} < \delta$, where
$$||y||_{r,\infty} := \sum_{i=0}^{r} ||y^{(i)}||_{\infty},$$
with $y^{(i)} = y^{\sigma^i\Delta^{r-i}}$ and $||y||_{\infty}:=
\sup_{t \in\mathbb{T}^{k^r}} |y(t)|$.
\end{definition}

\begin{definition}
We say that $\eta\in C^{2r}$ is an admissible variation for
problem (\ref{problema }) if
\begin{align}
\eta(a)=0,& \quad \eta\left(\rho^{r-1}(b)\right)=0 \nonumber\\
&\vdots\nonumber \\
\eta^{\Delta^{r-1}}(a)=0,&\quad
\eta^{\Delta^{r-1}}\left(\rho^{r-1}(b)\right)=0.\nonumber
\end{align}
\end{definition}

For simplicity of presentation, from now on we fix $r=3$ (the
reader can consult \cite{natorres} for a presentation with an
arbitrary $r$).

\begin{lemma}[cf. \cite{Ferr3}]
\label{lem3} Suppose that $f$ is defined on
$[a,\rho^6(b)]_\mathbb{T}$ and is continuous. Then, under
hypothesis (H), $\int_a^{\rho^5(b)}f(t)\eta^{\sigma^3}(t)\Delta
t=0$ for every admissible variation $\eta$ if and only if $f(t)=0$
for all $t\in[a,\rho^6(b)]_\mathbb{T}$.
\end{lemma}
\begin{proof}
If $f(t)=0$, then the result is obvious.

Now suppose without loss of generality that there exists
$t_0\in[a,\rho^6(b)]_\mathbb{T}$ such that $f(t_0)>0$. First we
consider the case in which $t_0$ is right-dense, hence left-dense
or $t_0=a$ (see Remark~\ref{remark0}). If $t_0=a$, then by the
continuity of $f$ at $t_0$ there exists a $\delta>0$ such that for
all $t\in[t_0,t_0+\delta)_\mathbb{T}$ we have $f(t)>0$. Let us
define $\eta$ by
\[ \eta(t) = \left\{ \begin{array}{ll}
(t-t_0)^8(t-t_0-\delta)^8 & \mbox{if $t \in [t_0,t_0+\delta)_\mathbb{T}$};\\
0 & \mbox{otherwise}.\end{array} \right. \] Clearly $\eta$ is a
$C^6$ function and satisfy the requirements of an admissible
variation. But with this definition for $\eta$ we get the
contradiction
$$\int_a^{\rho^5(b)}f(t)\eta^{\sigma^3}(t)\Delta t
=\int_{t_0}^{t_0+\delta}f(t)\eta^{\sigma^3}(t)\Delta t>0.$$ Now,
consider the case where $t_0\neq a$. Again, the continuity of $f$
ensures the existence of a $\delta>0$ such that for all
$t\in(t_0-\delta,t_0+\delta)_\mathbb{T}$ we have $f(t)>0$.
Defining $\eta$ by
\[ \eta(t) = \left\{ \begin{array}{ll}
(t-t_0+\delta)^8(t-t_0-\delta)^8 & \mbox{if $t \in (t_0-\delta,t_0+\delta)_\mathbb{T}$};\\
0 & \mbox{otherwise},\end{array} \right. \] and noting that it
satisfy the properties of an admissible variation, we obtain
$$\int_a^{\rho^5(b)}f(t)\eta^{\sigma^3}(t)\Delta t
=\int_{t_0-\delta}^{t_0+\delta}f(t)\eta^{\sigma^3}(t)\Delta t>0,$$
which is again a contradiction.

Assume now that $t_0$ is right-scattered. In view of
Remark~\ref{remark0}, all the points $t$ such that $t\geq t_0$
must be isolated. So, define $\eta$ such that
$\eta^{\sigma^3}(t_0)=1$ and is zero elsewhere. It is easy to see
that $\eta$ satisfies all the requirements of an admissible
variation. Further, using formula (\ref{sigma})
$$\int_a^{\rho^5(b)}f(t)\eta^{\sigma^3}(t)\Delta t
=\int_{t_0}^{\sigma(t_0)}f(t)\eta^{\sigma^3}(t)\Delta t
=\mu(t_0)f(t_0)\eta^{\sigma^3}(t_0)>0,$$ which is a contradiction.
\end{proof}

\begin{theorem}[cf. \cite{Ferr3}] \label{thm0000}
Let the Lagrangian $L(t,u_0,u_1,u_2,u_3)$ have standard partial
derivatives with respect to $u_0,u_1,u_2,u_3$ and partial
$\Delta$-derivative with respect to $t$ of order 4. On a time
scale $\mathbb{T}$ satisfying (H), if $y_\ast$ is a weak local
minimum for the problem of minimizing
$$
\int_{a}^{\rho^{2}(b)}
L\left(t,y^{\sigma^3}(t),y^{\sigma^{2}\Delta}(t),
y^{\sigma\Delta^2}(t),y^{\Delta^{3}}(t)\right)\Delta t
$$ subject to
\begin{align}
y(a)=y_a,& \ y\left(\rho^{2}(b)\right)=y_b,\nonumber\\
y^{\Delta}(a)=y_a^{1},&\
y^{\Delta}\left(\rho^{2}(b)\right)=y_b^{1},\nonumber\\
y^{\Delta^{2}}(a)=y_a^{2},&\
y^{\Delta^{2}}\left(\rho^{2}(b)\right)=y_b^{2},\nonumber
\end{align}
then $y_\ast$ satisfies the Euler--Lagrange equation
$$L_{u_0}(\cdot)-L^\Delta_{u_1}(\cdot)+\frac{1}{a_1}L^{\Delta^2}_{u_2}(\cdot)
-\frac{1}{a_1^3}L^{\Delta^3}_{u_3}(\cdot)=0,\quad
t\in[a,\rho^6(b)]_\mathbb{T},$$ where
$(\cdot)=(t,y_\ast^{\sigma^3}(t),y_\ast^{\sigma^2\Delta}(t),
y_\ast^{\sigma\Delta^2}(t),y_\ast^{\Delta^3}(t))$.
\end{theorem}

\begin{proof}
Suppose that $y_\ast$ is a weak local minimum of $\mathcal{L}$.
Let $\eta\in C^6$ be an admissible variation, \textrm{i.e.},
$\eta$ is an arbitrary function such that $\eta$, $\eta^\Delta$
and $\eta^{\Delta^2}$ vanish at $t=a$ and $t=\rho^2(b)$. Define
function $\Phi:\mathbb{R}\rightarrow\mathbb{R}$ by
$\Phi(\varepsilon)=\mathcal{L}(y_\ast+\varepsilon\eta)$. This
function has a minimum at $\varepsilon=0$, so we must have
\begin{equation*}
\Phi'(0)=0.
\end{equation*}
Differentiating $\Phi$ with respect to $\varepsilon$ and setting
$\varepsilon=0$, we obtain
\begin{equation}\label{rui000}
0=\int_a^{\rho^2(b)}\Big\{L_{u_0}(\cdot)\eta^{\sigma^3}(t)
+L_{u_1}(\cdot)\eta^{\sigma^2\Delta}(t)
+L_{u_2}(\cdot)\eta^{\sigma\Delta^2}(t)
+L_{u_3}(\cdot)\eta^{\Delta^3}(t)\Big\}\Delta t.
\end{equation}
Since we will $\Delta$-differentiate $L_{u_i}$, $i=1,2,3$, we
rewrite (\ref{rui000}) in the following form:
\begin{multline}
\label{rui111}
0=\int_a^{\rho^3(b)}\Big\{L_{u_0}(\cdot)\eta^{\sigma^3}(t)
+L_{u_1}(\cdot)\eta^{\sigma^2\Delta}(t)
+L_{u_2}(\cdot)\eta^{\sigma\Delta^2}(t)
+L_{u_3}(\cdot)\eta^{\Delta^3}(t)\Big\}\Delta t\\
+\mu(\rho^3(b))\left\{L_{u_0}\eta^{\sigma^3}
+L_{u_1}\eta^{\sigma^2\Delta}+L_{u_2}\eta^{\sigma\Delta^2}
+L_{u_3}\eta^{\Delta^3}\right\}(\rho^3(b)).
\end{multline}
Integrating \eqref{rui111} by parts gives
\begin{equation}
\label{rui222}
\begin{aligned}
0&=\int_a^{\rho^3(b)}\Big\{L_{u_0}(\cdot)\eta^{\sigma^3}(t)
-L^\Delta_{u_1}(\cdot)\eta^{\sigma^3}(t)
-L^\Delta_{u_2}(\cdot)\eta^{\sigma\Delta\sigma}(t)
-L^\Delta_{u_3}(\cdot)\eta^{\Delta^2\sigma}(t)\Big\}\Delta t\\
\quad
&+\left[L_{u_1}(\cdot)\eta^{\sigma^2}(t)\right]_{t=a}^{t=\rho^3(b)}
+\left[L_{u_2}(\cdot)\eta^{\sigma\Delta}(t)\right]_{t=a}^{t=\rho^3(b)}
+\left[L_{u_3}(\cdot)\eta^{\Delta^2}(t)\right]_{t=a}^{t=\rho^3(b)}\\
\quad &+\mu(\rho^3(b))\left\{L_{u_0}\eta^{\sigma^3}
+L_{u_1}\eta^{\sigma^2\Delta}+L_{u_2}\eta^{\sigma\Delta^2}
+L_{u_3}\eta^{\Delta^3}\right\}(\rho^3(b)).
\end{aligned}
\end{equation}
Now we show how to simplify \eqref{rui222}. We start by evaluating
$\eta^{\sigma^2}(a)$:
\begin{align}
\eta^{\sigma^2}(a)&=\eta^\sigma(a)+\mu(a)\eta^{\sigma\Delta}(a)\nonumber\\
&=\eta(a)+\mu(a)\eta^\Delta(a)+\mu(a)a_1\eta^{\Delta\sigma}(a)\label{rui333}\\
&=\mu(a)a_1\left(\eta^\Delta(a)+\mu(a)\eta^{\Delta^2}(a)\right)\nonumber\\
&=0 \nonumber,
\end{align}
where the last term of (\ref{rui333}) follows from (\ref{rui-1}).
Now, we calculate $\eta^{\sigma\Delta}(a)$. By (\ref{rui-1}) we
have $\eta^{\sigma\Delta}(a)=a_1\eta^{\Delta\sigma}(a)$ and
applying (\ref{transfor}) we obtain
$$a_1\eta^{\Delta\sigma}(a)=a_1\left(\eta^\Delta(a)+\mu(a)\eta^{\Delta^2}(a)\right)=0.$$
Now we turn to analyze what happens at $t=\rho^3(b)$. It is easy
to see that if $b$ is left-dense, then the last terms of
(\ref{rui222}) vanish. So suppose that $b$ is left-scattered.
Since $\sigma$ is $\Delta$-differentiable, by Lemma~\ref{lema0} we
cannot have points that are simultaneously left-dense and
right-scattered. Hence, $\rho(b)$, $\rho^{2}(b)$ and $\rho^3(b)$
are right-scattered points. Now, by hypothesis,
$\eta^\Delta(\rho^2(b))=0$. Hence we obtain using \eqref{derdiscr}
that
$$\frac{\eta(\rho(b))-\eta(\rho^2(b))}{\mu(\rho^2(b))}=0.$$
But $\eta(\rho^2(b))=0$, hence $\eta(\rho(b))=0$. Analogously, we
have
$$\eta^{\Delta^2}(\rho^2(b))=0\Leftrightarrow\frac{\eta^\Delta(\rho(b))
-\eta^\Delta(\rho^2(b))}{\mu(\rho^2(b))}=0,$$ from what follows
that $\eta^\Delta(\rho(b))=0$. This last equality implies
$\eta(b)=0$. Applying previous expressions to the last terms of
(\ref{rui222}), we obtain:
\begin{equation}
\eta^{\sigma^2}(\rho^3(b))=\eta(\rho(b))=0,\nonumber
\end{equation}
$$\eta^{\sigma\Delta}(\rho^3(b))=\frac{\eta^{\sigma^2}(\rho^3(b))
-\eta^\sigma(\rho^3(b))}{\mu(\rho^3(b))}=0,$$
$$\eta^{\sigma^3}(\rho^3(b))=\eta(b)=0,$$
$$\eta^{\sigma^2\Delta}(\rho^3(b))=\frac{\eta^{\sigma^3}(\rho^3(b))
-\eta^{\sigma^2}(\rho^3(b))}{\mu(\rho^3(b))}=0,$$
\begin{equation*}
\begin{split}
\eta^{\sigma\Delta^2}(\rho^3(b))&=\frac{\eta^{\sigma\Delta}(\rho^2(b))
-\eta^{\sigma\Delta}(\rho^3(b))}{\mu(\rho^3(b))}\\
&=\frac{\frac{\eta^\sigma(\rho(b))
-\eta^\sigma(\rho^2(b))}{\mu(\rho^2(b))}-\frac{\eta^\sigma(\rho^2(b))
-\eta^\sigma(\rho^3(b))}{{\mu(\rho^3(b))}}}{\mu(\rho^3(b))}\\
&=0.
\end{split}
\end{equation*}
In view of our previous calculations,
\begin{multline*}
\left[L_{u_1}(\cdot)\eta^{\sigma^2}(t)\right]_{t=a}^{t=\rho^3(b)}
+\left[L_{u_2}(\cdot)\eta^{\sigma\Delta}(t)\right]_{t=a}^{t=\rho^3(b)}
+\left[L_{u_3}(\cdot)\eta^{\Delta^2}(t)\right]_{t=a}^{t=\rho^3(b)}\\
+\mu(\rho^3(b))\left\{L_{u_0}\eta^{\sigma^3}
+L_{u_1}\eta^{\sigma^2\Delta}+L_{u_2}\eta^{\sigma\Delta^2}
+L_{u_3}\eta^{\Delta^3}\right\}(\rho^3(b))
\end{multline*}
is reduced to\footnote{There is some abuse of notation:
$L_{u_3}(\rho^3(b))$ denotes $\left.L_{u_3}(\cdot)\right|_{t =
\rho^3(b)}$, that is, we substitute $t$ in
$(\cdot)=(t,y_\ast^{\sigma^3}(t),y_\ast^{\sigma^2\Delta}(t),
y_\ast^{\sigma\Delta^2}(t),y_\ast^{\Delta^3}(t))$ by $\rho^3(b)$.}
\begin{equation}
\label{rui444} L_{u_3}(\rho^3(b))\eta^{\Delta^2}(\rho^3(b))
+\mu(\rho^3(b))L_{u_3}(\rho^3(b))\eta^{\Delta^3}(\rho^3(b)).
\end{equation}
Now note that
$$\eta^{\Delta^2\sigma}(\rho^3(b))=\eta^{\Delta^2}(\rho^3(b))
+\mu(\rho^3(b))\eta^{\Delta^3}(\rho^3(b))$$ and by hypothesis
$\eta^{\Delta^2\sigma}(\rho^3(b))=\eta^{\Delta^2}(\rho^2(b))=0$.
Therefore,
$$\mu(\rho^3(b))\eta^{\Delta^3}(\rho^3(b))=-\eta^{\Delta^2}(\rho^3(b)),$$
from which follows that (\ref{rui444}) must be zero. We have just
simplified (\ref{rui222}) to
\begin{equation}
\label{rui555} 0 =
\int_a^{\rho^3(b)}\Big\{L_{u_0}(\cdot)\eta^{\sigma^3}(t)
-L^\Delta_{u_1}(\cdot)\eta^{\sigma^3}(t)
-L^\Delta_{u_2}(\cdot)\eta^{\sigma
\Delta\sigma}(t)-L^\Delta_{u_3}(\cdot)\eta^{\Delta^2\sigma}(t)\Big\}\Delta
t.
\end{equation}
In order to apply again the integration by parts formula, we must
first make some transformations in $\eta^{\sigma\Delta\sigma}$ and
$\eta^{\Delta^2\sigma}$. By (\ref{rui-1}) we have
\begin{equation}
\label{rui12}
\eta^{\sigma\Delta\sigma}(t)=\frac{1}{a_1}\eta^{\sigma^2\Delta}(t)
\end{equation}
and
\begin{equation}
\label{rui121}
\eta^{\Delta^2\sigma}(t)=\frac{1}{a_1^2}\eta^{\sigma\Delta^2}(t).
\end{equation}
Hence, (\ref{rui555}) becomes
\begin{equation}
\label{rui9} 0 =
\int_a^{\rho^3(b)}\Big\{L_{u_0}(\cdot)\eta^{\sigma^3}(t)
-L^\Delta_{u_1}(\cdot)\eta^{\sigma^3}(t)\\
-\frac{1}{a_1}L^\Delta_{u_2}(\cdot)\eta^{\sigma^2\Delta}(t)
-\frac{1}{a_1^2}L^\Delta_{u_3}(\cdot)\eta^{\sigma\Delta^2}(t)\Big\}\Delta
t.
\end{equation}
By the same reasoning as before, (\ref{rui9}) is equivalent to
\begin{multline*}
0=\int_a^{\rho^4(b)}\Big\{L_{u_0}(\cdot)\eta^{\sigma^3}(t)
-L^\Delta_{u_1}(\cdot)\eta^{\sigma^3}(t)
-\frac{1}{a_1}L^\Delta_{u_2}(\cdot)\eta^{\sigma^2\Delta}(t)
-\frac{1}{a_1^2}L^\Delta_{u_3}(\cdot)\eta^{\sigma\Delta^2}(t)\Big\}\Delta t\\
\quad +\mu(\rho^4(b))\left\{L_{u_0}\eta^{\sigma^3}
-L^\Delta_{u_1}\eta^{\sigma^3}-\frac{1}{a_1}L^\Delta_{u_2}\eta^{\sigma^2\Delta}
-\frac{1}{a_1^2}L^\Delta_{u_3}\eta^{\sigma\Delta^2}\right\}(\rho^4(b))
\end{multline*}
and integrating by parts we obtain
\begin{equation}
\label{del:rui13}
\begin{aligned}
0&=\int_a^{\rho^4(b)}\Big\{L_{u_0}(\cdot)\eta^{\sigma^3}(t)
-L^\Delta_{u_1}(\cdot)\eta^{\sigma^3}(t)
+\frac{1}{a_1}L^{\Delta^2}_{u_2}(\cdot)\eta^{\sigma^3}(t)
+\frac{1}{a_1^2}L^{\Delta^2}_{u_3}(\cdot)\eta^{\sigma\Delta\sigma}(t)\Big\} \Delta t\\
&-\left[\frac{1}{a_1}L_{u_2}^\Delta(\cdot)\eta^{\sigma^2}(t)\right]_{t=a}^{t=\rho^4(b)}
-\left[\frac{1}{a_1^2}L_{u_3}^\Delta(\cdot)\eta^{\sigma\Delta}(t)\right]_{t=a}^{t
=\rho^4(b)}\\
\quad &+\mu(\rho^4(b))\left\{L_{u_0}\eta^{\sigma^3}
-L^\Delta_{u_1}\eta^{\sigma^3}-\frac{1}{a_1}L^\Delta_{u_2}\eta^{\sigma^2\Delta}
-\frac{1}{a_1^2}L^\Delta_{u_3}\eta^{\sigma\Delta^2}\right\}(\rho^4(b)).
\end{aligned}
\end{equation}
Using analogous arguments to those above, we simplify
(\ref{del:rui13}) to
\begin{equation*}
\int_a^{\rho^4(b)}\Big\{L_{u_0}(\cdot)\eta^{\sigma^3}(t)
-L^\Delta_{u_1}(\cdot)\eta^{\sigma^3}(t)
+\frac{1}{a_1}L^\Delta_{u_2}(\cdot)\eta^{\sigma^2\Delta}(t)
+\frac{1}{a_1^3}L^{\Delta^2}_{u_3}(\cdot)\eta^{\sigma^2\Delta}(t)\Big\}\Delta
t = 0.
\end{equation*}
Calculations as done before lead us to the final expression
\begin{equation*}
\int_a^{\rho^5(b)}\Big\{L_{u_0}(\cdot)\eta^{\sigma^3}(t)
-L^\Delta_{u_1}(\cdot)\eta^{\sigma^3}(t)
+\frac{1}{a_1}L^{\Delta^2}_{u_2}(\cdot)\eta^{\sigma^3}(t)
-\frac{1}{a_1^3}L^{\Delta^3}_{u_3}(\cdot)\eta^{\sigma^3}(t)\Big\}\Delta
t = 0,
\end{equation*}
which is equivalent to
\begin{equation}
\label{almostDone}
\int_a^{\rho^5(b)}\left\{L_{u_0}(\cdot)-L^\Delta_{u_1}(\cdot)
+\frac{1}{a_1}L^{\Delta^2}_{u_2}(\cdot)
-\frac{1}{a_1^3}L^{\Delta^3}_{u_3}(\cdot)\right\}\eta^{\sigma^3}(t)\Delta
t = 0.
\end{equation}
Applying Lemma~\ref{lem3} to (\ref{almostDone}), we obtain the
Euler--Lagrange equation
$$L_{u_0}(\cdot)-L^\Delta_{u_1}(\cdot)+\frac{1}{a_1}L^{\Delta^2}_{u_2}(\cdot)
-\frac{1}{a_1^3}L^{\Delta^3}_{u_3}(\cdot)=0,\quad
t\in[a,\rho^6(b)]_\mathbb{T}.$$ The proof is complete.
\end{proof}

Following exactly the same steps of the proofs of Lemma~\ref{lem3}
and Theorem~\ref{thm0000} for an arbitrary $r\in\mathbb{N}$, one
easily obtain the Euler--Lagrange equation for problem
(\ref{problema }).

\begin{theorem}[cf. \cite{Ferr3}]\label{E-LHigher}(Necessary optimality condition for problems of the calculus of
variations with higher-order $\Delta$-derivatives) On a time scale
$\mathbb{T}$ satisfying hypothesis (H), if $y_\ast$ is a weak
local minimum for problem (\ref{problema }), then $y_\ast$
satisfies the Euler-Lagrange equation
\begin{equation}
\label{eqrder}
\sum_{i=0}^{r}(-1)^i\left(\frac{1}{a_1}\right)^{\frac{(i
-1)i}{2}}L^{\Delta^i}_{u_i}\left(t,y_\ast^{\sigma^r}(t),y_\ast^{\sigma^{r-1}\Delta}(t),
\ldots,y_\ast^{\sigma\Delta^{r-1}}(t),y_\ast^{\Delta^r}(t)\right)
= 0 ,
\end{equation}
$t\in[a,\rho^{2r}(b)]_\mathbb{T}$.
\end{theorem}

\begin{remark}
The factor $\left(\frac{1}{a_1}\right)^{\frac{(i-1)i}{2}}$ in
(\ref{eqrder}) comes from the fact that, after each time we apply
the integration by parts formula, we commute successively $\sigma$
with $\Delta$ using (\ref{rui-1}) [see formulas (\ref{rui12}) and
(\ref{rui121})], doing this $\sum_{j=1}^{i-1}j=\frac{(i-1)i}{2}$
times for each of the parcels within the integral.
\end{remark}

\begin{example}
Let us consider the time scale
$\mathbb{T}=\{q^k:k\in\mathbb{N}_0\}$, $a,b\in\mathbb{T}$ with
$a<b$, and the problem
\begin{equation}\label{probex}
\begin{gathered}
\mathcal{L}[y(\cdot)]=\int_{a}^{\rho(b)}
(y^{\Delta^2}(t))^2\Delta t\longrightarrow\min, \\
y(a)=0,  \ \  y\left(\rho(b)\right)=1 ,\\
y^{\Delta}(a)=0,\ \ y^{\Delta}\left(\rho(b)\right)=0.
\end{gathered}
\end{equation}
We point out that for this time scale, $\sigma(t)=qt$,
$\mu(t)=(q-1)t$, and
$$f^\Delta(t)=\frac{f(qt)-f(t)}{(q-1)t},\quad\int_a^bf(t)\Delta t=\sum_{t\in[a,b)_\mathbb{T}}(q-1)tf(t).$$
By Theorem \ref{E-LHigher}, the Euler--Lagrange equation for
problem given in \eqref{probex} is
\begin{equation*}
\frac{1}{q}y^{\Delta^4}(t)=0,\quad t\in[a,\rho^4(b)]_\mathbb{T}.
\end{equation*}
It follows that
$$y^{\Delta^2}(t)=ct+d,$$
for some constants $c,d\in\mathbb{R}$. From this we get, using
$y^\Delta(a)=0$ and $y^\Delta(\rho(b))=0$,
$$y^\Delta(t)=c\int_a^ts\Delta s-\frac{c\int_a^{\rho(b)}s\Delta s}{\rho(b)-a}(t-a).$$
Finally, using the boundary conditions $y(a)=0$ and
$y(\rho(b))=1$, and defining
\begin{align}
A&=\frac{\int_a^{\rho(b)}s\Delta s}{\rho(b)-a},\nonumber\\
B&=\frac{1}{\int_a^{\rho(b)}\left(\int_a^s\tau\Delta\tau\right)\Delta
s-A\int_a^{\rho(b)}(s-a)\Delta s},\nonumber
\end{align}
we get
\begin{equation}\label{ex4}
y(t)=B\left[\int_a^{t}\left(\int_a^s\tau\Delta\tau\right)\Delta
s-A\int_a^{t}(s-a)\Delta s\right],\quad t\in[a,b]_\mathbb{T},
\end{equation}
provided the denominator of $B$ is nonzero. We end this example
showing that this is indeed the case. We start by rewriting the
denominator of $B$, which we denote by $D$, as
\begin{multline*}
D=\int_a^{\rho(b)}\left(\int_a^{s}(\tau-a)\Delta\tau\right)\Delta s\\
+a\int_a^{\rho(b)}(s-a)\Delta s-\frac{\int_a^{\rho(b)}(s-a)\Delta
s+a(\rho(b)-a)}{\rho(b)-a}\int_a^{\rho(b)}(s-a)\Delta s.
\end{multline*}
Now, note that
\begin{align}
\int_a^{\rho(b)}\left(\int_a^{s}(\tau-a)\Delta\tau\right)\Delta
s&=h_3(\rho(b),a)\nonumber\\
\int_a^{\rho(b)}(s-a)\Delta s&=h_2(\rho(b),a),\nonumber
\end{align}
where $h_i$ are given by Definition \ref{polinomios}. It is known
(cf. \cite[Example 1.104]{livro}) that
$$h_k(t,s)=\prod_{\nu=0}^{k-1}\frac{t-q^\nu s}{\sum_{\mu=0}^\nu q^\mu},\quad s,t\in\mathbb{T}.$$
Hence,
$$D=\prod_{\nu=0}^{2}\frac{\rho(b)-q^\nu a}{\sum_{\mu=0}^\nu q^\mu}+a\prod_{\nu=0}^{1}\frac{\rho(b)-q^\nu a}{\sum_{\mu=0}^\nu q^\mu}-\frac{\prod_{\nu=0}^{1}\frac{\rho(b)-q^\nu a}{\sum_{\mu=0}^\nu q^\mu}+a(\rho(b)-a)}{\rho(b)-a}\prod_{\nu=0}^{1}\frac{\rho(b)-q^\nu a}{\sum_{\mu=0}^\nu q^\mu}.$$
From this it is not difficult to achieve
$$D=\frac{q(-\rho(b)+a)(-\rho(b)+qa)(a-\rho(b)q)}{(1+q)^2(1+q+q^2)}.$$
Therefore, we conclude that $D\neq 0$. We also get another
representation for \eqref{ex4}, namely,
$$y(t)=\frac{(-t+a)(-t+qa)(a-tq)}{(-\rho(b)+a)(-\rho(b)+qa)(a-\rho(b)q)},\quad t\in[a,b]_\mathbb{T}.$$
\end{example}

\section{Isoperimetric problems}\label{iso}

In this section we make a study of isoperimetric problems on a
general time scale (see \cite{Brunt} and \cite{BHAR} for continuos
and discrete versions, respectively), proving the corresponding
necessary optimality condition. These problems add some
constraints (in integral form) to the basic problem. In the end,
we show that certain eigenvalue problems can be recast as an
isoperimetric problem.

We start by giving a proof of a technical lemma.

\begin{lemma}[cf. \cite{Ferr7}]
\label{lemtecn} Suppose that a continuous function
$f:\mathbb{T}\rightarrow\mathbb{R}$ is such that $f^\sigma(t)=0$
for all $t\in\mathbb{T}^\kappa$. Then, $f(t)=0$ for all
$t\in\mathbb{T}$ except possibly at $t=a$ if $a$ is
right-scattered.
\end{lemma}

\begin{proof}
First note that, since $f^\sigma(t)=0$, then $f^\sigma(t)$ is
$\Delta$-differentiable, hence continuous for all
$t\in\mathbb{T}^\kappa$. Now, if $t$ is right-dense, the result is
obvious. Suppose that $t$ is right-scattered. We will analyze two
cases: (i) if $t$ is left-scattered, then $t\neq a$ and by
hypothesis $0=f^\sigma(\rho(t))=f(t)$; (ii) if $t$ is left-dense,
then, by the continuity of $f^\sigma$ and $f$ at $t$, we can write
\begin{align}
\forall\varepsilon>0 \, \exists\delta_1>0 : \forall
s_1\in(t-\delta_1,t]_\mathbb{T},&\label{101} \mbox{ we have }
|f^\sigma(s_1)-f^\sigma(t)|<\varepsilon \, , \\
\forall \varepsilon>0 \, \exists\delta_2>0:\forall
s_2\in(t-\delta_2,t]_\mathbb{T},&\label{102} \mbox{ we have }
|f(s_2)-f(t)|<\varepsilon \, ,
\end{align}
respectively. Let $\delta=\min\{\delta_1,\delta_2\}$ and take
$s_1\in(t-\delta,t)_\mathbb{T}$. As
$\sigma(s_1)\in(t-\delta,t)_\mathbb{T}$, take $s_2=\sigma(s_1)$.
By (\ref{101}) and (\ref{102}), we have:
$$
|-f^\sigma(t)+f(t)|=|f^\sigma(s_1)-f^\sigma(t)
+f(t)-f(s_2)|\leq|f^\sigma(s_1)-f^\sigma(t)|+|f(s_2)-f(t)|<2\varepsilon.
$$
Since $\varepsilon$ is arbitrary, $|-f^\sigma(t)+f(t)|=0$, which
is equivalent to $f(t)=f^\sigma(t)$.
\end{proof}

We now define the isoperimetric problem on time scales. Let $J:$
C$_{\textrm{rd}}^1\rightarrow\mathbb{R}$ be a functional of the
form

\begin{equation}
\label{P0} J[y(\cdot)]=\int_a^b L(t,y^\sigma(t),y^\Delta(t))\Delta
t,
\end{equation}
where
$L(t,u,v):[a,b]^\kappa_\mathbb{T}\times\mathbb{R}\times\mathbb{R}\rightarrow\mathbb{R}$
satisfies the assumption of Lemma \ref{lema1}. The
\emph{isoperimetric problem}\index{Calculus of
Variations!isoperimetric problem} consists of finding functions
$y\in\textrm{C}_{\textrm{rd}}^1$ satisfying given boundary
conditions
\begin{equation}
\label{bouncond} y(a)=y_a,\ y(b)=y_b,
\end{equation}
and a constraint of the form
\begin{equation}
\label{isocons} I[y(\cdot)]=\int_a^b
g(t,y^\sigma(t),y^\Delta(t))\Delta t=l,
\end{equation}
where
$g(t,u,v):[a,b]^\kappa\times\mathbb{R}\times\mathbb{R}\rightarrow\mathbb{R}$
satisfies the assumption of Lemma \ref{lema1}, and $l$ is a
specified real number, that takes (\ref{P0}) to a minimum.

\begin{definition}
We say that a function $y\in$ C$_{\textrm{rd}}^1$ is
\emph{admissible} for the isoperimetric problem if it satisfies
(\ref{bouncond}) and (\ref{isocons}).
\end{definition}

\begin{definition}
An admissible function $y_\ast$ is said to be an \emph{extremal}
for $I$ if it satisfies the following equation:
$$g_v(t,y^\sigma_\ast(t),y^\Delta_\ast(t))-\int_a^t
g_u(\tau,y^\sigma_\ast(\tau),y^\Delta_\ast(\tau))\Delta\tau=c,$$
for all $t\in[a,b]^\kappa_\mathbb{T}$ and some constant $c$.
\end{definition}

\begin{theorem}[cf. \cite{Ferr7}]
\label{T1} Suppose that $J$ has a local minimum at $y_\ast\in$
$C_{\textrm{rd}}^1$ subject to the boundary conditions
(\ref{bouncond}) and the isoperimetric constraint (\ref{isocons}),
and that $y_\ast$ is not an extremal for the functional $I$. Then,
there exists a Lagrange multiplier constant $\lambda$ such that
$y_\ast$ satisfies the following equation:
\begin{equation}
\label{E-L} F_v^\Delta(t,y^\sigma_\ast(t),y^\Delta_\ast(t))
-F_u(t,y^\sigma_\ast(t),y^\Delta_\ast(t))=0,\quad\mbox{for all} \
t\in[a,b]^{\kappa^2}_\mathbb{T},
\end{equation}
where $F=L-\lambda g$.
\end{theorem}

\begin{proof}
Let $y_\ast$ be a local minimum for $J$ and consider neighboring
functions of the form
\begin{equation}
\label{admfunct}
\hat{y}=y_\ast+\varepsilon_1\eta_1+\varepsilon_2\eta_2,
\end{equation}
where for each $i\in\{1,2\}$, $\varepsilon_i$ is a sufficiently
small parameter ($\varepsilon_1$ and $\varepsilon_2$ must be such
that $\|\hat{y}-y^\ast\|<\delta$, for some $\delta>0$),
$\eta_i(x)\in$ C$_{\textrm{rd}}^1$ and $\eta_i(a)=\eta_i(b)=0$.
Here, $\eta_1$ is an arbitrary fixed function and $\eta_2$ is a
fixed function that we will choose later.

First we show that (\ref{admfunct}) has a subset of admissible
functions for the isoperimetric problem. Consider the quantity
$$I[\hat{y}(\cdot)]=\int_a^b
g(t,y_\ast^\sigma(t)+\varepsilon_1\eta_1^\sigma(t)
+\varepsilon_2\eta_2^\sigma(t),y_\ast^\Delta(t)
+\varepsilon_1\eta_1^\Delta(t)+\varepsilon_2\eta_2^\Delta(t))\Delta
t.$$ Then we can regard $I[\hat{y}(\cdot)]$ as a function of
$\varepsilon_1$ and $\varepsilon_2$, say
$I[\hat{y}(\cdot)]=\hat{Q}(\varepsilon_1,\varepsilon_2)$. Since
$y_\ast$ is a local minimum for $J$ subject to the boundary
conditions and the isoperimetric constraint, putting
$Q(\varepsilon_1,\varepsilon_2)=\hat{Q}(\varepsilon_1,\varepsilon_2)-l$
we have that
\begin{equation}
\label{implicit1} Q(0,0)=0.
\end{equation}
By the conditions imposed on $g$, we have
\begin{align}
\frac{\partial Q}{\partial\varepsilon_2}(0,0)&=\int_a^b\left[
g_u(t,y^\sigma_\ast(t),y^\Delta_\ast(t))\eta_2^\sigma(t)
+g_v(t,y^\sigma_\ast(t),y^\Delta_\ast(t))\eta_2^\Delta(t)\right]\Delta
t\nonumber\\ &=\int_a^b
\left[g_v(t,y^\sigma_\ast(t),y^\Delta_\ast(t))-\int_a^t
g_u(\tau,y^\sigma_\ast(\tau),y^\Delta_\ast(\tau))\Delta\tau\right]\eta_2^\Delta(t)\Delta
t\label{ruii0},
\end{align}
where \eqref{ruii0} follows from \eqref{partes1} and the fact that
$\eta_2(a)=\eta_2(b)=0$. Now, consider the function
$$E(t)=g_v(t,y^\sigma_\ast(t),y^\Delta_\ast(t))-\int_a^t
g_u(\tau,y^\sigma_\ast(\tau),y^\Delta_\ast(\tau))\Delta\tau,\quad
t\in[a,b]_\mathbb{T}^\kappa.$$ Then, we can apply
Lemma~\ref{lem:DR} to show that there exists a function
$\eta_2\in$ C$_{\textrm{rd}}^1$ such that
\begin{equation}
\int_a^b \left[g_v(t,y^\sigma_\ast(t),y^\Delta_\ast(t))-\int_a^t
g_u(\tau,y^\sigma_\ast(\tau),y^\Delta_\ast(\tau))\Delta\tau\right]\eta_2^\Delta(t)\Delta
t\neq 0\nonumber,
\end{equation}
provided $y_\ast$ is not an extremal for $I$, which is indeed the
case. We have just proved that
\begin{equation}
\label{implicit2} \frac{\partial
Q}{\partial\varepsilon_2}(0,0)\neq 0.
\end{equation}
Using (\ref{implicit1}) and (\ref{implicit2}), the implicit
function theorem asserts that there exist neighborhoods $N_{1}$
and $N_{2}$ of $0$, $N_{1}\subseteq\{\varepsilon_1\ \mbox{from}\
(\ref{admfunct})\}\cap\mathbb{R}$  and
$N_{2}\subseteq\{\varepsilon_2\ \mbox{from}\
(\ref{admfunct})\}\cap\mathbb{R}$, and a function
$\varepsilon_2:N_{1}\rightarrow\mathbb{R}$ such that for all
$\varepsilon_1\in N_{1}$ we have

$$Q(\varepsilon_1,\varepsilon_2(\varepsilon_1))=0,$$
which is equivalent to
$\hat{Q}(\varepsilon_1,\varepsilon_2(\varepsilon_1))=l$. Now we
derive the necessary condition (\ref{E-L}). Consider the quantity
$J(\hat{y})=K(\varepsilon_1,\varepsilon_2)$. By hypothesis, $K$ is
minimum at $(0,0)$ subject to the constraint $Q(0,0)=0$, and we
have proved that $\nabla Q(0,0)\neq \textbf{0}$. We can appeal to
the Lagrange multiplier rule (see, \textrm{e.g.},
\cite[Theorem~4.1.1]{Brunt}) to assert that there exists a number
$\lambda$ such that
\begin{equation}
\label{rui1} \nabla(K(0,0)-\lambda Q(0,0))=\textbf{0}.
\end{equation}
Having in mind that $\eta_1(a)=\eta_1(b)=0$, we can write:
\begin{align}
\label{rui2} \frac{\partial
K}{\partial\varepsilon_1}(0,0)&=\int_a^b\left[
L_u(t,y^\sigma_\ast(t),y^\Delta_\ast(t))\eta_1^\sigma(t)
+L_v(t,y^\sigma_\ast(t),y^\Delta_\ast(t))\eta_1^\Delta(t)\right]\Delta
t\nonumber\\ &=\int_a^b
\left[L_v(t,y^\sigma_\ast(t),y^\Delta_\ast(t))-\int_a^t
L_u(\tau,y^\sigma_\ast(\tau),y^\Delta_\ast(\tau))\Delta\tau\right]\eta_1^\Delta(t)\Delta
t.
\end{align}
Similarly, we have that
\begin{equation}
\label{rui3} \frac{\partial
Q}{\partial\varepsilon_1}(0,0)=\int_a^b
\left[g_v(t,y^\sigma_\ast(t),y^\Delta_\ast(t))-\int_a^t
g_u(\tau,y^\sigma_\ast(\tau),y^\Delta_\ast(\tau))\Delta\tau\right]\eta_1^\Delta(t)\Delta
t.
\end{equation}
Combining (\ref{rui1}), (\ref{rui2}) and (\ref{rui3}), we obtain
\begin{equation}
\int_a^b \left\{L_v(\cdot)-\int_a^t
L_u(\cdot\cdot)\Delta\tau-\lambda\left(g_v(\cdot)-\int_a^t
g_u(\cdot\cdot)\Delta\tau \right)\right\}\eta^\Delta_1(t)\Delta
t=0,\nonumber
\end{equation}
where $(\cdot)=(t,y^\sigma_\ast(t),y^\Delta_\ast(t))$ and
$(\cdot\cdot)=(\tau,y^\sigma_\ast(\tau),y^\Delta_\ast(\tau))$.
Since $\eta_1$ is arbitrary, Lemma~\ref{lem:DR} implies that there
exists a constant $d$ such that
\begin{equation}
L_v(\cdot)-\lambda g_v(\cdot)-\left(\int_a^t
[L_u(\cdot\cdot)-\lambda g_u(\cdot\cdot)]\Delta\tau\right)=d,\quad
t\in[a,b]^\kappa_\mathbb{T},\nonumber
\end{equation}
or
\begin{equation}
\label{quaseE-L} F_v(\cdot)-\int_a^t F_x(\cdot\cdot)\Delta\tau=d,
\end{equation}
with $F=L-\lambda g$. Since the integral and the constant in
(\ref{quaseE-L}) are $\Delta$-differentiable, we obtain the
desired necessary optimality condition (\ref{E-L}).
\end{proof}

\begin{remark}
\label{rem:max} Theorem~\ref{T1} remains valid when $y_\ast$ is
assumed to be a local maximizer of the isoperimetric problem
\eqref{P0}-\eqref{isocons}.
\end{remark}

\begin{remark}
In Theorem~\ref{T1} we assume that $y_\ast$ is not an extremal for
$I$. This corresponds to the normal case. A version of
Theorem~\ref{T1} that involves abnormal extremals can be found in
\cite{AlmeidaTorres09}. A different hypothesis excluding
abnormality is given in \cite{MalinowTorres}.
\end{remark}

\begin{example}
Let $a,-a\in\mathbb{T}$ and suppose that we want to find functions
defined on $[-a,a]_\mathbb{T}$ that take
$$J[y(\cdot)]=\int_{-a}^a y^\sigma(t)\Delta t$$
to its largest value (see Remark~\ref{rem:max}) and that satisfy
the conditions
$$y(-a)=y(a)=0,\ \ I[y(\cdot)]=\int_{-a}^a \sqrt{1+(y^\Delta(t))^2}\Delta
t=l>2a.$$ Note that if $y$ is an extremal for $I$, then $y$ is a
line segment \cite{CD:Bohner:2004}, and therefore $y(t)=0$ for all
$t\in[-a,a]_\mathbb{T}$. This implies that $I(y)=2a>2a$, which is
a contradiction. Hence, $I$ has no extremals satisfying the
boundary conditions and the isoperimetric constraint. Using
Theorem~\ref{T1}, let
$$F=L-\lambda g=y^\sigma-\lambda\sqrt{1+(y^\Delta)^2} \, .$$
Because
$$F_{x}=1,\ \
F_{v}=\lambda\frac{y^\Delta}{\sqrt{1+(y^\Delta)^2}},$$ a necessary
optimality condition is given by the following dynamic equation:
$$\lambda\left(\frac{y^\Delta}{\sqrt{1+(y^\Delta)^2}}\right)^\Delta-1=0,\quad t\in[-a,a]^{\kappa^2}_\mathbb{T}.$$
\end{example}

If we restrict ourselves to times scales $\mathbb{T}$ satisfying
hypothesis (H) of Section \ref{higherorder} it follows that the
same proof as in Theorem~\ref{T1} can be used, \emph{mutatis
mutandis}, to obtain a necessary optimality condition for the
higher-order isoperimetric problem (\textrm{i.e.}, when $L$ and
$g$ contain higher order $\Delta$-derivatives). In this case, the
necessary optimality condition \eqref{E-L} is generalized to
$$
\sum_{i=0}^{r}(-1)^i\left(\frac{1}{a_1}\right)^{\frac{(i-1)i}{2}}
F_{u_i}^{\Delta^i}\left(t,y_\ast^{\sigma^r}(t),y_\ast^{\sigma^{r-1}\Delta}(t),
\ldots,y_\ast^{\sigma\Delta^{r-1}}(t),y_\ast^{\Delta^r}(t)\right)=0
\, ,
$$
where again $F=L-\lambda g$.


\subsection{Sturm--Liouville eigenvalue problems}
\label{sub:sec:ep}

Eigenvalue problems on time scales have been studied in
\cite{SLP}. Consider the following Sturm--Liouville eigenvalue
problem\index{Sturm--Liouville eigenvalue problem}: find
nontrivial solutions to the $\Delta$-dynamic equation
\begin{equation}
\label{rui6} y^{\Delta^2}(t) +q(t)y^\sigma(t)+\lambda
y^\sigma(t)=0,\quad t\in[a,b]_\mathbb{T}^{\kappa^2},
\end{equation}
for the unknown $y:[a,b]_\mathbb{T}\rightarrow\mathbb{R}$ subject
to the boundary conditions
\begin{equation}
\label{rui4} y(a)=y(b)=0.
\end{equation}
Here $q:[a,b]^\kappa_\mathbb{T}\rightarrow\mathbb{R}$ is a
continuous function.

Generically, the only solution to equation (\ref{rui6}) that
satisfies the boundary conditions (\ref{rui4}) is the trivial
solution, $y(t)=0$ for all $t\in[a,b]_\mathbb{T}$. There are,
however, certain values of $\lambda$ that lead to nontrivial
solutions. These are called \emph{eigenvalues}\index{Eigenvalues}
and the corresponding nontrivial solutions are called
\emph{eigenfunctions}\index{Eigenfunctions}. These eigenvalues may
be arranged as $-\infty<\lambda_1<\lambda_2<\ldots$ (cf.
\cite[Theorem~1]{SLP}) and $\lambda_1$ is called the \emph{first
eigenvalue}.

Consider the functional defined by
\begin{equation}
\label{rui7}
J[y(\cdot)]=\int_a^b[(y^\Delta)^2(t)-q(t)(y^\sigma)^2(t)]\Delta t,
\end{equation}
and suppose that $y_\ast\in$ C$_{\textrm{rd}}^2$ (functions that
are twice $\Delta$-differentiable with rd-continuous second
$\Delta$-derivative) is a local minimum for $J$ subject to the
boundary conditions (\ref{rui4}) and the isoperimetric constraint
\begin{equation}
\label{isoc1} I[y(\cdot)]=\int_a^b (y^\sigma)^2(t)\Delta t=1.
\end{equation}
If $y_\ast$ is an extremal for $I$, then we would have
$-2y^\sigma(t)=0,\ t\in[a,b]_\mathbb{T}^{\kappa}$. Noting that
$y(a)=0$, using Lemma~\ref{lemtecn} we would conclude that
$y(t)=0$ for all $t\in[a,b]_\mathbb{T}$. No extremals for $I$ can
therefore satisfy the isoperimetric condition (\ref{isoc1}).
Hence, by Theorem~\ref{T1} there is a constant $\lambda$ such that
$y_\ast$ satisfies
\begin{equation}
\label{rui5} F_{v}^\Delta(t,y^\sigma_\ast(t),y^\Delta_\ast(t))
-F_{u}(t,y^\sigma_\ast(t),y^\Delta_\ast(t))=0,\quad
t\in[a,b]_\mathbb{T}^{\kappa^2},
\end{equation}
with $F=v^2-qu^2-\lambda u^2$. It is easily seen that (\ref{rui5})
is equivalent to (\ref{rui6}). The isoperimetric problem thus
corresponds to the Sturm-Liouville problem augmented by the
normalizing condition (\ref{isoc1}), which simply scales the
eigenfunctions. Here, the Lagrange multiplier plays the role of
the eigenvalue.

The first eigenvalue has the notable property that the
corresponding eigenfunction produces the minimum value for the
functional $J$.

\begin{theorem}
\label{T2} Let $\lambda_1$ be the first eigenvalue for the
Sturm-Liouville problem (\ref{rui6}) with boundary conditions
(\ref{rui4}), and let $y_1$ be the corresponding eigenfunction
normalized to satisfy the isoperimetric constraint (\ref{isoc1}).
Then, among functions in C$_{\textrm{rd}}^2$ that satisfy the
boundary conditions (\ref{rui4}) and the isoperimetric condition
(\ref{isoc1}), the functional $J$ defined by (\ref{rui7}) has a
minimum at $y_1$. Moreover, $J[y_1(\cdot)]=\lambda_1$.
\end{theorem}

\begin{proof}
Suppose that $J$ has a minimum at $y$ satisfying conditions
(\ref{rui4}) and (\ref{isoc1}). Then $y$ satisfies equation
(\ref{rui6}) and multiplying this equation by $y^\sigma$ and
$\Delta$-integrating from $a$ to $\rho(b)$ (note that
$y^{\Delta^2}$ is defined only on $[a,b]_\mathbb{T}^{\kappa^2}$),
we obtain
\begin{equation}
\label{rui8} \int_a^{\rho(b)} y^\sigma(t)y^{\Delta^2}(t)\Delta t
+\int_a^{\rho(b)} q(t)(y^\sigma)^2(t)\Delta
t+\lambda\int_a^{\rho(b)} (y^\sigma)^2(t)\Delta t=0.
\end{equation}
Since $y(a)=y(b)=0$, we get
\begin{align}
\int_a^{\rho(b)} y^\sigma(t)y^{\Delta^2}(t)\Delta
t&=\left[y(t)y^\Delta(t)\right]_{t=a}^{t=\rho(b)}-\int_a^{\rho(b)}
(y^\Delta)^2(t)\Delta t\nonumber\\
&=y(\rho(b))y^\Delta(\rho(b))-\int_a^{\rho(b)}
(y^\Delta)^2(t)\Delta t\label{isop0}.
\end{align}
If $b$ is left-dense it is clear that \eqref{isop0} is equal to
$-\int_a^b (y^\Delta)^2(t)\Delta t$. If $b$ is left-scattered
(hence $\rho(b)$ is right-scattered) we obtain for \eqref{isop0}
\begin{align}
y(\rho(b))y^\Delta(\rho(b))-\int_a^{\rho(b)} (y^\Delta)^2(t)\Delta
t&=y(\rho(b))\frac{y(b)-y(\rho(b))}{\mu(\rho(b))}-\int_a^{\rho(b)}
(y^\Delta)^2(t)\Delta t\nonumber\\
&=-\frac{y(\rho(b))^2}{\mu(\rho(b))}-\int_a^{\rho(b)}
(y^\Delta)^2(t)\Delta t\nonumber\\
&=-\int_a^{b} (y^\Delta)^2(t)\Delta t\nonumber,
\end{align}
where the last equality follows by \eqref{sigma}. With the help of
\eqref{sigma} it is easy to prove that $\int_a^{\rho(b)}
q(t)(y^\sigma)^2(t)\Delta t=\int_a^{b} q(t)(y^\sigma)^2(t)\Delta
t$ and $\int_a^{\rho(b)} (y^\sigma)^2(t)\Delta t=\int_a^{b}
(y^\sigma)^2(t)\Delta t$. By (\ref{isoc1}), (\ref{rui8}) reduces
to
$$\int_a^b [(y^\Delta)^2(t)-q(t)(y^\sigma)^2(t)]\Delta t=\lambda,$$
that is, $J[y(\cdot)]=\lambda$. Due to the isoperimetric
condition, $y$ must be a nontrivial solution to \eqref{rui6} and
therefore $\lambda$ must be an eigenvalue. Since there exists a
least element within the eigenvalues, $\lambda_1$, and a
corresponding eigenfunction $y_1$ normalized to meet the
isoperimetric condition, the minimum value for $J$ is $\lambda_1$
and $J[y_1(\cdot)]=\lambda_1$.
\end{proof}

\section{State of the Art}

The results of this chapter are already published in the following
international journals and chapters in books: \cite{Ferr3,Ferr7}.
In particular, the original results of the papers
\cite{Ferr3,Ferr7} were presented in the International Workshop on
Mathematical Control Theory and Finance, Lisbon, 10-14 April,
2007, and in the Conference on Nonlinear Analysis and
Optimization, June 18-24, 2008, Technion, Haifa, Israel,
respectively. It is worth mentioning that nowadays other
researchers are dedicating their time to the development of the
theory of the calculus of variations on time scales (see
\cite{AlmeidaTorres09,dicovots,zeidan2,MalTorres1,MalTorres2,natorres,Zhan}
and references therein).

\clearpage{\thispagestyle{empty}\cleardoublepage}

\chapter{Inequalities applied to some Calculus of Variations and Optimal Control
problems}\label{chap2}

In this chapter we prove a result that complements Jensen's
inequality on time scales (cf. Proposition \ref{prop0}) and state
some useful consequences of it. These are then applied in
Section~\ref{sec:app:CV} to solve some classes of variational
problems on time scales. A simple illustrative example is given in
Section~\ref{sec:ex}.

The method here proposed is direct\index{Direct methods}, in the
sense that permits to find directly the optimal solution instead
of using variational arguments and go through the usual procedure
of solving the associated $\Delta$ Euler--Lagrange equation. This
is particularly useful since even simple classes of problems of
the calculus of variations on time scales lead to dynamic
Euler--Lagrange equations for which methods to compute explicit
solutions are not known. A second advantage of the method here
promoted is that it provides directly an optimal solution, while
the variational method on time scales is based on necessary
optimality conditions, being necessary further analysis in order
to conclude if the candidate is a local minimizer, a local
maximizer, or just a saddle. Finally, while all the previous
methods of the calculus of variations on time scales only
establish local optimality, here we provide global solutions.

The use of inequalities to solve certain classes of optimal
control problems is an old idea with a rich history
\cite{wsc,otto,H:L:1932,inequalities,Sbordone}. We trust that the
present study will be the beginning of a class of direct methods
for optimal control problems on time scales, to be investigated
with the help of dynamic inequalities.


\section{Some integral inequalities}\label{ineq}

The next theorem is a generalization of the Jensen inequality on
time scales.

\begin{theorem}[Generalized Jensen's inequality \cite{wong}]
\label{theor1} Let $a,b\in\mathbb{T}$ and $c,d\in\mathbb{R}$.
Suppose $f:[a,b]^\kappa_{\mathbb{T}}\rightarrow(c,d)$ is
rd-continuous and $F:(c,d)\rightarrow\mathbb{R}$ is convex.
Moreover, let $h:[a,b]^\kappa_{\mathbb{T}}\rightarrow\mathbb{R}$
be rd-continuous with
$$\int_a^b|h(t)|\Delta t>0.$$
Then,
\begin{equation}
\label{ine1} \frac{\int_a^b |h(t)|F(f(t))\Delta
t}{\int_a^b|h(t)|\Delta t}\geq
F\left(\frac{\int_a^b|h(t)|f(t)\Delta t}{\int_a^b|h(t)|\Delta
t}\right).
\end{equation}
\end{theorem}
The next observation is crucial to solve variational problems.
Follows the statement and a proof.
\begin{prop}[cf. \cite{BFT}]
\label{prop0}If in Theorem~\ref{theor1} $F$ is strictly convex and
$h(t)\neq 0$ for all $t\in[a,b]^\kappa_\mathbb{T}$, then the
equality in (\ref{ine1}) holds if and only if $f$ is constant.
\end{prop}
\begin{proof}
Consider $x_0\in(c,d)$ defined by
$$x_0=\frac{\int_a^b|h(t)|f(t)\Delta t}{\int_a^b|h(t)|\Delta
t}.$$
From the definition of strictly convexity, there exists
$m\in\mathbb{R}$ such that
$$F(x)-F(x_0)>m(x-x_0),$$
for all $x\in (c,d)\backslash\{x_0\}$. Assume $f$ is not constant.
Then, $f(t_0)\neq x_0$ for some $t_0\in[a,b]^\kappa_{\mathbb{T}}$.
We split the proof in two cases. (i) Assume that $t_0$ is
right-dense. Then, since $f$ is rd-continuous, we have that
$f(t)\neq x_0$ on $[t_0,t_0+\delta)_\mathbb{T}$ for some
$\delta>0$. Hence,
\begin{align*}
\int_a^b |h(t)|F(f(t))\Delta t-\int_a^b |h(t)|\Delta t
F(x_0)&=\int_a^b |h(t)|[F(f(t))-F(x_0)]\Delta t\\
&>m\int_a^b |h(t)|[f(t)-x_0]\Delta t\\ &=0 \, .
\end{align*}
(ii) Assume now that $t_0$ is right-scattered. Then [note that
$\int_{t_0}^{\sigma(t_0)} f(t)\Delta t=\mu(t_0)f(t_0)$],
\begin{equation*}
\begin{split}
\int_a^b & |h(t)|F(f(t))\Delta t-\int_a^b |h(t)|\Delta t
F(x_0)\\
&= \int_a^b |h(t)|[F(f(t))-F(x_0)]\Delta t\\
&= \int_a^{t_0} |h(t)|[F(f(t))-F(x_0)]\Delta
t+\int_{t_0}^{\sigma(t_0)}
|h(t)|[F(f(t))-F(x_0)]\Delta t\\
& \qquad +\int_{\sigma(t_0)}^{b}
|h(t)|[F(f(t))-F(x_0)]\Delta t\\
&> \int_a^{t_0} |h(t)|[F(f(t))-F(x_0)]\Delta
t+m\int_{t_0}^{\sigma(t_0)} |h(t)|[f(t)-x_0]\Delta t\\
& \qquad +\int_{\sigma(t_0)}^{b} |h(t)|[F(f(t))-F(x_0)]\Delta t\\
&\geq m\left\{\int_{a}^{t_0} |h(t)|[f(t)-x_0]\Delta
t+\int_{t_0}^{\sigma(t_0)} |h(t)|[f(t)-x_0]\Delta t\right.\\
& \qquad \left.+\int_{\sigma(t_0)}^{b} |h(t)|[f(t)-x_0]\Delta t\right\}\\
&= m\int_a^b|h(t)|[f(t)-x_0]\Delta t =0 \, .
\end{split}
\end{equation*}
Finally, if $f$ is constant, it is obvious that the equality in
(\ref{ine1}) holds.
\end{proof}

\begin{remark}
If in Theorem~\ref{theor1} $F$ is a concave function, then the
inequality sign in (\ref{ine1}) must be reversed. Obviously,
Proposition~\ref{prop0} remains true if we let $F$ to be strictly
concave.
\end{remark}

Before proceeding, we state a particular case of
Theorem~\ref{theor1}.

\begin{cor}
Let $a=q^n$ and $b=q^m$ for some $n,m\in\mathbb{N}_0$ with $n<m$.
Define $f$ and $h$ on $[q^n,q^{m-1}]_{q^{\mathbb{N}_0}}$ and
assume $F:(c,d)\rightarrow\mathbb{R}$ is convex, where
$(c,d)\supset[f(q^n),f(q^{m-1})]_{q^{\mathbb{N}_0}}$. If
$$\sum_{k=n}^{m-1}(q-1)q^k|h(q^k)|>0,$$
then
\begin{equation*}
\frac{\sum_{k=n}^{m-1}q^k
|h(q^k)|F(f(q^k))}{\sum_{k=n}^{m-1}q^k|h(q^k)|}\geq
F\left(\frac{\sum_{k=n}^{m-1}q^k|h(q^k)|f(q^k)}{\sum_{k=n}^{m-1}q^k|h(q^k)|}\right).
\end{equation*}
\end{cor}
\begin{proof}
Choose $\mathbb{T}=q^{\mathbb{N}_0}=\{q^{k}:k\in\mathbb{N}_0\}$,
$q>1$, in Theorem~\ref{theor1}.
\end{proof}

We now present Jensen's inequality on time scales, complemented by
Proposition \ref{prop0}\index{Jensen's inequality on time scales}.

\begin{theorem}[cf. \cite{BFT}]
\label{theor0} Let $a,b\in\mathbb{T}$ and $c,d\in\mathbb{R}$.
Suppose $f:[a,b]^\kappa_{\mathbb{T}}\rightarrow(c,d)$ is
rd-continuous and $F:(c,d)\rightarrow\mathbb{R}$ is convex (resp.,
concave). Then,
\begin{equation}
\label{ine0} \frac{\int_a^b F(f(t))\Delta t}{b-a}\geq
F\left(\frac{\int_a^bf(t)\Delta t}{b-a}\right)
\end{equation}
(resp., the reverse inequality). Moreover, if $F$ is strictly
convex or strictly concave, then equality in (\ref{ine0}) holds if
and only if $f$ is constant.
\end{theorem}
\begin{proof}
Particular case of Theorem~\ref{theor1} and
Proposition~\ref{prop0} with $h(t)=1$ for all
$t\in[a,b]^\kappa_{\mathbb{T}}$.
\end{proof}

We now state and prove some consequences of Theorem~\ref{theor0}.

\begin{cor}[cf. \cite{BFT}]
\label{cor0} Let $a,b\in\mathbb{T}$ and $c,d\in\mathbb{R}$.
Suppose $f:[a,b]^\kappa_{\mathbb{T}}\rightarrow(c,d)$ is
rd-continuous and $F:(c,d)\rightarrow\mathbb{R}$ is such that
$F''\geq 0$ (resp., $F''\leq 0$). Then,
\begin{equation}
\label{ine2} \frac{\int_a^b F(f(t))\Delta t}{b-a}\geq
F\left(\frac{\int_a^bf(t)\Delta t}{b-a}\right)
\end{equation}
(resp., the reverse inequality). Furthermore, if $F''>0$ or
$F''<0$, equality in (\ref{ine2}) holds if and only if $f$ is
constant.
\end{cor}
\begin{proof}
This follows immediately from Theorem \ref{theor0} and the facts
that a function $F$ with $F''\geq 0$ (resp., $F''\leq 0$) is
convex (resp., concave) and with $F''>0$ (resp., $F''\leq 0$) is
strictly convex (resp., strictly concave).
\end{proof}

\begin{cor}[cf. \cite{BFT}]
Let $a,b\in\mathbb{T}$ and $c,d\in\mathbb{R}$. Suppose
$f:[a,b]^\kappa_{\mathbb{T}}\rightarrow(c,d)$ is rd-continuous and
$\varphi,\psi:(c,d)\rightarrow\mathbb{R}$ are continuous functions
such that $\varphi^{-1}$ exists, $\psi$ is strictly increasing,
and $\psi\circ\varphi^{-1}$ is convex (resp., concave) on
\textup{Im}($\varphi$). Then,
\begin{equation*}
\psi^{-1}\left(\frac{\int_a^b \psi(f(t))\Delta t}{b-a}\right)\geq
\varphi^{-1}\left(\frac{\int_a^b \varphi(f(t))\Delta
t}{b-a}\right)
\end{equation*}
(resp., the reverse inequality). Furthermore, if
$\psi\circ\varphi^{-1}$ is strictly convex or strictly concave,
the equality holds if and only if $f$ is constant.
\end{cor}
\begin{proof}
Since $\varphi$ is continuous and $\varphi\circ f$ is
rd-continuous, it follows from Theorem~\ref{theor0} with
$f=\varphi\circ f$ and $F=\psi\circ\varphi^{-1}$ that
$$\frac{\int_a^b (\psi\circ\varphi^{-1})((\varphi\circ
f)(t))\Delta t}{b-a}\geq
(\psi\circ\varphi^{-1})\left(\frac{\int_a^b (\varphi\circ
f)(t)\Delta t}{b-a}\right).$$ Since $\psi$ is strictly increasing,
we obtain
$$\psi^{-1}\left(\frac{\int_a^b \psi(f(t))\Delta
t}{b-a}\right)\geq \varphi^{-1}\left(\frac{\int_a^b
\varphi(f(t))\Delta t}{b-a}\right).$$ Finally, when
$\psi\circ\varphi^{-1}$ is strictly convex, the equality holds if
and only if $\varphi\circ f$ is constant, or equivalently (since
$\varphi$ is invertible), $f$ is constant. The case when
$\psi\circ\varphi^{-1}$ is concave is treated analogously.
\end{proof}

\begin{cor}[cf. \cite{BFT}]
\label{cor1} Assume
$f:[a,b]^\kappa_{\mathbb{T}}\rightarrow\mathbb{R}$ is
rd-continuous and positive. If $\alpha<0$ or $\alpha>1$, then
\begin{equation*}
\int_a^b (f(t))^\alpha\Delta t\geq(b-a)^{1-\alpha}\left(\int_a^b
f(t)\Delta t\right)^\alpha \, .
\end{equation*}
If $0<\alpha<1$, then
\begin{equation*}
\int_a^b (f(t))^\alpha\Delta t\leq(b-a)^{1-\alpha}\left(\int_a^b
f(t)\Delta t\right)^\alpha \, .
\end{equation*}
Furthermore, in both cases equality holds if and only if $f$ is
constant.
\end{cor}
\begin{proof}
Define $F(x)=x^\alpha,\ x>0$. Then
$$F''(x)=\alpha(\alpha-1)x^{\alpha-2},\quad x>0.$$
Hence, when $\alpha<0$ or $\alpha>1$, $F''>0$, \textrm{i.e.}, $F$
is strictly convex. When $0<\alpha<1$, $F''<0$, \textrm{i.e.}, $F$
is strictly concave. Applying Corollary \ref{cor0} with this
function $F$, we obtain the above inequalities with equality if
and only if $f$ is constant.
\end{proof}

\begin{cor}[cf. \cite{BFT}]
Assume $f:[a,b]^\kappa_{\mathbb{T}}\rightarrow\mathbb{R}$ is
rd-continuous and positive. If $\alpha<-1$ or $\alpha>0$, then
\begin{equation*}
\left(\int_a^b \frac{1}{f(t)}\Delta t\right)^\alpha\int_a^b
(f(t))^\alpha\Delta t\geq(b-a)^{1+\alpha} \, .
\end{equation*}
If $-1<\alpha<0$, then
\begin{equation*}
\left(\int_a^b \frac{1}{f(t)}\Delta t\right)^\alpha\int_a^b
(f(t))^\alpha\Delta t\leq(b-a)^{1+\alpha}.
\end{equation*}
Furthermore, in both cases the equality holds if and only if $f$
is constant.
\end{cor}
\begin{proof}
This follows from Corollary \ref{cor1} by replacing $f$ by $1/f$
and $\alpha$ by $-\alpha$.
\end{proof}

\begin{cor}[cf. \cite{BFT}]\label{cor3}
If $f:[a,b]^\kappa_{\mathbb{T}}\rightarrow\mathbb{R}$ is
rd-continuous, then
\begin{equation}\label{ine3}
\int_a^b e^{f(t)}\Delta
t\geq(b-a)e^{\frac{1}{b-a}\int_a^bf(t)\Delta t} \, .
\end{equation}
Moreover, the equality in (\ref{ine3}) holds if and only if $f$ is
constant.
\end{cor}
\begin{proof}
Choose $F(x)=e^x$, $x\in\mathbb{R}$, in Corollary~\ref{cor0}.
\end{proof}

\begin{cor}[cf. \cite{BFT}]\label{cor4}
If $f:[a,b]^\kappa_{\mathbb{T}}\rightarrow\mathbb{R}$ is
rd-continuous and positive, then
\begin{equation}\label{ine4}
\int_a^b \ln(f(t))\Delta t\leq(b-a)\ln\left({\frac{1}{b-a}\int_a^b
f(t)\Delta t}\right) \, .
\end{equation}
Moreover, the equality in (\ref{ine4}) holds if and only if $f$ is
constant.
\end{cor}
\begin{proof}
Let $F(x)=\ln(x)$, $x>0$, in Corollary~\ref{cor0}.
\end{proof}

\begin{cor}[cf. \cite{BFT}]\label{cor5}
If $f:[a,b]^\kappa_{\mathbb{T}}\rightarrow\mathbb{R}$ is
rd-continuous and positive, then
\begin{equation}
\label{eq:cor5} \int_a^b f(t)\ln(f(t))\Delta t\geq\int_a^b
f(t)\Delta t\ln\left({\frac{1}{b-a}\int_a^b f(t)\Delta t}\right)
\, .
\end{equation}
Moreover, the equality in \eqref{eq:cor5} holds if and only if $f$
is constant.
\end{cor}
\begin{proof}
Let $F(x)=x\ln(x)$, $x>0$. Then, $F''(x)=1/x$, \textrm{i.e.},
$F''(x)>0$ for all $x>0$. By Corollary \ref{cor0}, we get
$$\frac{1}{b-a}\int_a^b f(t)\ln(f(t))\Delta
t\geq\frac{1}{b-a}\int_a^b f(t)\Delta
t\ln\left({\frac{1}{b-a}\int_a^b f(t)\Delta t}\right),$$ and the
result follows.
\end{proof}


\section{Applications to the Calculus of Variations}
\label{sec:app:CV}

We now show how the results obtained in Section~\ref{ineq} can be
applied to determine the minimum or maximum of certain problems of
calculus of variations and optimal control on time scales.

\begin{theorem}[cf. \cite{BFT}]\label{theor3}
Let $\mathbb{T}$ be a time scale, $a$, $b\in\mathbb{T}$ with $a <
b$, and $\varphi:\mathbb{R}\rightarrow\mathbb{R}$ be a positive
and continuous function. Consider the functional
$$F[y(\cdot)]=\int_a^b
\left[\left\{\int_0^1\varphi(y(t)+h\mu(t)y^\Delta(t))dh\right\}
y^\Delta(t)\right]^\alpha\Delta t \, , \quad
\alpha\in\mathbb{R}\backslash\{0,1\} \, ,$$ defined on all
$C_{\textup{rd}}^1$ functions
$y:[a,b]_{\mathbb{T}}\rightarrow\mathbb{R}$ satisfying
$y^\Delta(t)>0$ on $[a,b]^\kappa_{\mathbb{T}}$, $y(a)=0$, and
$y(b)=B$. Define a function $G(x)=\int_0^x\varphi(s)ds$, $x\geq
0$, and let $G^{-1}$ denote its inverse. Let
\begin{equation}
\label{equacao0} C=\frac{\int_0^B \varphi(s)ds}{b-a}.
\end{equation}
\begin{description}
\item[(i)] If $\alpha<0$ or $\alpha>1$, then the minimum of $F$
occurs when
$$y(t)=G^{-1}(C(t-a)),\quad t\in[a,b]_\mathbb{T},$$
and $F_{\min}=(b-a)C^\alpha$. \item[(ii)] If $0<\alpha<1$, then
the maximum of $F$ occurs when
$$y(t)=G^{-1}(C(t-a)),\quad t\in[a,b]_\mathbb{T},$$
and $F_{\max}=(b-a)C^\alpha$.
\end{description}
\end{theorem}

\begin{remark}
Since $\varphi$ is continuous and positive, $G$ and $G^{-1}$ are
well defined.
\end{remark}
\begin{remark}
In cases $\alpha=0$ or $\alpha=1$ there is nothing to minimize or
maximize, \textrm{i.e.}, the problem of extremizing $F[y(\cdot)]$
is trivial. Indeed, if $\alpha=0$, then $F[y(\cdot)]=b-a$; if
$\alpha=1$, then it follows from Theorem \ref{teor1} that
\begin{align*}
F[y(\cdot)]&=\int_a^b
\left\{\int_0^1\varphi(y(t)+h\mu(t)y^\Delta(t))dh\right\}
y^\Delta(t)\Delta t\\ &=\int_a^b (G\circ y)^\Delta(t)\Delta t\\
&=G(B) \, .
\end{align*}
In both cases $F$ is a constant and does not depend on function
$y$.
\end{remark}

\begin{proof}[Proof of Theorem~\ref{theor3}]
Suppose that $\alpha<0$ or $\alpha>1$. Using Corollary~\ref{cor1}
we can write
\begin{multline*}
F[y(\cdot)]\geq
(b-a)^{1-\alpha}\left[\int_a^b\left\{\int_0^1\varphi(y(t)+h\mu(t)y^\Delta(t))dh\right\}
y^\Delta(t)\Delta
t\right]^\alpha\\=(b-a)^{1-\alpha}(G(y(b))-G(y(a)))^\alpha \, ,
\end{multline*}
where the equality holds if and only if
$$\left\{\int_0^1\varphi(y(t)+h\mu(t)y^\Delta(t))dh\right\}
y^\Delta(t)=c,\quad \mbox{for some}\ c\in\mathbb{R},\quad
t\in[a,b]^\kappa_\mathbb{T}.$$ Using Theorem \ref{teor1} we arrive
at
$$(G\circ y)^\Delta(t)=c.$$
$\Delta$-integrating from $a$ to $t$ yields (note that $y(a)=0$
and $G(0)=0$)
$$G(y(t))=c(t-a),$$
from which we get
$$y(t)=G^{-1}(c(t-a)).$$
The value of $c$ is obtained using the boundary condition
$y(b)=B$:
$$c=\frac{G(B)}{b-a}=\frac{\int_0^B \varphi(s)ds}{b-a}=C,$$
with $C$ as in (\ref{equacao0}). Finally, in this case
$$F_{\min}=\int_a^b C^\alpha\Delta t=(b-a)C^\alpha.$$
The proof of the second part of the theorem is done analogously
using the second part of Corollary~\ref{cor1}.
\end{proof}

\begin{remark}
We note that the optimal solution found in the proof of the
previous theorem satisfies $y^\Delta>0$. Indeed,
\begin{align*}
y^\Delta(t)&=\left(G^{-1}(C(t-a))\right)^\Delta\\
&=\int_0^1 (G^{-1})'[C(t-a)+h\mu(t)C]dh\ C\\ &>0,
\end{align*}
because $C>0$ and $(G^{-1})'(x)=\frac{1}{\varphi(G^{-1}(x))}>0$
for all $x\geq 0$.
\end{remark}

\begin{theorem}[cf. \cite{BFT}]
Let $\varphi:[a,b]^\kappa_\mathbb{T}\rightarrow\mathbb{R}$ be a
positive and rd-continuous function. Then, among all
$C_{\textup{rd}}^1$ functions
$y:[a,b]_{\mathbb{T}}\rightarrow\mathbb{R}$ with $y(a)=0$ and
$y(b)=B$, the functional
$$F[y(\cdot)]=\int_a^b\varphi(t)e^{y^\Delta(t)}\Delta t$$
has minimum value $F_{\min}=(b-a)e^C$ attained when
$$y(t)=-\int_a^t\ln(\varphi(s))\Delta s+C(t-a),\quad
t\in[a,b]_\mathbb{T},$$ where
\begin{equation}
\label{equa1} C=\frac{\int_a^b\ln(\varphi(t))\Delta t+B}{b-a} \, .
\end{equation}
\end{theorem}
\begin{proof}
By Corollary \ref{cor3},
\begin{multline*}
F[y(\cdot)]=\int_a^b e^{\ln(\varphi(t))+y^\Delta(t)}\Delta t\\
\geq(b-a)e^{\frac{1}{b-a}\int_a^b
[\ln(\varphi(t))+y^\Delta(t)]\Delta
t}=(b-a)e^{\frac{1}{b-a}\left[\int_a^b \ln(\varphi(t))\Delta
t+B\right]} \, ,
\end{multline*}
with $F(y(\cdot))=(b-a)e^{\frac{1}{b-a}\left[\int_a^b
\ln(\varphi(t)) \Delta t+B\right]}$ if and only if
\begin{equation}
\label{eq:prf:d} \ln(\varphi(t))+y^\Delta(t)=c,\quad \mbox{for
some}\ c\in\mathbb{R},\quad t\in[a,b]^\kappa_\mathbb{T} \, .
\end{equation}
Integrating \eqref{eq:prf:d} from $a$ to $t$ (note that $y(a)=0$)
gives
$$
y(t)=-\int_a^t\ln(\varphi(s))\Delta s+c(t-a),\quad
t\in[a,b]_\mathbb{T} \, .
$$
Using the boundary condition $y(b)=B$ we have
$$c=\frac{\int_a^b\ln(\varphi(t))\Delta t+B}{b-a}=C,$$
with $C$ as in (\ref{equa1}). A simple calculation shows that
$F_{\min}=(b-a)e^C$.
\end{proof}

\begin{remark}
If we let $\mathbb{T}=\mathbb{R}$ in the previous theorem we get
\cite[Theorem~3.4]{wsc}.
\end{remark}

\begin{theorem}[cf. \cite{BFT}]
\label{thm7} Let
$\varphi:[a,b]^\kappa_\mathbb{T}\rightarrow\mathbb{R}$ be a
positive and rd-continuous function. Then, among all
$C_{\textup{rd}}^1$-functions
$y:[a,b]_{\mathbb{T}}\rightarrow\mathbb{R}$ satisfying
$y^\Delta>0$, $y(a)=0$, and $y(b)=B$, with
\begin{equation}
\label{ine5} \frac{B+\int_a^b\varphi(s)\Delta
s}{b-a}>\varphi(t),\quad t\in[a,b]_\mathbb{T}^\kappa \, ,
\end{equation}
the functional
$$F[y(\cdot)]=\int_a^b[\varphi(t)+y^\Delta(t)]\ln[\varphi(t)+y^\Delta(t)]\Delta
t$$ has minimum value $F_{\min}=(b-a)C\ln(C)$ attained when
$$y(t)=C(t-a)-\int_a^t\varphi(s)\Delta s,\quad
t\in[a,b]_\mathbb{T} \, ,
$$
where
\begin{equation}
\label{equa2} C=\frac{B+\int_a^b\varphi(s)\Delta s}{b-a} \, .
\end{equation}
\end{theorem}

\begin{proof}
By Corollary~\ref{cor5},
\begin{multline*}
F[y(\cdot)]\geq\int_a^b [\varphi(t)+y^\Delta(t)]\Delta
t\ln\left({\frac{1}{b-a}\int_a^b [\varphi(t)+y^\Delta(t)]\Delta
t}\right)\\ =\left(\int_a^b \varphi(t)\Delta
t+B\right)\ln\left({\frac{\int_a^b \varphi(t)\Delta
t+B}{b-a}}\right)
\end{multline*}
with $F[y(\cdot)] = \left(\int_a^b \varphi(t)\Delta
t+B\right)\ln\left({\frac{\int_a^b \varphi(t)\Delta
t+B}{b-a}}\right)$ if and only if
$$\varphi(t)+y^\Delta(t)=c,\quad \mbox{for some}\
c\in\mathbb{R},\quad t\in[a,b]^\kappa_\mathbb{T}.$$ Upon
integration from $a$ to $t$ (note that $y(a)=0$),
$$y(t)=c(t-a)-\int_a^t\varphi(s)\Delta s,\
t\in[a,b]_\mathbb{T}.$$ Using the boundary condition $y(b)=B$, we
have
$$c=\frac{B+\int_a^b\varphi(s)\Delta s}{b-a}=C,$$
where $C$ is as in (\ref{equa2}). Note that with this choice of
$y$ we have, using (\ref{ine5}), that
$y^\Delta(t)=C-\varphi(t)>0$, $t\in[a,b]^\kappa_\mathbb{T}$. It
follows that $F_{\min}=(b-a)C\ln(C)$.
\end{proof}

In order to close this subject we would like to point out that
Theorem~3.6 in \cite{wsc} is not true. This is due to the fact
that the bound on the functional $I$ considered in the proof is
not constant. Let us quote the ``theorem":
\begin{quote}
Let $\varphi:\mathbb{R}\rightarrow\mathbb{R}$ be a positive and
continuous function and $a>0$. Then, among all $C^1$ functions
$y:[0,a]\rightarrow\mathbb{R}$ satisfying $y'>0$, $y(0)=0$, and
$y(a)=A$, the functional
$$I=\int_0^a\ln(\varphi(x)y'(x))dx$$
attains its maximum when
$$y=\frac{1}{C}\int_0^x\frac{1}{\varphi(s)}ds,$$
where
\begin{equation}\label{e2}
C=\frac{1}{A}\int_0^a\frac{1}{\varphi(s)}ds,
\end{equation}
and
\begin{equation}\label{e3}
I_{\max}=-a\ln(C).
\end{equation}
\end{quote}
Now we give a counterexample to \cite[Theorem~3.6]{wsc}. Let
$a=A=1$, $\varphi(x)=x+1$, and $\tilde{y}(x)=x$ for all
$x\in[0,1]$. Then, the hypotheses of the ``theorem" are satisfied.
Moreover,
$$I[\tilde{y}(x)]=\int_0^1\ln(\varphi(x)\tilde{y}'(x))dx
=\left[(x+1)(\ln(x+1)-1)\right]_{x=0}^{x=1}=2\ln(2)-1\approx
0.386.$$ According with \eqref{e2} and \eqref{e3} the maximum of
the functional $I$ is given by $I_{\max}=-\ln(C)$, where
$$C=\int_0^1\frac{1}{\varphi(s)}ds.$$
A simple calculation shows that $C=\ln(2)$, hence
$I_{\max}=-\ln(\ln(2))\approx 0.367$. Therefore,
$I[\tilde{y}(x)]>I_{\max}$, which proves our claim.


\subsection{An example} \label{sec:ex}

Let us consider $\mathbb{T}=\mathbb{Z}$, $a=0$, $b=5$, $B=25$ and
$\varphi(t)=2t+1$ in Theorem~\ref{thm7}:
\begin{cor}[cf. \cite{BFT}]
The functional
$$F[y(\cdot)]=\sum_{t=0}^4[(2t+1)+(y(t+1)-y(t))]\ln[(2t+1)+(y(t+1)-y(t))],$$
defined for all $y:[0,5]\cap\mathbb{Z}\rightarrow\mathbb{R}$ such
that $y(t+1)>y(t)$ for all $t\in[0,4]\cap\mathbb{T}$, attains its
minimum when
$$y(t)=10t-t^2,\quad t\in[0,5]\cap\mathbb{Z},$$
and $F_{\min}=50\ln(10)$.
\end{cor}
\begin{proof}
First we note that $\max\{\varphi(t):t\in[0,4]\cap\mathbb{Z}\}=9$,
hence
$$\frac{B+\sum_{k=0}^4\varphi(k)}{b-a}=\frac{25+25}{5}=10>9\geq\varphi(t).$$
Observing that, when $\mathbb{T}=\mathbb{Z}$, $(t^2)^\Delta=2t+1$,
we just have to invoke Theorem~\ref{thm7} to get the desired
result.
\end{proof}

\section{State of the Art}

The results of this chapter are already published
\cite{BFT}. Direct methods are an important subject to the
calculus of variations theory and further research is in progress,
extending Leitmann's direct method to time scales
\cite{MalTorres3}.

\clearpage{\thispagestyle{empty}\cleardoublepage}

\chapter{Inequalities on Time Scales}
\label{sec:Prel}

This chapter is devoted to the development of some integral
inequalities as well to some of its applications. These
inequalities will be used to prove existence of solution(s) to
some dynamic equations and to estimate them, and this is shown in
Sections \ref{sec:mainResults} and \ref{ineconstan}. In Section
\ref{diamond} we prove H$\ddot{\mbox{o}}$lder, Cauchy--Schwarz,
Minkowski and Jensen's type inequalities in the more general
setting of $\Diamond_\alpha$-integrals. Finally, in Section \ref
{duasvar}, we obtain some Gronwall--Bellman--Bihari type
inequalities for functions depending on two time scales variables.

Throughout we use the notations $\mathbb{R}^+_0=[0,\infty)$ and
$\mathbb{R}^+=(0,\infty)$.


\section{Gronwall's type inequalities}
\label{sec:mainResults}

We start by proving a lemma which is essential in the proofs of
the next theorems.

\begin{lemma}[cf. \cite{Ferr4}]
\label{lemimp} Let $a,b\in\mathbb{T}$, and consider a function
$r\in C_{\textrm{rd}}^1([a,b]_{\mathbb{T}},\mathbb{R}^+)$ with
$r^\Delta(t)\geq 0$ on $[a,b]_{\mathbb{T}}^\kappa$. Suppose that a
function $g\in C(\mathbb{R}^+_0,\mathbb{R}^+_0)$ is positive and
nondecreasing on $\mathbb{R}^+$ and define,
\begin{equation*}
G(x)=\int_{x_0}^{x}\frac{ds}{g(s)},
\end{equation*}
where $x\geq 0$, $x_0\geq 0$ if $\int_{0}^x\frac{ds}{g(s)}<\infty$
and $x>0$, $x_0>0$ if $\int_{0}^x\frac{ds}{g(s)}=\infty$. Then,
for each $t\in[a,b]_{\mathbb{T}}$, we have
\begin{equation}
\label{seila7} G(r(t))\leq
G(r(a))+\int_{a}^{t}\frac{r^\Delta(\tau)}{g(r(\tau))}\Delta\tau.
\end{equation}
\end{lemma}

\begin{proof}
Since $g$ is positive and nondecreasing on $(0,\infty)$, we have,
successively, that
\begin{equation}
\label{seila8}
\begin{gathered}
r(t) \leq r(t)+h\mu(t)r^\Delta(t) \, ,\\
g(r(t)) \leq g(r(t)+h\mu(t)r^\Delta(t)) \, ,\\
\frac{1}{g(r(t)+h\mu(t)r^\Delta(t))} \leq\frac{1}{g(r(t))} \, ,\\
\int_0^1\frac{1}{g(r(t)+h\mu(t)r^\Delta(t))}dh \leq\int_0^1\frac{1}{g(r(t))}dh=\frac{1}{g(r(t))} \, , \\
\left\{\int_0^1\frac{1}{g(r(t)+h\mu(t)r^\Delta(t))}dh\right\}
r^\Delta(t) \leq\frac{r^\Delta(t)}{g(r(t))} \, ,
\end{gathered}
\end{equation}
for all $t\in[a,b]_{\mathbb{T}}^\kappa$ and $h\in[0,1]$. By
$\Delta$-integrating the last inequality in (\ref{seila8}) from
$a$ to $t$ and having in mind that Theorem~\ref{teor1} guarantees
that
\begin{align}
(G\circ r)^\Delta(t)&=\left\{\int_0^1
G'(r(t)+h\mu(t)r^\Delta(t))dh\right\}r^\Delta(t)\nonumber\\
&=\left\{\int_0^1
\frac{1}{g(r(t)+h\mu(t)r^\Delta(t))}dh\right\}r^\Delta(t)\nonumber,
\end{align} we obtain the desired result [note that the case $t=b$ if $\rho(b)<b$ is also proved because of \eqref{desbasic}].
\end{proof}

\begin{theorem}[cf. \cite{Ferr4}]
\label{thma1} Let $u(t)$ and $f(t)$ be nonnegative rd-continuous
functions in the time scales interval
$\mathbb{T}_\ast:=[a,b]_{\mathbb{T}}$ and
$\mathbb{T}_\ast^\kappa$, respectively. Let $k(t,s)$ be defined as
in Theorem~\ref{teork} in such a way that $k(t,s)$ and
$k^{\Delta_1}(t,s)$ are nonnegative for every
$t,s\in\mathbb{T}_\ast$ with $s\leq t$ for which they are defined
(it is assumed that $k$ is not identically zero on
$\mathbb{T}_\ast^\kappa\times\mathbb{T}_\ast^{\kappa^2}$). Let
$\Phi\in C(\mathbb{R}^+_0,\mathbb{R}^+_0)$ be a nondecreasing,
subadditive and submultiplicative function, such that $\Phi(u)>0$
for $u>0$ and let $W\in C(\mathbb{R}^+_0,\mathbb{R}^+_0)$ be a
nondecreasing function such that for $u> 0$ we have $W(u)>0$.
Assume that $a(t)$ is a positive rd-continuous function and
nondecreasing for $t\in\mathbb{T}_\ast$. If
\begin{equation}
\label{equa0} u(t)\leq a(t)+\int_a^t f(s)u(s)\Delta
s+\int_a^tf(s)W\left(\int_a^s
k(s,\tau)\Phi(u(\tau))\Delta\tau\right)\Delta s,
\end{equation}
for $a\leq\tau\leq s\leq t\leq b$, $\tau, s, t\in\mathbb{T}_\ast$,
then for all $t\in\mathbb{T}_\ast$ satisfying
\begin{equation*}
\Psi(\zeta)+\int_a^{\rho(t)} k(\rho(t),s)\Phi(p(s))\Phi(\int_a^s
f(\tau)\Delta\tau)\Delta s\in \Dom(\Psi^{-1})
\end{equation*}
we have
\begin{multline}
\label{eq11}
u(t)\leq p(t) a(t) \\
+ p(t) \int_a^t f(s)W\left[\Psi^{-1} \left(\Psi(\zeta)
+\int_a^sk(s,\tau)\Phi(p(\tau)) \Phi\left(\int_a^\tau
f(\theta)\Delta\theta\right)\Delta\tau\right)\right]\Delta s \, ,
\end{multline}
where
\begin{gather}
\label{p} p(t) = 1+\int_a^t f(s)e_{f}(t,\sigma(s))\Delta s \, ,\\
\zeta = \int_a^{\rho(b)} k(\rho(b),s)\Phi(p(s)a(s))\Delta s \, , \nonumber \\
\label{p1}\Psi(x) = \int_{x_0}^x\frac{1}{\Phi(W(s))}ds,\ x>0,\
x_0>0 \, ,
\end{gather}
and, as usual, $\Psi^{-1}$ denotes the inverse of $\Psi$.
\end{theorem}
\begin{remark}
We are interested to study the situation when $k$ is not
identically zero on
$\mathbb{T}_\ast^\kappa\times\mathbb{T}_\ast^{\kappa^2}$. That
comprise the new cases, not considered previously in the
literature. The case $k(t,s) \equiv 0$ was studied in
\cite[Theorem~3.1]{Pach} and is not discussed here.
\end{remark}
\begin{proof}
Define the function $z(t)$ in $\mathbb{T}_\ast$ by
\begin{equation}
\label{equacao1linha} z(t)=a(t)+\int_a^tf(s)W\left(\int_a^s
k(s,\tau)\Phi(u(\tau))\Delta\tau\right)\Delta s \, .
\end{equation}
Then, (\ref{equa0}) can be rewritten as
\begin{equation*}
u(t)\leq z(t)+\int_a^t f(s)u(s)\Delta s.
\end{equation*}
Clearly, $z(t)$ is rd-continuous in $t\in\mathbb{T}_\ast$. Using
Theorem~\ref{gronw}, we get
\begin{equation*}
u(t)\leq z(t)+\int_{a}^t f(s)z(s)e_f(t,\sigma(s))\Delta s \, .
\end{equation*}
Moreover, it is easy to see that $z(t)$ is nondecreasing in
$t\in\mathbb{T}_\ast$. We get
\begin{equation}
\label{equacao2} u(t)\leq z(t)p(t),
\end{equation}
where $p(t)$ is defined by (\ref{p}). Define
$$v(t)=\int_a^t k(t,s)\Phi(u(s))\Delta s,\ t\in\mathbb{T}_\ast^\kappa.$$
From (\ref{equacao2}), and taking into account the properties of
$\Phi$, we get
\begin{equation*}
\begin{split}
v(t)&\leq\int_a^t
k(t,s)\Phi\left[p(s)\left(a(s)+\int_a^sf(\tau)W(v(\tau))
\Delta\tau\right)\right]\Delta s\\
&\leq\int_a^t k(t,s)\Phi(p(s)a(s))\Delta s+\int_a^t
k(t,s)\Phi\left(p(s)\int_a^sf(\tau)W(v(\tau))\Delta\tau\right)\Delta
s\\
&\leq\int_a^{\rho(b)} k(\rho(b),s)\Phi(p(s)a(s))\Delta s\\
& \qquad +\int_a^t
k(t,s)\Phi\left(p(s)\int_a^sf(\tau)\Delta\tau\right)
\Phi(W(v(s))\Delta
s\\
&=\zeta+\int_a^t
k(t,s)\Phi\left(p(s)\int_a^sf(\tau)\Delta\tau\right)
\Phi(W(v(s))\Delta s \, .
\end{split}
\end{equation*}
Define function $r(t)$ on $\mathbb{T}_\ast^\kappa$ by
$$r(t) = \zeta+\int_a^t
k(t,s)\Phi\left(p(s)\int_a^sf(\tau)\Delta\tau\right)
\Phi(W(v(s))\Delta s \, .$$ Since $p$ and $a$ are positive
functions, we have that $\Phi(a(s)p(s))>0$ for all
$s\in\mathbb{T}_\ast$. Since $k^{\Delta_1} \geq 0$, we must have
$\zeta>0$, hence $r(t)$ is a positive function on
$\mathbb{T}_\ast^\kappa$. In addition, $r(t)$ is
$\Delta$-differentiable on $\mathbb{T}_\ast^{\kappa^2}$ with
\begin{align}
r^\Delta(t)&=k(\sigma(t),t)\Phi\left(p(t)\int_a^t
f(\tau)\Delta\tau\right) \Phi(W(v(t))\nonumber\\
&\ \ \ +\int_a^t
k^{\Delta_1}(t,s)\Phi\left(p(s)\int_a^sf(\tau)\Delta\tau\right)
\Phi(W(v(s))\Delta s\nonumber\\
& \label{equacao4}\\
&\leq\Phi(W(r(t)))\left[k(\sigma(t),t)\Phi\left(p(t)\int_a^t
f(\tau)\Delta\tau\right)\right.\nonumber\\
&\ \ \ +\left.\int_a^t
k^{\Delta_1}(t,s)\Phi\left(p(s)\int_a^sf(\tau)\Delta\tau\right)\Delta
s\right].\nonumber
\end{align}
Dividing both sides of inequality (\ref{equacao4}) by
$\Phi(W(r(t)))$, we obtain
\begin{equation*}
\frac{r^\Delta(t)}{\Phi(W(r(t)))}\leq\left[\int_a^t
k(t,s)\Phi\left(p(s)\int_a^sf(\tau)\Delta\tau\right)\Delta
s\right]^\Delta.
\end{equation*}
Let us consider the function $\Psi$ defined by (\ref{p1}).
$\Delta$-integrating this last inequality from $a$ to $t$ and
using Lemma~\ref{lemimp}, we obtain
$$\Psi(r(t))\leq\Psi(r(a))+\int_a^t
k(t,s)\Phi\left(p(s)\int_a^sf(\tau)\Delta\tau\right)\Delta s,$$
from which it follows that
\begin{equation}
\label{seila5} r(t)\leq\Psi^{-1}\left(\Psi(\zeta)+\int_a^t
k(t,s)\Phi(p(s))\Phi(\int_a^s f(\tau)\Delta\tau)\Delta s\right),\
t\in\mathbb{T}_\ast^\kappa.
\end{equation}
Combining (\ref{seila5}), (\ref{equacao2}) and
(\ref{equacao1linha}), we obtain the desired inequality
(\ref{eq11}).
\end{proof}
If we let $\mathbb{T}=\mathbb{R}$ in Theorem~\ref{thma1}, we get
\cite[Theorem~2.1]{motivacao}. If in turn we consider
$\mathbb{T}=\mathbb{Z}$, then we obtain the following result:
\begin{cor}[cf. \cite{Ferr4}]
Let $u(t)$ and $f(t)$ be nonnegative functions in the time scales
interval $\mathbb{T}_\ast:=[a,b]_{\mathbb{Z}}$ and
$[a,b-1]_{\mathbb{Z}}$, respectively. Let $k(t,s)$ be defined as
in Theorem~\ref{teork} in such a way that $k(t,s)$ and
$k^{\Delta_1}(t,s)=k(\sigma(t),s)-k(t,s)$ are nonnegative for
every $t,s\in\mathbb{T}_\ast$ with $s\leq t$ for which they are
defined (it is assumed that $k$ is not identically zero on
$[a,b-1]_{\mathbb{T}_\ast}\times[a,b-2]_{\mathbb{T}_\ast}$). Let
$\Phi\in C(\mathbb{R}^+_0,\mathbb{R}^+_0)$ be a nondecreasing,
subadditive and submultiplicative function such that $\Phi(u)>0$
for $u>0$ and let $W\in C(\mathbb{R}^+_0,\mathbb{R}^+_0)$ be a
nondecreasing function such that for $u> 0$ we have $W(u)>0$.
Assume that $a(t)$ is a positive and nondecreasing  function for
$t\in\mathbb{T}_\ast$. If
\begin{equation*}
u(t)\leq a(t)+\sum_{s=a}^{t-1}
f(s)u(s)+\sum_{s=a}^{t-1}f(s)W\left(\sum_{\tau=a}^{s-1}
k(s,\tau)\Phi(u(\tau))\right),
\end{equation*}
for $a\leq\tau\leq s\leq t\leq b$, $\tau, s, t\in\mathbb{T}_\ast$,
then for all $t\in\mathbb{T}_\ast$ satisfying
\begin{equation*}
\Psi(\zeta)+\sum_{s=a}^{t-2}
k(t-1,s)\Phi(p(s))\Phi(\sum_{\tau=a}^{s-1} f(\tau))\in
\Dom(\Psi^{-1})
\end{equation*}
we have
\begin{equation*}
u(t)\leq p(t)\left\{a(t)+\sum_{s=a}^{t-1} f(s)W\left[\Psi^{-1}
\left(\Psi(\zeta)+\sum_{\tau=a}^{s-1}
k(s,\tau)\Phi(p(\tau))\Phi(\sum_{\theta=a}^{\tau-1}
f(\theta))\right)\right]\right\},
\end{equation*}
where
\begin{eqnarray*}
p(t)=1+\sum_{s=a}^{t-1} f(s)e_{f}(t,s+1) \, ,\\
\zeta=\sum_{s=a}^{b-1} k(b-1,s)\Phi(p(s)a(s)) \, ,\\
\Psi(x)=\int_{x_0}^x\frac{1}{\Phi(W(s))}ds,\ x>0,\ x_0>0 \, ,
\end{eqnarray*}
and $\Psi^{-1}$ is the inverse of $\Psi$.
\end{cor}

For the particular case $\mathbb{T}=\mathbb{R}$,
Theorem~\ref{thma2} generalizes the result obtained by Oguntuase
in \cite[Theorems ~2.3 and 2.9]{og}.

\begin{theorem}[cf. \cite{Ferr4}]
\label{thma2} Suppose that $u(t)$ is a nonnegative rd-continuous
function in the time scales interval
$\mathbb{T}_\ast=[a,b]_\mathbb{T}$ and that $h(t)$, $f(t)$ are
nonnegative rd-continuous functions in the time scales interval
$\mathbb{T}_\ast^\kappa$. Assume that $b(t)$ is a nonnegative
rd-continuous function and not identically zero on
$\mathbb{T}_\ast^{\kappa^2}$. Let $\Phi(u)$, $W(u)$, and $a(t)$ be
as defined in Theorem~\ref{thma1}. If
\begin{equation*}
u(t)\leq a(t)+\int_a^t f(s)u(s)\Delta
s+\int_a^tf(s)h(s)W\left(\int_a^s
b(\tau)\Phi(u(\tau))\Delta\tau\right)\Delta s,
\end{equation*}
for $a\leq\tau\leq s\leq t\leq b$, $\tau, s, t\in\mathbb{T}_\ast$,
then for all $t\in\mathbb{T}_\ast$ satisfying
$$
\Psi(\xi)+\int_a^{\rho(t)} b(\tau)\Phi(p(\tau))\Phi(\int_a^\tau
f(\theta)h(\theta)\Delta\theta)\Delta\tau\in \Dom(\Psi^{-1})
$$
we have
\begin{multline*}
u(t)\leq p(t) a(t)\\
+p(t) \int_a^t f(s)h(s)W\left[\Psi^{-1}
\left(\Psi(\xi)+\int_a^sb(\tau)\Phi(p(\tau))\Phi(\int_a^\tau
f(\theta)h(\theta)\Delta\theta)\Delta\tau\right)\right]\Delta s \,
,
\end{multline*}
where $p(t)$ is defined by (\ref{p}), $\Psi$ is defined by
(\ref{p1}), and
\begin{eqnarray*}
\xi=\int_a^{\rho(b)} b(s)\Phi(p(s)a(s))\Delta s.
\end{eqnarray*}
\end{theorem}

\begin{proof}
similar to the proof of Theorem~\ref{thma1}.
\end{proof}

For the remaining of this section, we use the following class of
$S$ functions.

\begin{definition}[$S$ function]
A nondecreasing continuous function
$g:\mathbb{R}^+_0\rightarrow\mathbb{R}^+_0$ is said to belong to
class $S$ if it satisfies the following conditions:
\begin{enumerate}
    \item $g(x)$ is positive for $x>
    0$;
    \item $(1/z)g(x)\leq g(x/z)$ for $x\geq 0$ and $z\geq 1$.
\end{enumerate}
\end{definition}

\begin{remark}
For a brief discussion about this class of $S$ functions, the
reader is invited to consult \cite[Section~4]{Beesack}.
\end{remark}

\begin{theorem}[cf. \cite{Ferr4}]
\label{thm:nr:3.6} Let $u(t)$, $f(t)$, $k(t,s)$, $\Phi$ and $W$ be
as defined in Theorem~\ref{thma1} and assume that $g\in S$.
Suppose that $a(t)$ is a positive, rd-continuous and nondecreasing
function. If
\begin{equation}
\label{eq8} u(t)\leq a(t)+\int_a^t f(s)g(u(s))\Delta
s+\int_a^tf(s)W\left(\int_a^s
k(s,\tau)\Phi(u(\tau))\Delta\tau\right)\Delta s,
\end{equation}
for $a\leq\tau\leq s\leq t\leq b$, $\tau, s, t\in\mathbb{T}_\ast$,
then for all $t\in\mathbb{T}_\ast$ satisfying
$$G(1)+\int_a^t f(\tau) \Delta\tau\in \Dom(G^{-1})$$
and
$$\Psi(\bar{\zeta})+\int_a^{\rho(t)} k(\rho(t),\tau)\Phi(q(\tau))\Phi(\int_a^\tau
f(\theta)\Delta\theta)\Delta\tau\in \Dom(\Psi^{-1}),$$ we have
\begin{multline*}
u(t)\leq q(t) \max\{a(t),1\}\\
+ q(t) \int_a^t f(s)W\left[\Psi^{-1}
\left(\Psi(\bar{\zeta})+\int_a^s
k(s,\tau)\Phi(q(\tau))\Phi(\int_a^\tau
f(\theta)\Delta\theta)\Delta\tau\right)\right]\Delta s \, ,
\end{multline*}
where $\Psi$ is defined by (\ref{p1}),
\begin{eqnarray}
G(x)=\int_{\delta}^{x}\frac{ds}{g(s)},\ x>0,\ \delta >0, \nonumber\\
\label{seila4} q(t)=G^{-1}\left(G(1)+\int_a^t f(\tau) \Delta\tau\right),\\
\bar{\zeta}=\int_a^{\rho(b)}
k(\rho(b),s)\Phi(q(s)\max\{a(s),1\})\Delta s, \nonumber
\end{eqnarray}
and $G^{-1}$ is the inverse function of $G$.
\end{theorem}
\begin{proof}
Define the function
$$z(t)=\max\{a(t),1\}+\int_a^tf(s)W\left(\int_a^s
k(s,\tau)\Phi(u(\tau))\Delta\tau\right)\Delta s.$$ Then, from
(\ref{eq8}) we have that
\begin{equation*}
u(t)\leq z(t)+\int_a^t f(s)g(u(s))\Delta s.
\end{equation*}
Clearly, $z(t)\geq 1$ is rd-continuous and nondecreasing. Since
$g\in S$, we have
\begin{equation*}
\frac{u(t)}{z(t)}\leq 1+\int_a^t
f(s)g\left(\frac{u(s)}{z(s)}\right)\Delta s,
\end{equation*}
or
\begin{equation}
\label{eq:li} x(t)\leq 1+\int_a^t f(s)g(x(s))\Delta s,
\end{equation}
with $x(t)=u(t)/z(t)$. If we define $v(t)$ as the right-hand side
of inequality (\ref{eq:li}), we have that $v(a)=1$,
$$v^\Delta(t)=f(t)g(x(t)),$$
and since $g$ is nondecreasing,
$$v^\Delta(t)\leq f(t)g(v(t)),$$
or
\begin{equation}
\label{eq9} \frac{v^\Delta(t)}{g(v(t))}\leq f(t).
\end{equation}
Being the case that $v^\Delta(t)\geq 0$, $\Delta$-integrating
(\ref{eq9}) from $a$ to $t$ and applying Lemma~\ref{lemimp}, we
obtain
$$G(v(t))\leq G(1)+\int_a^t f(\tau)\Delta\tau,$$
which implies that
$$v(t)\leq G^{-1}\left(G(1)+\int_a^t f(\tau)\Delta\tau\right).$$
We have just proved that $x(t)\leq q(t)$, which is equivalent to
\begin{equation*}
u(t)\leq q(t)z(t).
\end{equation*}
Following the same arguments as in the proof of
Theorem~\ref{thma1}, we obtain the desired inequality.
\end{proof}

If we consider the time scale
$\mathbb{T}=h\mathbb{Z}=\{hk:k\in\mathbb{Z}\}$, where $h>0$, then
we obtain the following result.

\begin{cor}[cf. \cite{Ferr4}]
Let $a,b\in h\mathbb{Z}$,  $h>0$. Suppose that $u(t)$, $f(t)$,
$k(t,s)$, $\Phi$ and $W$ are as defined in Theorem~\ref{thma1} and
assume that $g\in S$. Suppose that $a(t)$ is a positive and
nondecreasing function. If
\begin{equation*}
u(t)\leq a(t)+\sum_{s\in[a,t)_{\mathbb{T}_\ast}}
f(s)g(u(s))h+\sum_{s\in[a,t)_{\mathbb{T}_\ast}}
f(s)W\left(\sum_{\tau\in[a,s)_{\mathbb{T}_\ast}} k(s,\tau
)\Phi(u(\tau ))h\right)h,
\end{equation*}
for $a\leq\tau\leq s\leq t\leq b$, $\tau, s,
t\in\mathbb{T}_\ast=[a,b]_{h\mathbb{Z}}$, then for all
$t\in\mathbb{T}_\ast$ satisfying
$$G(1)+\sum_{\tau\in[a,t)_{\mathbb{T}_\ast}} f(\tau)h\in \Dom(G^{-1})$$
and
$$\Psi(\bar{\zeta})+\sum_{\tau\in[a,t-h)_{\mathbb{T}_\ast}} k(t-h,\tau )\Phi(q(\tau ))\Phi\left(\sum_{\theta\in[a,\tau)_{\mathbb{T}_\ast}}
f(\theta)h\right)h\in \Dom(\Psi^{-1}),$$ we have
\begin{multline*}
u(t)\leq q(t)\Biggl\{\max\{a(t),1\}\\
+\sum_{s\in[a,t)_{\mathbb{T}_\ast}} f(s)W\left[\Psi^{-1}
\left(\Psi(\bar{\zeta})+\sum_{\tau\in[a,s)_{\mathbb{T}_\ast}}
k(s,\tau )\Phi(q(\tau
))\Phi\left(\sum_{\theta\in[a,\tau)_{\mathbb{T}_\ast}} f(\theta
)h\right)h\right)\right]h\Biggr\},
\end{multline*}
where $\Psi$ is defined by (\ref{p1}),
\begin{eqnarray*}
G(x)=\int_{\delta}^{x}\frac{ds}{g(s)},\ x>0,\ \delta >0,\\
q(t)=G^{-1}\left(G(1)+\sum_{\tau\in[a,t)_{\mathbb{T}_\ast}} f(\tau)h\right),\\
\bar{\zeta}=\sum_{s\in[a,b-h){\mathbb{T}_\ast}}
k(b-h,s)\Phi(q(s)\max\{a(s),1\})h,
\end{eqnarray*}
and $G^{-1}$ is the inverse function of $G$.
\end{cor}

\begin{theorem}[cf. \cite{Ferr4}]
Let $u(t)$, $f(t)$, $b(t)$, $h(t)$, $\Phi$ and $W$ be as defined
in Theorem~\ref{thma2} and assume that $g\in S$. Suppose that
$a(t)$ is a positive, rd-continuous and nondecreasing function. If
\begin{equation*}
u(t)\leq a(t)+\int_a^t f(s)g(u(s))\Delta
s+\int_a^tf(s)h(s)W\left(\int_a^s
b(\tau)\Phi(u(\tau))\Delta\tau\right)\Delta s,
\end{equation*}
for $a\leq\tau\leq s\leq t\leq b$, $\tau, s, t\in\mathbb{T}_\ast$,
then for all $t\in\mathbb{T}_\ast$ satisfying
$$\Psi(\bar{\xi})+\int_a^{\rho(t)} b(\tau)\Phi(q(\tau))\Phi(\int_a^\tau
f(\theta)h(\theta)\Delta\theta)\Delta\tau\in \Dom(\Psi^{-1}),$$ we
have
\begin{multline*}
u(t)\leq q(t) \max\{a(t),1\}\\
+ q(t) \int_a^t f(s)h(s)W\left[\Psi^{-1}
\left(\Psi(\bar{\xi})+\int_a^sb(\tau)\Phi(q(\tau))\Phi(\int_a^\tau
f(\theta)h(\theta)\Delta\theta)\Delta\tau\right)\right]\Delta s,
\end{multline*}
where $\Psi$ is defined by (\ref{p1}), $q(t)$ is defined by
(\ref{seila4}) and
$$\bar{\xi}=\int_a^{\rho(b)} b(s)\Phi(q(s)\max\{a(s),1\})\Delta s,$$
\end{theorem}

\begin{proof}
similar to the proof of Theorem~\ref{thm:nr:3.6}.
\end{proof}


We now make use of Theorem~\ref{thma2} to the qualitative analysis
of a nonlinear dynamic equation. Let $a,b\in\mathbb{T}$ and
consider the initial value problem
\begin{equation}
u^\Delta(t)=F\left(t,u(t),\int_a^t K(t,u(s))\Delta s\right),\quad
t\in\mathbb{T}_\ast^\kappa,\quad u(a)=u_a,\label{ivp}
\end{equation}
where $\mathbb{T}_\ast=[a,b]_{\mathbb{T}}$, $u\in$
C$_{\textrm{rd}}^1(\mathbb{T}_\ast)$, $F\in$
C$_{\textrm{rd}}(\mathbb{T}_\ast\times\mathbb{R}\times\mathbb{R},\mathbb{R})$
and $K\in$
C$_{\textrm{rd}}(\mathbb{T}_\ast\times\mathbb{R},\mathbb{R})$.

In what follows, we shall assume that the IVP (\ref{ivp}) has a
unique solution, which we denote by $u_\ast(t)$.

\begin{theorem}[cf. \cite{Ferr4}]
Assume that the functions $F$ and $K$ in (\ref{ivp}) satisfy the
conditions
\begin{align}
|K(t,u)|&\leq h(t)\Phi(|u|)\label{cond1},\\
|F(t,u,v)|&\leq |u|+|v|\label{cond2},
\end{align}
where $h$ and $\Phi$ are as defined in Theorem~\ref{thma2}. Then,
for $t\in\mathbb{T}_\ast$ such that
$$\Psi(\xi)+\int_a^{\rho(t)} \Phi(p(\tau))\Phi(\int_a^\tau
h(\theta)\Delta\theta)\Delta\tau\in \Dom(\Psi^{-1}),$$ we have the
estimate
\begin{equation}
\label{seila3} |u_\ast(t)|\leq p(t)\left\{|u_a|+\int_a^t
h(s)\Psi^{-1} \left(\Psi(\xi)+\int_a^s
\Phi(p(\tau))\Phi(\int_a^\tau
h(\theta)\Delta\theta)\Delta\tau\right)\Delta s\right\},
\end{equation}
where
\begin{eqnarray*}
p(t)=1+\int_a^t e_{1}(t,\sigma(s))\Delta s,\\
\xi=\int_a^{\rho(b)} \Phi(p(s)|u_a|)\Delta s,\\
\Psi(x)=\int_{x_0}^x\frac{1}{\Phi(s)}ds,\ x>0,\ x_0>0 \, .
\end{eqnarray*}
\end{theorem}
\begin{proof}
Let $u_\ast(t)$ be the solution of the IVP (\ref{ivp}). Then, we
have
\begin{equation}
\label{seila11111} u_\ast(t)=u_a+\int_a^t
F\left(s,u_\ast(s),\int_a^s
K(s,u_\ast(\tau))\Delta\tau)\right)\Delta s.
\end{equation}
Using (\ref{cond1}) and (\ref{cond2}) in (\ref{seila11111}), we
have
\begin{align}
|u_\ast|&\leq|u_a|+\int_a^t
\left(|u_\ast(s)|+\int_a^s|K(s,u_\ast(\tau))|\Delta\tau\right)\Delta
s\nonumber\\
&\leq|u_a|+\int_a^t
\left(|u_\ast(s)|+h(s)\int_a^s\Phi(|u_\ast(\tau))|)\Delta\tau\right)\Delta
s\label{seila2}.
\end{align}
A suitable application of Theorem~\ref{thma2} to (\ref{seila2}),
with $a(t)=|u_a|$, $f(t)=b(t)=1$ and $W(u)=u$, yields
(\ref{seila3}).
\end{proof}

\section{Some more integral inequalities and applications}\label{ineconstan}

In this section we shall be concern with the existence of
solutions of the following integrodynamic equation
\begin{equation}
\label{ip1} x^\Delta(t)= F\left(t,x(t),\int_a^{t}
K[t,\tau,x(\tau)]\Delta \tau\right),\quad x(a)=A,\quad
t\in[a,b]^\kappa_{\mathbb{T}},
\end{equation}
where $a,b\in\mathbb{T}$,
$K:[a,b]^\kappa_{\mathbb{T}}\times[a,b]^{\kappa^2}_{\mathbb{T}}\times\mathbb{R}\rightarrow\mathbb{R}$
and
$F:[a,b]^\kappa_{\mathbb{T}}\times\mathbb{R}^2\rightarrow\mathbb{R}$
are continuous functions.

As is well known integrodifferential equations and their discrete
analogues find many applications in various mathematical problems.
Moreover, it appears to be advantageous to model certain processes
by employing a suitable combination of both differential equations
and difference equations at different stages in the process under
consideration (see \cite{tisdell} and references therein for more
details).

Some results concerning existence of a solution to some particular
cases of the integrodynamic equation in \eqref{ip1} were obtained
in \cite{Tad,Xing}. Here we will make use of the well-known in the
literature \emph{Topological Transversality Theorem} to prove the
existence of a solution to the above mentioned equation.

We first introduce some basic definitions and results on fixed
point theory. Let $\mathcal{B}$ be a Banach space and
$C\subset\mathcal{B}$ be convex. By a \emph{pair} $(X,A)$ in $C$
is meant an arbitrary subset $X$ of $C$ and an $A\subset X$ closed
in $X$. We call a homotopy $H:X\times[0,1]\rightarrow Y$
\emph{compact} if it is a compact map. If $X\subset Y$, the
homotopy $H$ is called \emph{fixed point free} on $A\subset X$ if
for each $\lambda\in[0,1]$, the map
$H|A\times\{\lambda\}:A\rightarrow Y$ has no fixed point. We
denote by $\mathcal{C}_A(X,C)$ the set of all compact maps
$F:X\rightarrow C$ such that the restriction $F|A:A\rightarrow C$
is fixed point free.

Two maps $F,G\in\mathcal{C}_A(X,C)$ are called \emph{homotopic},
written $F\simeq G$ in $\mathcal{C}_A(X,C)$, provided there is a
compact homotopy $H_\lambda:X\rightarrow C$ ($\lambda\in[0,1]$)
that is fixed point free on $A$ and such that $H_0=F$ and $H_1=G$.
\begin{definition}
Let $(X,A)$ be a pair in a convex $C\subset\mathcal{B}$. A map
$F\in\mathcal{C}_A(X,C)$ is called \emph{essential} provided every
$G\in\mathcal{C}_A(X,C)$ such that $F|A=G|A$ has a fixed point.
\end{definition}
\begin{theorem}[Topological Transversality \cite{livroFPT}]\label{topo}
Let $(X,A)$ be a pair in a convex $C\subset\mathcal{B}$, and let
$F,G$ be maps in $\mathcal{C}_A(X,C)$ such that $F\simeq G$ in
$\mathcal{C}_A(X,C)$. Then, $F$ is essential if and only if $G$ is
essential.
\end{theorem}
The next theorem is very useful to the application of the
Topological Transversality Theorem. Its proof can be found in
\cite{livroFPT}.
\begin{theorem}
Let $U$ be an open subset of a convex set $C\subset\mathcal{B}$,
and let $(\bar{U},\partial U)$ be the pair consisting of the
closure of $U$ in $C$ and the boundary of $U$ in $C$. Then, for
any $u_0\in U$, the constant map $F|\bar{U}=u_0$ is essential in
$\mathcal{C}_{\partial U}(\bar{U},C)$.
\end{theorem}

If more details are needed about fixed point theory we recommend
the books \cite{BR,livroFPT}.

Here, as usual, $C([a,b]_{\mathbb{T}},\mathbb{R})$ [sometimes we
will only write $C[a,b]_{\mathbb{T}}$] denotes the set of all real
valued continuous functions on $[a,b]_{\mathbb{T}}$ and
$C^1([a,b]_{\mathbb{T}},\mathbb{R})$ the set of all
$\Delta$-differentiable functions whose derivative is continuous
on $[a,b]^\kappa_{\mathbb{T}}$, and we equip the spaces
$C([a,b]_{\mathbb{T}},\mathbb{R})$,
$C^1([a,b]_{\mathbb{T}},\mathbb{R})$ with the norms
${\|u\|}_0=\sup_{t\in[a,b]_\mathbb{T}}|u(t)|$,
${\|u\|}_1=\sup_{t\in[a,b]_\mathbb{T}}|u(t)|+\sup_{t\in[a,b]_\mathbb{T}^\kappa}|u^\Delta(t)|$,
respectively.

To make use of the Topological Transversality Theorem it is
essential to obtain a priori bounds on the possible solutions of
the equation in study. To achieve that, the next lemma is
indispensable.
\begin{lemma}\label{lem0}
Let $f,g\in C([a,b]^\kappa_\mathbb{T},\mathbb{R}_0^+)$ and $u,p\in
C([a,b]_\mathbb{T},\mathbb{R}_0^+)$, being $p(t)$ positive and
nondecreasing on $[a,b]_\mathbb{T}$, $k\in
C([a,b]^\kappa_\mathbb{T}\times[a,b]^{\kappa^2}_\mathbb{T},\mathbb{R}_0^+)$,
and $c\in\mathbb{R}_0^+$. Moreover, let $w,\tilde{w}\in
C(\mathbb{R}_0^+,\mathbb{R}_0^+)$ be nondecreasing functions with
$\{w,\tilde{w}\}(x)>0$ for $x>0$. Let $B,D>0$ and $M=(b-a)D$.
Define a function $G$ by,
\begin{equation*}
G(x)=\int_1^x\frac{1}{w(\max\{s,M\tilde{w}(s)\})}ds,\quad x>0,
\end{equation*}
and assume that $\lim_{x\rightarrow\infty}G(x)=\infty$. If, for
all $t\in[a,b]_\mathbb{T}$, the following inequality holds
\begin{equation*}
u(t)\leq p(t)+B\int_a^{t}w\left(\max\left\{u(s)+c,D\int_a^{s}
\tilde{w}(u(\tau)+c)\Delta\tau\right\}\right)\Delta s,
\end{equation*}
then,
\begin{equation*}
u(t)\leq G^{-1}\left[G(p(t)+c)+B(t-a)\right]-c,\quad
t\in[a,b]_\mathbb{T}.
\end{equation*}
\end{lemma}
\begin{proof}
We start by noting that the result trivially holds for $t=a$.
Consider an arbitrary number $t_0\in(a,b]_\mathbb{T}$ and define a
positive and nondecreasing function $z(t)$ on $[a,t_0]_\mathbb{T}$
by
\begin{equation*}
z(t)=p(t_0)+B\int_a^{t}w\left(\max\left\{u(s)+c,D\int_a^{s}
\tilde{w}(u(\tau)+c)\Delta\tau\right\}\right)\Delta s.
\end{equation*}
Then, $z(a)=p(t_0)$, for all $t\in[a,t_0]_\mathbb{T}$ we have
$u(t)\leq z(t)$, and
\begin{align}
z^\Delta(t)&=Bw\left(\max\left\{u(t)+c,D\int_a^{t}
\tilde{w}(u(\tau)+c)\Delta\tau\right\}\right)\nonumber\\
&\leq Bw\left(\max\left\{u(t)+c,D(b-a)
\tilde{w}(z(t)+c)\right\}\right)\nonumber\\
&=Bw\left(\max\left\{u(t)+c,M
\tilde{w}(z(t)+c)\right\}\right),\nonumber
\end{align}
for all $t\in[a,t_0]_\mathbb{T}^\kappa$. Hence,
\begin{equation}
\frac{z^\Delta(t)}{w\left(\max\left\{u(t)+c,M
\tilde{w}(z(t)+c)\right\}\right)}\leq B\nonumber,
\end{equation}
and, after integrating from $a$ to $t$ and with the help of Lemma
\ref{lemimp}, we get
\begin{equation}
G(z(t)+c)\leq G(z(a)+c)+B(t-a),\quad
t\in[a,t_0]_{\mathbb{T}}\nonumber.
\end{equation}
In view of the hypothesis on function $G$, we may write
\begin{equation}
z(t)\leq G^{-1}\left[G(z(a)+c)+B(t-a)\right]-c.\nonumber
\end{equation}
Therefore,
\begin{equation}
u(t)\leq G^{-1}\left[G(p(t_0)+c)+B(t-a)\right]-c,\nonumber
\end{equation}
for all $t\in[a,t_0]_{\mathbb{T}}$. Setting $t=t_0$ in the above
inequality and having in mind that $t_0$ is arbitrary, we conclude
the proof.
\end{proof}

\begin{remark}
We note that by item 1 of Theorem \ref{teorema0} there is an
inclusion of $C^1[a,b]_{\mathbb{T}}$ into $C[a,b]_{\mathbb{T}}$.
\end{remark}

\begin{theorem}\label{compl}
The embedding $j:C^1[a,b]_{\mathbb{T}}\rightarrow
C[a,b]_{\mathbb{T}}$ is completely continuous.
\end{theorem}
\begin{proof}
Let $B$ be a bounded set in $C^1[a,b]_{\mathbb{T}}$. Then, there
exists $M>0$ such that ${\|x\|}{_1}\leq M$ for all $x\in B$. By
the definition of the norm ${\|\cdot\|}{_1}$, we have that
$\sup_{t\in[a,b]_{\mathbb{T}}}|x(t)|\leq M$, hence $\|x\|_0\leq
M$, i.e., $B$ is bounded in $C[a,b]_{\mathbb{T}}$. Let now
$\varepsilon>0$ be given and put $\delta=\frac{\varepsilon}{M}$.
Then, it is easily seen that, for $x\in B$ we have
$|x^\Delta(t)|\leq M$ for all $t\in[a,b]^\kappa_\mathbb{T}$. For
arbitrary $t_1$, $t_2\in[a,b]_\mathbb{T}$ with $t_1\neq t_2$, we
use Theorem \ref{meanvaluethm} to get
$$\frac{|x(t_1)-x(t_2)|}{|t_1-t_2|}\leq M,$$
i.e., for $|t_1-t_2|<\delta$ we have $|x(t_1)-x(t_2)|<\varepsilon$
and this proves that $B$ is equicontinuous on $C[a,b]_\mathbb{T}$
(the case $t_1=t_2$ is obvious). By the Arzela--Ascoli Theorem,
$B$ is relatively compact and therefore $j$ is completely
continuous.
\end{proof}

\begin{theorem}\label{teor}
Assume that the functions $F$ and $K$ introduced in \eqref{ip1}
satisfy
\begin{align}
\label{in1}|F(t,x,y)|&\leq
Bw(\max\{|x|,|y|\})+C,\\
\label{in2}|K(t,s,x)|&\leq D\tilde{w}(|x|),
\end{align}
where $B$, $C$ and $D$ are positive constants and $w,\tilde{w}$
are defined as in Lemma \ref{lem0}.

Then, the integrodynamic equation in \eqref{ip1} has a solution
$x\in C^1([a,b]_{\mathbb{T}},\mathbb{R})$.
\end{theorem}
\begin{proof}
We start considering the following auxiliary problem,
\begin{equation}\label{seilaa1}
y^\Delta(t)=F\left(t,y(t)+A,\int_a^{t} K[t,s,y(s)+A]\Delta
s\right),\quad y(a)=0,\quad t\in[a,b]^\kappa_{\mathbb{T}},
\end{equation}
and showing that it has a solution. To do this, we first establish
a priori bounds to the (possible) solutions of the family of
problems
\begin{equation}\label{family}
y^\Delta(t)=\lambda F\left(t,y(t)+A,\int_a^{t} K[t,s,y(s)+A]\Delta
s\right),\quad y(a)=0,\quad t\in[a,b]^\kappa_{\mathbb{T}},
\end{equation}
independently of $\lambda\in[0,1]$.

Integrating both sides of the equation in \eqref{family} on
$[a,t]_{\mathbb{T}}$, we obtain,
$$y(t)=\lambda\int_a^t F\left(s,y(s)+A,\int_a^{s}
K[s,\tau,y(\tau)+A]\Delta\tau\right)\Delta s,$$ for all
$t\in[a,b]_{\mathbb{T}}$. Applying the modulus function to both
sides of the last equality we obtain, using the triangle
inequality and inequalities \eqref{in1} and \eqref{in2},
\begin{equation*}
|y(t)|\leq p(t)+B\int_a^{t}w\left(\max\left\{|y(s)+A|,D\int_a^{s}
\tilde{w}(|y(\tau)+c|)\Delta\tau\right\}\right)\Delta s,
\end{equation*}
with $p(t)=C(t-a)$. By Lemma \ref{lem0} (with $u(t)=|y(t)|$,
$c=|A|$), we deduce that (note also that
$|y(\cdot)+A|\leq|y(\cdot)|+|A|$)
\begin{equation}\label{in55555}
|y(t)|\leq G^{-1}\left[G(p(t)+|A|)+B(t-a)\right]-|A|,
\end{equation}
for all $t\in[a,b]_{\mathbb{T}}$. Denote the right-hand side of
\eqref{in55555} by $R(t)$.

Let $\mathcal{B}=C^1[a,b]_{\mathbb{T}}$ be the Banach space
equipped with the norm $\|\cdot\|_1$ and define a set $C=\{u\in
C^1[a,b]_{\mathbb{T}}:u(a)=0\}$. Observe that $C$ is a convex
subset of $\mathcal{B}$.

Define the linear operator $L:C\rightarrow C[a,b]_{\mathbb{T}}$ by
$Lu=u^\Delta$. It is clearly bijective with a continuous inverse
$L^{-1}:C[a,b]_{\mathbb{T}}\rightarrow C$. Let $M_0$ and $M_1$ be
defined by
\begin{align}
M_0&=R(b),\nonumber\\
M_1&=\sup_{t\in[a,b]_{\mathbb{T}},|y|\leq
M_0}\left|F\left(t,y,\int_a^{t} K[t,s,y]\Delta
s\right)\right|.\nonumber
\end{align}
Consider the family of maps
$\mathcal{F}_\lambda:C[a,b]_{\mathbb{T}}\rightarrow
C[a,b]_{\mathbb{T}}$, $0\leq\lambda\leq 1$, defined by
$$(\mathcal{F}_{\lambda}y)(t)=\lambda F\left(t,y(t)+A,\int_a^{t}
K[t,s,y(s)+A]\Delta s\right),$$ and the completely continuous
embedding $j:C\rightarrow C[a,b]_{\mathbb{T}}$ (it is easily seen
that the restriction here used of $j$ of Theorem \ref{compl} is
also completely continuous). Let us consider the set $Y=\{y\in
C:\|y\|_1\leq r\}$ with $r=1+M_0+M_1$. Then we can define a
homotopy $H_\lambda:Y\rightarrow C$ by
$H_\lambda=L^{-1}\mathcal{F}_\lambda j$. Since $L^{-1}$ and
$\mathcal{F}_\lambda$ are continuous and $j$ is completely
continuous, $H$ is a compact homotopy. Moreover, it is fixed point
free on the boundary of $Y$. Since $H_0$ is the zero map, it is
essential. Because $H_0\simeq H_1$, Theorem \ref{topo} implies
that $H_1$ is also essential. In particular, $H_1$ has a fixed
point which is a solution of \eqref{seilaa1}. To finish the proof,
let $y\in C^1[a,b]_{\mathbb{T}}$ be a solution of \eqref{seilaa1}
and define the function $x(t)=y(t)+A$, $t\in[a,b]_{\mathbb{T}}$.
Then, it is easily seen that $x(a)=A$ and
$$x^\Delta(t)=F\left(t,x(t),\int_a^{t}
K[t,s,x(s)]\Delta s\right),\quad t\in[a,b]^\kappa_{\mathbb{T}},$$
i.e., $x\in C^1[a,b]_{\mathbb{T}}$ is a solution of \eqref{ip1}.
\end{proof}

Theorem \ref{teor} was inspired by the work of Adrian Constantin
in \cite{contantin}. It provides a nontrivial generalization to
the one presented there, and as a consequence, Theorem \ref{teor}
seems to be new even if we let the time scale to be
$\mathbb{T}=\mathbb{R}$. Indeed, let
\[ F(t,x,y) =|y|+ \left\{ \begin{array}{ll}
(x+1)\ln(x+1) & \mbox{if $x \geq 0$};\\
0 & \mbox{if $x < 0$}.\end{array} \right. \] We will show that,
for all $f,h\in C([a,b]^\kappa_\mathbb{T},\mathbb{R}_0^+)$,
$|F(t,x,y)|>f(t)|x|+|y|+h(t)$ for some
$(t,x,y)\in[a,b]^\kappa_\mathbb{T}\times\mathbb{R}^2$ [in
\cite{contantin} the assumption on $F$ was that $|F(t,x,y)|\leq
f(t)|x|+|y|+h(t)$]. For that, suppose the contrary, i.e., assume
that there exist $f,h\in
C([a,b]^\kappa_\mathbb{T},\mathbb{R}_0^+)$ such that
$|F(t,x,y)|\leq f(t)|x|+|y|+h(t)$ for all
$(t,x,y)\in[a,b]^\kappa_\mathbb{T}\times\mathbb{R}^2$. In
particular, for arbitrary $x>0$, we have that,
$$|y|+(x+1)\ln(x+1)=|F(t,x,y)|\leq f(t)x+|y|+h(t),$$ which is
equivalent to
$$(x+1)\ln(x+1)\leq f(t)x+h(t),$$
or
\begin{equation}\label{in0}
x[\ln(x+1)-f(t)]+\ln(x+1)\leq h(t).
\end{equation}
Now, we fix $t\in[a,b]^\kappa_\mathbb{T}$ and let
$x\rightarrow\infty$ in both sides of \eqref{in0} getting a
contradiction.

Note that the function $\bar{w}(x)=(x+1)\ln(x+1)$, $x\geq 0$, is
nondecreasing and is such that
$\int_1^\infty\frac{1}{\bar{w}(s)}=\infty$. Moreover, for $F$ as
above, we have
\begin{align}
|F(t,x,y)|&\leq|y|+\bar{w}(|x|)\nonumber\\
&\leq\max\{|x|,|y|\}+\bar{w}(\max\{|x|,|y|\})\nonumber.
\end{align}
If we define $w(x)=\max\{x,\bar{w}(x)\}$, $x\geq 0$, then
$w(x)\leq x + \bar{w}(x)$ and (see \cite{constantin})
$\int_1^\infty\frac{1}{w(s)}ds=\infty$ which shows that we can
apply our result to such a function $F$.

\begin{cor}
Assume that $p\in C([a,b]^\kappa_\mathbb{T},\mathbb{R})$. Then,
the initial value problem
\begin{equation}\label{eq5}
x^\Delta(t)=p(t)x(t),\quad x(a)=A,
\end{equation}
has a unique solution $x\in C^1([a,b]_\mathbb{T},\mathbb{R})$.
\end{cor}
\begin{proof}
Define $F(t,x,y)=p(t)x$, for
$(t,x)\in[a,b]^\kappa_\mathbb{T}\times\mathbb{R}$. Then,
$|F(t,x,y)|=|p(t)||x|$ and the existence part follows by Theorem
\ref{teor}.

Suppose now that $x,y\in C^1([a,b]_\mathbb{T},\mathbb{R})$ both
satisfy \eqref{eq5}. An integration on $[a,t]_\mathbb{T}$ yields,
\begin{align}
x(t)&=A+\int_a^t p(s)x(s)\Delta s\nonumber,\\
y(t)&=A+\int_a^t p(s)y(s)\Delta s\nonumber.
\end{align}
Hence,
$$|x(t)-y(t)|\leq\int_a^t |p(s)||x(s)-y(s)|\Delta s.$$
Using Gronwall's inequality (cf. Theorem \ref{gronw}) [note that
$|p(s)|\in\mathcal{R}^+$] it follows that $|x(t)-y(t)|\leq 0$ and
finally that $x(t)=y(t)$ for all $t\in[a,b]_\mathbb{T}$. Hence,
the solution is unique.
\end{proof}
\begin{remark}
S. Hilger in his seminal work \cite{Hilger90} proved (not only but
also) the existence of a unique solution to equation \eqref{eq5}
with $p\in C_\textrm{rd}([a,b]^\kappa_\mathbb{T},\mathbb{R})$
being regressive. In this paper we are requiring $p(t)$ to be
continuous but not a regressive function.
\end{remark}
In view of the previous remark it is interesting to think of what
happens when the function $p$ is not regressive. This is shown in
the following result.
\begin{prop}
Suppose that $p\in C([a,b]^\kappa_\mathbb{T},\mathbb{R})$ is not
regressive. Then, there exists $t_0\in[a,b]^\kappa_\mathbb{T}$
such that the solution of \eqref{eq5} satisfies $x(t)=0$ for all
$t\in[\sigma(t_0),b]_\mathbb{T}$.
\end{prop}
\begin{proof}
Since $p$ is not regressive, there exists a point
$t_0\in[a,b]^\kappa_\mathbb{T}$ such that $1+\mu(t_0)p(t_0)=0$.
This immediately implies that $\mu(t_0)\neq 0$ and $p(t_0)\neq 0$.
Let $x\in C^1([a,b]_\mathbb{T},\mathbb{R})$ be the unique solution
of \eqref{eq5}. Using formula \eqref{transfor} it is easy to
derive from \eqref{eq5} that
$x^\Delta(t)[1+\mu(t)p(t)]=p(t)x(\sigma(t))$ and in particular
that $x^\Delta(t_0)[1+\mu(t_0)p(t_0)]=p(t_0)x(\sigma(t_0))$. It
follows that $x(\sigma(t_0))=0$. Now note that, since $t_0$ is
right-scattered, we must have $\sigma(t_0)\neq a$. Moreover,
$\tilde{x}(t)=0$ satisfies $x^\Delta(t)=p(t)x(t)$ for all
$t\in[\sigma(t_0),b]_\mathbb{T}$. The uniqueness of the solution
completes the proof.
\end{proof}

\subsection{Some particular integrodynamic equations} \label{appl}

We now use Theorem \ref{teor} to prove existence of solution to an
integral equation describing some physical phenomena. We emphasize
that, for each time scale we get an equation, i.e., we can
construct various models in order to (hopefully) better describe
the phenomena in question.

\begin{theorem}\label{thm11111}
Let $0,b\in\mathbb{T}$ with $b>0$ and $L\in
C^1([0,b]_{\mathbb{T}},\mathbb{R}_0^+)$. Moreover, assume that a
function $M(t,s)\in
C([0,b]_{\mathbb{T}}\times[0,b]_{\mathbb{T}}^\kappa,\mathbb{R}_0^+)$
has continuous partial $\Delta$-derivative with respect to its
first variable [denoted by $M^{\Delta_t}(t,s)$] and
$M(\sigma(t),t)$ is continuous for all
$t\in[0,b]_{\mathbb{T}}^\kappa$. Suppose that the function
$\bar{w}\in C(\mathbb{R}_0^+,\mathbb{R}_0^+)$ is nondecreasing,
such that $\bar{w}(x)>0$ for all $x>0$ and
$\int_1^\infty\frac{1}{\bar{w}(s)}ds=\infty$. Then, the integral
equation,
\begin{equation}\label{eq0}
x(t)=L(t)+\int_0^{t}M(t,s)\bar{w}(x(s))\Delta s,\quad
t\in[0,b]_{\mathbb{T}},
\end{equation}
has a solution $x\in C^1([0,b]_{\mathbb{T}},\mathbb{R}_0^+)$.
\end{theorem}
\begin{proof}
Let us define
\begin{align}
F(t,x,y)&=L^\Delta(t)+M(\sigma(t),t)\bar{w}(|x|)+y,\quad (t,x,y)\in[0,b]_{\mathbb{T}}^\kappa\times\mathbb{R}^2,\nonumber\\
K(t,s,x)&=M^{\Delta_t}(t,s)\bar{w}(|x|),\quad
(t,s,x)\in[0,b]_{\mathbb{T}}^\kappa\times[0,b]_{\mathbb{T}}^{\kappa^2}\times\mathbb{R}.\nonumber
\end{align}
Then, we have
\begin{align}
|F(t,x,y)|&\leq |L^\Delta(t)|+M(\sigma(t),t)\bar{w}(|x|)+|y|,\nonumber\\
|K(t,s,x)|&\leq |M^{\Delta_t}(t,s)|\bar{w}(|x|).\nonumber
\end{align}
Using Theorem \ref{teor} with
\begin{align}
B&>\max_{t\in[0,b]^\kappa_\mathbb{T}}M(\sigma(t),t),\nonumber\\
C&>\max_{t\in[0,b]^\kappa_\mathbb{T}}|L^\Delta(t)|,\nonumber\\
D&>\max_{(t,s)\in[0,b]^\kappa_\mathbb{T}\times[0,b]^{\kappa^2}_\mathbb{T}}|M^{\Delta_t}(t,s)|,\nonumber
\end{align}
$\tilde{w}(x)=\bar{w}(x)$ and $w(x)=\max\{x,\bar{w}(x)\}$ we
conclude that the equation
$$x^\Delta(t)=L^\Delta(t)+M(\sigma(t),t)\bar{w}(|x(t)|)+\int_0^{t}M^{\Delta_t}(t,s)\bar{w}(|x(s)|)\Delta s,$$
with initial value $x(0)=L(0)$ has a solution $x\in
C^1([0,b]_{\mathbb{T}},\mathbb{R})$. Now, an integration on
$[0,t]_{\mathbb{T}}$ gives, using Lemma \ref{teork},
$$x(t)=L(t)+\int_0^{t}M(t,s)\bar{w}(|x(s)|)\Delta s,\quad t\in[0,b]_{\mathbb{T}}.$$
By the assumptions on functions $L$, $M$ and $\bar{w}$ we conclude
that $x(t)\geq 0$ for all $t\in[0,b]_{\mathbb{T}}$, hence $x$ is a
nonnegative solution of \eqref{eq0}.
\end{proof}

\begin{cor}\label{coro0}
If in equation \eqref{eq0}, $L(t)>0$ on $[0,b]_{\mathbb{T}}$ and
$w=\bar{w}$ is continuously differentiable on $(0,\infty)$, then
the solution obtained by Theorem \ref{thm11111} is unique.
\end{cor}
\begin{proof}
Assume that there exist two positive solutions $x,y$ on
$[0,b]_{\mathbb{T}}$ to \eqref{eq0}. Then,
\begin{equation}
x(t)-y(t)=\int_0^{t}M(t,s)[w(x(s))-w(y(s))]\Delta s,\quad
t\in[0,b]_{\mathbb{T}},\nonumber
\end{equation}
hence,
\begin{equation}
|x(t)-y(t)|\leq\int_0^{t}M(t,s)|w(x(s))-w(y(s))|\Delta s,\quad
t\in[0,b]_{\mathbb{T}}.\nonumber
\end{equation}
Let us now define the following numbers:
\begin{align}
\gamma_1&=\min_{s\in[0,b]_{\mathbb{T}}}\{x(s),y(s)\},\nonumber\\
\gamma_2&=\max_{s\in[0,b]_{\mathbb{T}}}\{x(s),y(s)\},\nonumber\\
\nu&=\max_{x\in[\gamma_1,\gamma_2]}w'(x),\nonumber\\
\mu&=\max_{(t,s)\in[0,b]_{\mathbb{T}}\times[0,b]_\mathbb{T}^\kappa}M(t,s).\nonumber
\end{align}
The mean value theorem guarantees that
$$|w(x(t))-w(y(t))|\leq\nu|(x(t)-y(t))|,\quad t\in[0,b]_\mathbb{T},$$
hence,
\begin{equation}
|x(t)-y(t)|\leq\int_0^{t}\mu\nu|x(s)-y(s)|\Delta s,\quad
t\in[0,b]_\mathbb{T}.\nonumber
\end{equation}
Using Theorem \ref{gronw}, we conclude that $|x(t)-y(t)|\leq 0$ on
$[0,b]_\mathbb{T}$ and this implies that $x(t)=y(t)$ on
$[0,b]_\mathbb{T}$.
\end{proof}

Now we want to mention some particular cases of equation
\eqref{eq0}. Define on $\mathbb{R}_0^+$ the function
$\bar{w}(x)=x^r$, $0\leq r\leq 1$. Then, the equation
\begin{equation}\label{eq2}
x(t)=L(t)+\int_0^{t}M(t,s)[x(s)]^r \Delta s,\quad
t\in[0,b]_\mathbb{T},
\end{equation}
has a unique positive solution if $L(t)>0$ for all
$t\in[0,b]_\mathbb{T}$. This type of equation appears in many
applications such as nonlinear diffusion, cellular mass change
dynamics, or studies concerning the shape of liquid drops
\cite{Okr0}. Some results concerning existence and uniqueness of
solutions to \eqref{eq2} [the time scale being
$\mathbb{T}=\mathbb{R}$] were previously obtained (consult
\cite{Okr0} and references therein).

If we define $u(t)=[x(t)]^r$, it follows from \eqref{eq2} that,
\begin{equation}\label{eq3}
[u(t)]^{\frac{1}{r}}=L(t)+\int_0^{t}M(t,s)u(s)\Delta s,\quad
t\in[0,b]_\mathbb{T}.
\end{equation}
It is easily seen that $u(t)$ is the unique solution of
\eqref{eq3} and, if we let $r=\frac{1}{2}$ and $M(t,s)=K(t-s)$ for
$K\in C^1(\mathbb{R}_0^+,\mathbb{R}_0^+)$, it follows that
\begin{equation}\label{eq4}
[u(t)]^{2}=L(t)+\int_0^{t}K(t-s)u(s)\Delta s,\quad
t\in[0,b]_\mathbb{T}.
\end{equation}
This equation appears in the mathematical theory of the
infiltration of a fluid from a cylindrical reservoir into an
isotropic homogeneous porous medium \cite{Okr1,Okr2}. Some
existence and uniqueness theorems regarding solutions of
\eqref{eq4} were obtained in
\cite{contantin,const1,const2,Okr1,Okr2}.

\begin{remark}
Corollary \ref{coro0} proves uniqueness of a solution to the
integral equation \eqref{eq4} under different assumptions on the
function $L$ than those in
\cite{contantin,const1,const2,Okr1,Okr2} (obviously considering
$\mathbb{T}=\mathbb{R}$). For example, (to prove uniqueness) in
\cite{const1}, the author considered $L\in
C^1(\mathbb{R}_0^+,\mathbb{R}_0^+)$ such that $L'(0)\neq 0$.
Therefore, if we, e.g., let $L(t)=t^2+1,\ t\in[0,b]$, we see that
$L'(0)=0$, hence we cannot use the results obtained in
\cite{const1}. If we let $L(t)=t$, Corollary \ref{coro0} cannot be
applied.
\end{remark}

We close this section with an example of what the integrodynamic
equation in \eqref{ip1} could be on a particular time scale
different from $\mathbb{R}$, namely, in
$\mathbb{T}=h\mathbb{Z}=\{hk:k\in\mathbb{Z}\}$ for some $h>0$. Let
$a=A=0$ and $b=hm$ for some $m\in\mathbb{N}$. Then, on this time
scale, \eqref{ip1} becomes
\begin{equation}
\frac{x(t+h)-x(t)}{h}= F\left(t,x(t),\sum_{k=0}^{m-1}h
K[t,kh,x(kh)]\right),\quad x(0)=0,\quad
t\in[a,b]^\kappa_{h\mathbb{Z}}.\nonumber
\end{equation}

\section{$\Diamond_\alpha$-integral inequalities}\label{diamond}

In this section we propose to prove inequalities of Jensen's,
Minkowski's and Holder's type using the $\Diamond_\alpha$--
integral (cf. definition in \eqref{diamint} below).

Based on the $\Delta$ and $\nabla$ dynamic derivatives, a combined
dynamic derivative, so called $\Diamond_{\alpha}$-- derivative,
was introduced as a linear combination of $\Delta$ and $\nabla$
dynamic derivatives on time scales \cite{diam1}, i.e, for
$0\leq\alpha\leq 1$,
\begin{equation}
\label{eq:defSmp} f^{\Diamond_{\alpha}}(t)= \alpha
f^{\Delta}(t)+(1-\alpha)f^{\nabla}(t),\quad
t\in\mathbb{T}^\kappa_\kappa,
\end{equation}
provided $f^\Delta$ and $f^\nabla$ exist. The
$\Diamond_\alpha$-derivative reduces to the $\Delta$-derivative
for $\alpha =1$ and to the $\nabla$-derivative for $\alpha =0$. On
the other hand, it represents a ``weighted dynamic derivative'' on
any uniformly discrete time scale when $\alpha =\frac{1}{2}$. A
motivation for study such derivatives is for designing more
balanced adaptive algorithms on nonuniform grids with reduced
spuriosity \cite{{diam1}}.

Let $a, t \in \mathbb{T}$, and $h: \mathbb{T} \rightarrow
\mathbb{R}$. Then, the $\Diamond_\alpha$-integral of $h$ from $a$
to $t$ is defined by
\begin{equation}
\int_{a}^{t}h(\tau) \Diamond_{\alpha} \tau = \alpha
\int_{a}^{t}h(\tau) \Delta \tau +(1- \alpha) \int_{a}^{t}h(\tau)
\nabla \tau, \quad 0 \leq \alpha \leq 1,\label{diamint}
\end{equation}
provided that there exist $\Delta$ and $\nabla$ integrals of $h$
on $\mathbb{T}$. It is clear that the diamond-$\alpha$ integral of
$h$ exists when $h$ is a continuous function.  We refer the reader
to \cite{MalTorres,diam,diam1} for an account of the calculus
associated with the $\Diamond_\alpha$-derivatives and integrals.

The next lemma provides a straightforward but useful result for
what follows.

\begin{lemma}[cf. \cite{rchid}]\label{lem1}
Assume that $f$ and $g$ are continuous functions on
$[a,b]_{\mathbb{T}}$. If $f(t)\leq g(t)$ for all
$t\in[a,b]_{\mathbb{T}}$, then $\int_a^b f(t)\Diamond_\alpha
t\leq\int_a^b g(t)\Diamond_\alpha t$.
\end{lemma}

\begin{proof}
Let $f(t)$ and $g(t)$ be continuous functions on
$[a,b]_{\mathbb{T}}$. Using \eqref{desbasic} and the analogous
inequality for the $\nabla$-integral we conclude that $\int_a^b
f(t)\Delta t\leq\int_a^b g(t)\Delta t$ and $\int_a^b f(t)\nabla
t\leq\int_a^b g(t)\nabla t$. Since $\alpha\in[0,1]$ the result
follows.
\end{proof}

We now obtain Jensen's type $\Diamond_\alpha$-integral
inequalities.

\begin{theorem}[cf. \cite{rchid}]
\label{thm3} Let $\mathbb{T}$ be a time scale, $a$, $b \in
\mathbb{T}$ with $a < b$, and $c$, $d \in \mathbb{R}$. If $ g \in
C([a, b]_{\mathbb{T}}, (c, d))$ and $f \in C((c, d), \mathbb{R} )
$ is convex, then
\begin{equation}
\label{eq:JI:mr} f \left( \frac{\int_{a}^{b}g(s) \Diamond_{\alpha}
s}{b-a}\right ) \leq \frac{\int_{a}^{b}f(g(s)) \Diamond_{\alpha}
s}{b-a}.
\end{equation}
\end{theorem}

\begin{remark}
In the particular case $\alpha=1$, inequality \eqref{eq:JI:mr}
reduces to that of Theorem~\ref{theor0}. If
$\mathbb{T}=\mathbb{R}$, then Theorem~\ref{thm3} gives the
classical Jensen inequality, \textrm{i.e.}, Theorem~\ref{thm1}.
However, if $\mathbb{T}=\mathbb{Z}$ and $f(x)=-\ln(x)$, then one
gets the well-known arithmetic-mean geometric-mean inequality
(\ref{eq:am:gm:i}).
\end{remark}

\begin{proof}
Since $f$ is convex we have
\begin{equation*}
\begin{split}
f \left( \frac{\int_{a}^{b}g(s) \Diamond_{\alpha} s}{b-a}\right )&
= f \left(\frac{\alpha}{b-a} \int_{a}^{b} g(s) \Delta s+
\frac{1-\alpha}{b-a} \int_{a}^{b} g(s) \nabla s \right)\\
& \leq \alpha f \left(\frac{1}{b-a} \int_{a}^{b} g(s) \Delta s
\right)+ (1-\alpha)f\left( \frac{1}{b-a} \int_{a}^{b} g(s) \nabla
s \right).
\end{split}
\end{equation*}
Using now Jensen's inequality on time scales (cf.
Theorem~\ref{theor0}), we get
\begin{equation*}
\begin{split}
f \left( \frac{\int_{a}^{b}g(s) \Diamond_{\alpha} s}{b-a}\right )&
\leq \frac{\alpha}{b-a} \int_{a}^{b} f(g(s)) \Delta s +
\frac{1-\alpha}{b-a} \int_{a}^{b} f(g(s)) \nabla s \\
& = \frac{1}{b-a} \left( \alpha \int_{a}^{b} f(g(s)) \Delta s
+ (1-\alpha) \int_{a}^{b} f(g(s)) \nabla s \right) \\
& = \frac{1}{b-a} \int_{a}^{b}f(g(s)) \Diamond_{\alpha} s.
\end{split}
\end{equation*}
\end{proof}

Now we give an extended Jensen's inequality (cf. \cite{rchid}).

\begin{theorem}[Generalized Jensen's inequality]
\label{thm4} Let $\mathbb{T}$ be a time scale, $a$, $b \in
\mathbb{T}$ with $a < b$, $c$, $d \in \mathbb{R}$, $g \in C([a,
b]_{\mathbb{T}}, (c, d))$ and $h\in C([a, b]_{\mathbb{T}},
\mathbb{R} )$ with
$$ \int_{a}^{b} |h(s)| \Diamond_{\alpha} s > 0 \, . $$
If $f \in C((c, d), \mathbb{R}) $ is convex, then
\begin{equation}
\label{eq:GJI} f \left( \frac{\int_{a}^{b} |h(s)|g(s)
\Diamond_{\alpha} s}{\int_{a}^{b} |h(s)| \Diamond_{\alpha}s}\right
) \leq \frac{\int_{a}^{b} |h(s)|f(g(s)) \Diamond_{\alpha}
s}{\int_{a}^{b} |h(s)| \Diamond_{\alpha}s}.
\end{equation}
\end{theorem}

\begin{remark}
In the particular case $h=1$, Theorem~\ref{thm4} reduces to
Theorem~\ref{thm3}.
\end{remark}

\begin{remark}
Similar result to Theorem~\ref{thm4} holds if one changes the
condition ``$f$ is convex'' to ``$f$ is concave'', by replacing
the inequality sign ``$\leq$'' in (\ref{eq:GJI}) by ``$\geq$''.
\end{remark}

\begin{proof}
Since $f$ is convex, it follows, for example from \cite[Exercise
3.42C]{fol}, that for $t \in (c, d)$ there exists $a_{t} \in
\mathbb{R} $ such that
\begin{equation}\label{eq1}
a_{t}(x-t) \leq f(x)-f(t) \mbox{ for all } x \in (c, d).
\end{equation}
Setting $$ t= \frac{\int_{a}^{b} |h(s)|g(s) \Diamond_{\alpha}
s}{\int_{a}^{b} |h(s)| \Diamond_{\alpha}s} \, ,
$$
then using \eqref{eq1} and Lemma~\ref{lem1}, we get
\begin{equation*}
\begin{split}
\int_{a}^{b} & |h(s)|f(g(s)) \Diamond_{\alpha} s - \left (
\int_{a}^{b} |h(s)| \Diamond_{\alpha} s \right) f \left(
\frac{\int_{a}^{b} |h(s)|g(s) \Diamond_{\alpha} s}{\int_{a}^{b}
|h(s)| \Diamond_{\alpha}s}\right ) \\
 = & \int_{a}^{b} |h(s)|f(g(s)) \Diamond_{\alpha} s - \left (
\int_{a}^{b} |h(s)| \Diamond_{\alpha} s \right) f(t) =
\int_{a}^{b} |h(s)| \left (f(g(s))-f(t) \right)
\Diamond_{\alpha} s \\
\geq & a_{t} \left( \int_{a}^{b} |h(s)| (g(s)-t ) \right)
\Diamond_{\alpha} s =  a_{t} \left( \int_{a}^{b} |h(s)|g(s)
\Diamond_{\alpha}
s - t \int_{a}^{b} |h(s)| \Diamond_{\alpha} s \right)\\
\\ = & a_{t} \left( \int_{a}^{b} |h(s)|g(s) \Diamond_{\alpha}
s- \int_{a}^{b} |h(s)|g(s) \Diamond_{\alpha} s \right) = 0 \, .
\end{split}
\end{equation*}
This leads to the desired inequality.
\end{proof}
Now we sate two (well-known in the literature) corollaries of the
previous theorem.

\begin{cor}$(\mathbb{T}=\mathbb{R})$
Let $g, h: [a, b]\rightarrow \mathbb{R}$ be continuous functions
with $g([a, b]) \subseteq (c, d)$ and $\int_{a}^{b}|h(x)| dx >0$.
If $f \in C((c, d), \mathbb{R})$ is convex, then
$$
f \left( \frac{\int_{a}^{b} |h(x)|g(x) dx}{\int_{a}^{b} |h(x)|
dx}\right ) \leq \frac{\int_{a}^{b} |h(x)|f(g(x)) dx}{\int_{a}^{b}
|h(x)| dx}.
$$
\end{cor}

\begin{cor}$(\mathbb{T}=\mathbb{Z})$
\label{cor:E:SMC} Given a convex function $f$, we have for any
$x_{1}, \ldots ,x_{n} \in \mathbb{R}$ and $c_{1}, \ldots ,c_{n}
\in \mathbb{R}$ with $\sum_{k=1}^{n}|c_{k}|
>0$:
\begin{equation}
\label{eq:des:pr} f \left
(\frac{\sum_{k=1}^{n}|c_{k}|x_{k}}{\sum_{k=1}^{n}|c_{k}|}\right)
\leq \frac{\sum_{k=1}^{n}|c_{k}|f(x_{k})}{\sum_{k=1}^{n}|c_{k}|}.
\end{equation}
\end{cor}


\subsection*{Particular Cases}

\begin{itemize}
\item[(i)] Let $g(t) > 0$ on $[a, b]_{\mathbb{T}}$ and $f(t)=
t^{\beta}$ on $(0, +\infty)$. One can see that $f$ is convex on
$(0, +\infty)$ for $\beta < 0$ or $\beta >1$, and $f$ is concave
on $(0, +\infty)$ for $\beta \in (0, 1)$. Therefore,
$$
 \left( \frac{\int_{a}^{b} |h(s)|g(s) \diamondsuit_{\alpha}
s}{\int_{a}^{b} |h(s)| \diamondsuit_{\alpha}s}\right )^{\beta}
\leq \frac{\int_{a}^{b} |h(s)|g^{\beta}(s) \diamondsuit_{\alpha}
s}{\int_{a}^{b} |h(s)| \diamondsuit_{\alpha}s}, \mbox{ if } \beta
< 0 \mbox{ or } \beta >1;
$$

$$
 \left( \frac{\int_{a}^{b} |h(s)|g(s) \diamondsuit_{\alpha}
s}{\int_{a}^{b} |h(s)| \diamondsuit_{\alpha}s}\right )^{\beta}
\geq \frac{\int_{a}^{b} |h(s)|g^{\beta}(s) \diamondsuit_{\alpha}
s}{\int_{a}^{b} |h(s)| \diamondsuit_{\alpha}s}, \mbox{ if } \beta
\in (0, 1).
$$

\item[(ii)] Let $g(t) > 0$ on $[a, b]_{\mathbb{T}}$ and $f(t)=
\ln(t)$ on $(0, +\infty)$. One can also see that $f$ is concave on
$(0, +\infty)$. It follows that

$$
 \ln \left( \frac{\int_{a}^{b} |h(s)|g(s) \Diamond_{\alpha}
s}{\int_{a}^{b} |h(s)| \Diamond_{\alpha}s}\right ) \geq
\frac{\int_{a}^{b} |h(s)|\ln (g(s)) \Diamond_{\alpha}
s}{\int_{a}^{b} |h(s)| \Diamond_{\alpha}s}.
$$

\item[(iii)]Let $h=1$, then
$$
 \ln \left(\frac{\int_{a}^{b} g(s) \diamondsuit_{\alpha}
s}{b - a}\right) \geq \frac{\int_{a}^{b} \ln (g(s))
\diamondsuit_{\alpha} s}{b - a} .
$$

\item[(iv)] Let $\mathbb{T}=\mathbb{R}$, $g: [0, 1]\rightarrow (0,
\infty)$ and $h(t)=1$. Applying Theorem~\ref{thm4} with the convex
and continuous function $f=-\ln$ on $(0, \infty)$, $a=0$ and
$b=1$, we get:
$$
 \ln \int_{0}^{1} g(s)ds \geq \int_{0}^{1}\ln( g(s))ds.
$$
Therefore,
$$
\int_{0}^{1} g(s)ds \geq \exp \left(\int_{0}^{1}\ln( g(s))ds
\right).
$$
\item[(v)] Let $\mathbb{T}=\mathbb{Z}$ and $n\in\mathbb{N}$. Fix
$a=1$, $b=n+1$ and consider a function
$g:\{1,\ldots,n+1\}\rightarrow(0,\infty)$. Obviously, $f=-\ln$ is
convex and continuous on $(0,\infty)$, so we may apply Jensen's
inequality to obtain
\end{itemize}
\begin{equation*}
\begin{split}
\ln\Biggl[ & \frac{1}{n}\left(\alpha\sum_{t=1}^n
g(t)+(1-\alpha)\sum_{t=2}^{n+1}g(t)\right)\Biggr] =
\ln\left[\frac{1}{n}\int_1^{n+1}g(t)\Diamond_\alpha
t\right]\\
&\geq\frac{1}{n}\int_1^{n+1}\ln(g(t))\Diamond_\alpha t\\
&=\frac{1}{n}\left[\alpha\sum_{t=1}^n
\ln(g(t))+(1-\alpha)\sum_{t=2}^{n+1}\ln(g(t))\right]\\
&=\ln\left\{\prod_{t=1}^n
g(t)\right\}^{\frac{\alpha}{n}}+\ln\left\{\prod_{t=2}^{n+1}
g(t)\right\}^{\frac{1-\alpha}{n}} \, ,
\end{split}
\end{equation*}
and hence
$$\frac{1}{n}\left(\alpha\sum_{t=1}^n
g(t)+(1-\alpha)\sum_{t=2}^{n+1}g(t)\right)\geq\left\{\prod_{t=1}^n
g(t)\right\}^{\frac{\alpha}{n}}\left\{\prod_{t=2}^{n+1}
g(t)\right\}^{\frac{1-\alpha}{n}}.$$ When $\alpha=1$, we obtain
the well-known arithmetic-mean geometric-mean inequality:
\begin{equation}
\label{eq:am:gm:i} \frac{1}{n}\sum_{t=1}^n
g(t)\geq\left\{\prod_{t=1}^n g(t)\right\}^{\frac{1}{n}}.
\end{equation}
When $\alpha=0$, we also have $$\frac{1}{n}\sum_{t=2}^{n+1}
g(t)\geq\left\{\prod_{t=2}^{n+1} g(t)\right\}^{\frac{1}{n}}.$$

\begin{itemize}

\item[(vi)] Let $\mathbb{T}= 2^{\mathbb{N}_{0}}$ and $N \in
\mathbb{N}$. We can apply Theorem~\ref{thm4} with $a=1, b=2^{N}$
and $g: \{ 2^{k}: 0 \leq k \leq N \} \rightarrow (0, \infty)$.
Then, we get:
\end{itemize}

\begin{equation*}
\begin{split}
\ln \left \{
\frac{\int_{1}^{2^{N}}g(t)\Diamond_{\alpha}t}{2^{N}-1} \right \}&=
\ln \left \{ \alpha \frac{\int_{1}^{2^{N}}g(t)\Delta
t}{2^{N}-1}+(1-\alpha) \frac{\int_{1}^{2^{N}}g(t)\nabla
t}{2^{N}-1}
\right \}\\
&= \ln \left \{ \frac{\alpha
\sum_{k=0}^{N-1}2^{k}g(2^{k})}{2^{N}-1}
+ \frac{(1-\alpha) \sum_{k=1}^{N}2^{k}g(2^{k})}{2^{N}-1} \right \}\\
& \geq \frac{\int_{1}^{2^{N}} \ln
(g(t))\Diamond_{\alpha}t}{2^{N}-1}
\end{split}
\end{equation*}

\begin{equation*}
\begin{split}
&= \alpha \frac{\int_{1}^{2^{N}} \ln (g(t))\Delta t}{2^{N}-1}+
(1 - \alpha) \frac{\int_{1}^{2^{N}} \ln (g(t))\nabla t}{2^{N}-1}\\
&= \frac{\alpha \sum_{k=0}^{N-1}2^{k}\ln(g(2^{k})) }{2^{N}-1} +
\frac{(1-\alpha) \sum_{k=1}^{N}2^{k}\ln(g(2^{k})) }{2^{N}-1}\\
&= \frac{ \sum_{k=0}^{N-1}\ln(g(2^{k}))^{\alpha 2^{k}} }{2^{N}-1}
+
\frac{ \sum_{k=1}^{N}\ln(g(2^{k}))^{(1-\alpha)2^{k}} }{2^{N}-1}\\
&=\frac{ \ln \prod_{k=0}^{N-1}(g(2^{k}))^{\alpha 2^{k}} }{2^{N}-1}
+ \frac{ \ln(\prod_{k=1}^{N}g(2^{k}))^{(1-\alpha)2^{k}}
}{2^{N}-1}\\
&= \ln \left \{\prod_{k=0}^{N-1}(g(2^{k}))^{\alpha 2^{k}} \right
\}^{\frac{1}{2^{N}-1}} + \ln \left
\{\prod_{k=1}^{N}(g(2^{k}))^{(1-\alpha) 2^{k}} \right
\}^{\frac{1}{2^{N}-1}}\\
&= \ln \left ( \left \{\prod_{k=0}^{N-1}(g(2^{k}))^{\alpha 2^{k}}
\right \}^{\frac{1}{2^{N}-1}} \left
\{\prod_{k=1}^{N}(g(2^{k}))^{(1-\alpha) 2^{k}} \right
\}^{\frac{1}{2^{N}-1}} \right) \, .
\end{split}
\end{equation*}
We conclude that
\begin{multline*}
  \ln \left \{ \frac{\alpha
\sum_{k=0}^{N-1}2^{k}g(2^{k})+(1-\alpha)
\sum_{k=1}^{N}2^{k}g(2^{k})}{2^{N}-1} \right \}\\
\geq \ln \left ( \left \{\prod_{k=0}^{N-1}(g(2^{k}))^{\alpha
2^{k}} \right \}^{\frac{1}{2^{N}-1}} \left
\{\prod_{k=1}^{N}(g(2^{k}))^{(1-\alpha) 2^{k}} \right
\}^{\frac{1}{2^{N}-1}} \right).
\end{multline*}
On the other hand,
$$
\alpha \sum_{k=0}^{N-1}2^{k}g(2^{k})+(1-\alpha)
\sum_{k=1}^{N}2^{k}g(2^{k})= \sum_{k=1}^{N-1}2^{k}g(2^{k})+\alpha
g(1)+(1-\alpha) 2^{N}g(2^{N}).
$$
It follows that
\begin{multline*}
\frac{\sum_{k=1}^{N-1}2^{k}g(2^{k})+\alpha g(1)+(1-\alpha)
2^{N}g(2^{N})}{2^{N}-1} \\
\geq \left\{\prod_{k=0}^{N-1}(g(2^{k}))^{\alpha 2^{k}} \right
\}^{\frac{1}{2^{N}-1}} \left
\{\prod_{k=1}^{N}(g(2^{k}))^{(1-\alpha) 2^{k}} \right
\}^{\frac{1}{2^{N}-1}}.
\end{multline*}
In the particular case when $\alpha =1$ we have
$$
\frac{\sum_{k=0}^{N-1}2^{k}g(2^{k})}{2^{N}-1} \geq \left
\{\prod_{k=0}^{N-1}(g(2^{k}))^{ 2^{k}} \right
\}^{\frac{1}{2^{N}-1}},
$$
and when $\alpha =0$ we get the inequality
$$
\frac{\sum_{k=1}^{N}2^{k}g(2^{k})}{2^{N}-1} \geq \left
\{\prod_{k=1}^{N}(g(2^{k}))^{ 2^{k}} \right
\}^{\frac{1}{2^{N}-1}}\, .
$$


\subsection{Related $\Diamond_\alpha$-integral inequalities}
\label{sec:app}

The usual proof of H\"{o}lder's inequality use the basic Young
inequality $x^{\frac{1}{p}}y^{\frac{1}{q}} \leq \frac{x}{p}+
\frac{y}{q}$ for nonnegative $x$ and $y$. Here we present a proof
based on the application of Jensen's inequality (cf.
Theorem~\ref{thm4}).

\begin{theorem}[H\"{o}lder's inequality]
\label{app:th:hi} Let $\mathbb{T}$ be a time scale, $a$, $b \in
\mathbb{T}$ with $a < b$, and $f$, $g$, $h \in C([a,
b]_{\mathbb{T}}, [0, \infty))$ with
$\int_{a}^{b}h(x)g^{q}(x)\Diamond_{\alpha} x
>0$, where $q$ is the H\"{o}lder conjugate number of $p$,
\textrm{i.e.} $\frac{1}{p}+\frac{1}{q}=1$ with $1<p$. Then, we
have:
\begin{equation}
\label{app:eq:hi} \int_{a}^{b}h(x)f(x)g(x)\Diamond_{\alpha} x \leq
\left(\int_{a}^{b}h(x)f^{p}(x)\Diamond_{\alpha}
x\right)^{\frac{1}{p}}
\left(\int_{a}^{b}h(x)g^{q}(x)\Diamond_{\alpha}
x\right)^{\frac{1}{q}} \, .
\end{equation}
\end{theorem}

\begin{proof}
Choosing $f(x)=x^{p}$ in Theorem~\ref{thm4}, which for $p>1$ is
obviously a convex function on $[0, \infty)$, we have
\begin{equation}
\label{eq:R} \left( \frac{\int_{a}^{b} |h(s)|g(s)
\Diamond_{\alpha} s}{\int_{a}^{b} |h(s)| \Diamond_{\alpha}s}\right
)^p \leq \frac{\int_{a}^{b} |h(s)|(g(s))^p \Diamond_{\alpha}
s}{\int_{a}^{b} |h(s)| \Diamond_{\alpha}s}.
\end{equation}
Inequality \eqref{app:eq:hi} is trivially true in the case when
$g$ is identically zero. We consider two cases: (i) $g(x) > 0$ for
all $x \in [a, b]_{\mathbb{T}}$; (ii) there exists at least one $t
\in [a, b]_{\mathbb{T}}$ such that $g(x) = 0$. We begin with
situation (i). Replacing $g$ by $fg^{\frac{-q}{p}}$ and $|h(x)|$
by $hg^{q}$ in inequality (\ref{eq:R}), we get:
$$
\left(
\frac{\int_{a}^{b}h(x)g^{q}(x)f(x)g^{\frac{-q}{p}}(x)\Diamond_{\alpha}
x}{\int_{a}^{b}h(x)g^{q}(x)\Diamond_{\alpha}x} \right)^{p} \leq
\frac{\int_{a}^{b}h(x)g^{q}(x)(f(x)g^{\frac{-q}{p}}(x))^{p}\Diamond_{\alpha}
x}{\int_{a}^{b}h(x)g^{q}(x)\Diamond_{\alpha}x}.
$$
Using the fact that $\frac{1}{p}+\frac{1}{q}=1$, we obtain that
\begin{equation}
\label{eq:parti} \int_{a}^{b}h(x)f(x)g(x)\Diamond_{\alpha} x \leq
\left(\int_{a}^{b}h(x)f^{p}(x)\Diamond_{\alpha}
x\right)^{\frac{1}{p}}
\left(\int_{a}^{b}h(x)g^{q}(x)\Diamond_{\alpha}
x\right)^{\frac{1}{q}} \, .
\end{equation}
We now consider situation (ii). Let $G= \left\{x \in [a,
b]_{\mathbb{T}} \, | \, g(x) =0 \right\}$. Then,
\begin{gather*}
\int_{a}^{b}h(x) f(x) g(x) \Diamond_{\alpha} x = \int_{[a,
b]_{\mathbb{T}}-G} h(x) f(x) g(x) \Diamond_{\alpha} x +
\int_{G} h(x) f(x) g(x) \Diamond_{\alpha} x\\
= \int_{[a, b]_{\mathbb{T}}-G} h(x) f(x) g(x) \Diamond_{\alpha} x
\end{gather*}
because $\int_{G} h(x) f(x) g(x) \Diamond_{\alpha} x =0$. For the
set $[a, b]_{\mathbb{T}}-G$ we are in case (i), \textrm{i.e.}
$g(x) > 0$, and it follows from \eqref{eq:parti} that
\begin{equation*}
\begin{split}
\int_{a}^{b}h(x) f(x) g(x) \Diamond_{\alpha} x &= \int_{[a,
b]_{\mathbb{T}}-G}h(x) f(x) g(x) \Diamond_{\alpha} x \\
&\leq \left(\int_{[a, b]_{\mathbb{T}}-G} h(x) f^{p}(x)
\Diamond_{\alpha} x \right )^{\frac{1}{p}} \quad \left (\int_{[a,
b]_{\mathbb{T}}-G} h(x)
g^{q}(x) \Diamond_{\alpha} x \right )^{\frac{1}{q}}\\
&\leq \left (\int_a^b h(x) f^{p}(x) \Diamond_{\alpha} x \right
)^{\frac{1}{p}} \quad \left (\int_a^b h(x) g^{q}(x)
\Diamond_{\alpha} x \right )^{\frac{1}{q}} \, .
\end{split}
\end{equation*}
\end{proof}

\begin{remark}
In the particular case $h=1$, Theorem~\ref{app:th:hi} gives the
$\Diamond_\alpha$ version of classical H\"{o}lder's inequality:
$$
\int_{a}^{b}|f(x)g(x)|\Diamond_{\alpha} x \leq
\left(\int_{a}^{b}|f|^{p}(x)\Diamond_{\alpha}
x\right)^{\frac{1}{p}}
\left(\int_{a}^{b}|g|^{q}(x)\Diamond_{\alpha}
x\right)^{\frac{1}{q}},
$$
where $p>1$ and $q=\frac{p}{p-1}$.
\end{remark}

\begin{remark}
In the special case $p=q=2$, (\ref{app:eq:hi}) reduces to the
following $\Diamond_\alpha$ Cauchy--Schwarz integral inequality on
time scales:
$$
\int_{a}^{b}|f(x)g(x)|\Diamond_{\alpha} x \leq \sqrt{
\left(\int_{a}^{b}f^{2}(x)\Diamond_{\alpha} x\right)
 \left(\int_{a}^{b}g^{2}(x)\Diamond_{\alpha}
x\ \right)} \, .
$$
\end{remark}

We are now in position to prove a Minkowski inequality using
H\"{o}lder's inequality (\ref{app:eq:hi}).

\begin{theorem}[Minkowski's inequality]
Let $\mathbb{T}$ be a time scale, $a$, $b \in \mathbb{T}$ with $a
< b$, and $p>1$. For continuous functions $f, g: [a,
b]_{\mathbb{T}} \rightarrow \mathbb{R}$ we have
$$
\left (\int_{a}^{b}|(f+g)(x)|^{p}\Diamond_{\alpha} x
\right)^{\frac{1}{p}} \leq
\left(\int_{a}^{b}|f(x)|^{p}\Diamond_{\alpha}
x\right)^{\frac{1}{p}} +
\left(\int_{a}^{b}|g(x)|^{p}\Diamond_{\alpha}
x\right)^{\frac{1}{p}}.
$$
\end{theorem}

\begin{proof}
We have, by the triangle inequality, that
\begin{multline}
\label{mink1} \int_a^b |f(x) + g(x)|^p\Diamond_{\alpha} x
=\int_a^b |f(x)+g(x)|^{p-1}|f(x)+g(x)|\Diamond_\alpha x\\
\leq \int_a^b|f(x)||f(x)+g(x)|^{p-1}\Diamond_\alpha
x+\int_a^b|g(x)||f(x)+g(x)|^{p-1}\Diamond_\alpha x.
\end{multline}
Applying now H\"{o}lder's inequality with $q=p/(p-1)$ to
\eqref{mink1}, we obtain:
\begin{multline*}
\int_a^b |f(x)+g(x)|^p\Diamond_{\alpha} x
\leq\left[\int_a^b|f(x)|^p\Diamond_\alpha
x\right]^{\frac{1}{p}}\left[\int_a^b
|f(x)+g(x)|^{(p-1)q}\Diamond_\alpha
x\right]^{\frac{1}{q}}\\
 +\left[\int_a^b|g(x)|^p\Diamond_\alpha
x\right]^{\frac{1}{p}}\left[\int_a^b
|f(x)+g(x)|^{(p-1)q}\Diamond_\alpha
x\right]^{\frac{1}{q}}\\
=\left\{\left[\int_a^b|f(x)|^p\Diamond_\alpha
x\right]^{\frac{1}{p}}+\left[\int_a^b|g(x)|^p\Diamond_\alpha
x\right]^{\frac{1}{p}}\right\}\left[\int_a^b|f(x)+g(x)|^p\Diamond_\alpha
x\right]^{\frac{1}{q}}.
\end{multline*}
Dividing both sides of the last inequality by
$$\left[\int_a^b|f(x)+g(x)|^p\Diamond_\alpha
x\right]^{\frac{1}{q}},$$ we get the desired conclusion.
\end{proof}

\section{Some integral inequalities with two independent
variables}\label{duasvar}

We establish some linear and nonlinear integral inequalities of
Gronwall--Bellman--Bihari type for functions with two independent
variables. Some particular time scales are considered as examples
and we use one of our results to estimate solutions of a partial
delta dynamic equation (cf. \eqref{diff1} below).

As was mentioned earlier, several integral inequalities of
Gronwall-Bellman-Bihari type on time scales but for functions of a
single variable were obtained in the papers \cite{inesurvey},
\cite{Pach}, \cite{Bihary} and \cite{Gronwall}. However, to the
best of the author's knowledge, no such a paper exists on the
literature when functions of two independent variables are
considered. Therefore we aim to give a first insight on such a
type of inequalities.


\subsection{Linear inequalities} \label{sec:mainResults2}

Throughout we let
$\mathbb{\tilde{T}}_1=[a_1,\infty)_{\mathbb{T}_1}$ and
$\mathbb{\tilde{T}}_2=[a_2,\infty)_{\mathbb{T}_2}$, for
$a_1\in\mathbb{T}_1$, $a_2\in\mathbb{T}_2$ being $\mathbb{T}_1$
and $\mathbb{T}_2$ time scales.

\begin{theorem}[cf. \cite{Ferr6}]
\label{teoo1} Let $u(t_1,t_2), a(t_1,t_2), f(t_1,t_2)\in
C(\tilde{\mathbb{T}}_1\times\tilde{\mathbb{T}}_2,\mathbb{R}_0^+)$
with $a(t_1,t_2)$ nondecreasing in each of the variables. If
\begin{equation}
\label{in010} u(t_1,t_2)\leq
a(t_1,t_2)+\int_{a_1}^{t_1}\int_{a_2}^{t_2}
f(s_1,s_2)u(s_1,s_2)\Delta_1 s_1\Delta_2 s_2,
\end{equation}
for $(t_1,t_2)\in \tilde{\mathbb{T}}_1\times\tilde{\mathbb{T}}_2$,
then
\begin{equation}
\label{in020} u(t_1,t_2)\leq
a(t_1,t_2)e_{\int_{a_2}^{t_2}f(t_1,s_2)\Delta_2
s_2}(t_1,a_1),\quad (t_1,t_2)\in
\tilde{\mathbb{T}}_1\times\tilde{\mathbb{T}}_2.
\end{equation}
\end{theorem}

\begin{proof}
Since $a(t_1,t_2)$ is nondecreasing on
$(t_1,t_2)\in\tilde{\mathbb{T}}_1\times\tilde{\mathbb{T}}_2$,
inequality (\ref{in010}) implies that, for an arbitrary
$\varepsilon>0$,
$$r(t_1,t_2)\leq 1+\int_{a_1}^{t_1}\int_{a_2}^{t_2}
f(s_1,s_2)r(s_1,s_2)\Delta_1 s_1\Delta_2 s_2,$$ where
$r(t_1,t_2)=\frac{u(t_1,t_2)}{a(t_1,t_2)+\varepsilon}$. Define
$v(t_1,t_2)$ by the right-hand side of the last inequality. Then
\begin{equation}
\label{in1110} \frac{\partial}{\Delta_2 t_2}\left(\frac{\partial
v(t_1,t_2)}{\Delta_1 t_1}\right)=f(t_1,t_2)r(t_1,t_2)\leq
f(t_1,t_2)v(t_1,t_2),\quad
(t_1,t_2)\in\tilde{\mathbb{T}}_1^k\times\tilde{\mathbb{T}}_2^k.
\end{equation}
From \eqref{in1110} and taking into account that $v(t_1,t_2)$ is
positive and nondecreasing, we obtain
$$\frac{v(t_1,t_2)\frac{\partial}{\Delta_2 t_2}\left(\frac{\partial
v(t_1,t_2)}{\Delta_1
t_1}\right)}{v(t_1,t_2)v(t_1,\sigma_2(t_2))}\leq f(t_1,t_2),$$
from which it follows that
$$\frac{v(t_1,t_2)\frac{\partial}{\Delta_2 t_2}\left(\frac{\partial
v(t_1,t_2)}{\Delta_1
t_1}\right)}{v(t_1,t_2)v(t_1,\sigma_2(t_2))}\leq
f(t_1,t_2)+\frac{\frac{\partial v(t_1,t_2)}{\Delta_1
t_1}\frac{\partial v(t_1,t_2)}{\Delta_2
t_2}}{v(t_1,t_2)v(t_1,\sigma_2(t_2))}.$$ The previous inequality
can be rewritten as
$$\frac{\partial}{\Delta_2 t_2}\left(\frac{\frac{\partial v(t_1,t_2)}{\Delta_1 t_1}}{v(t_1,t_2)}\right)\leq f(t_1,t_2).$$
Integrating with respect to the second variable from $a_2$ to
$t_2$ and noting that $\frac{\partial v(t_1,t_2)}{\Delta_1
t_1}\mid_{(t_1,a_2)}=0$, we have
$$\frac{\frac{\partial v(t_1,t_2)}{\Delta_1 t_1}}{v(t_1,t_2)}\leq\int_{a_2}^{t_2}f(t_1,s_2)\Delta_2 s_2,$$
that is,
$$\frac{\partial v(t_1,t_2)}{\Delta_1 t_1}\leq\int_{a_2}^{t_2}f(t_1,s_2)\Delta_2 s_2 v(t_1,t_2).$$
Fixing $t_2\in\tilde{\mathbb{T}}_2$ arbitrarily, we have that
$p(t_1):=\int_{a_2}^{t_2}f(t_1,s_2)\Delta_2 s_2\in \mathcal{R}^+$
and by Theorem \ref{teo3}
$$v(t_1,t_2)\leq e_p(t_1,a_1),$$
since $v(a_1,t_2)=1$. Now, inequality (\ref{in020}) follows from
the inequality
$$u(t_1,t_2)\leq [a(t_1,t_2)+\varepsilon]v(t_1,t_2),$$
and the arbitrariness of $\varepsilon$.
\end{proof}

\begin{cor}[cf. Lemma 2.1 of \cite{khellaf}]
\label{corolario1} Let $\mathbb{T}_1=\mathbb{T}_2=\mathbb{R}$ and
assume that the functions $u(x,y), a(x,y), f(x,y)\in
C([x_0,\infty)\times[y_0,\infty),\mathbb{R}_0^+)$ with $a(x,y)$
nondecreasing in its variables. If
\begin{equation}
u(x,y)\leq a(x,y)+\int_{x_0}^{x}\int_{y_0}^{y} f(t,s)u(t,s)dt ds,
\end{equation}
for $(x,y)\in [x_0,\infty)\times[y_0,\infty)$, then
\begin{equation}
u(x,y)\leq a(x,y)\exp\left(\int_{x_0}^{x}\int_{y_0}^y f(t,s)dt
ds\right),
\end{equation}
for $(x,y)\in [x_0,\infty)\times[y_0,\infty)$.
\end{cor}

\begin{cor}[cf. Theorem 2.1 of \cite{saldisc}]
Let $\mathbb{T}_1=\mathbb{T}_2=\mathbb{Z}$ and assume that the
functions $u(m,n), a(m,n), f(m,n)$ are nonnegative and that
$a(m,n)$ is nondecreasing for $m\in[m_0,\infty)\cap\mathbb{Z}$ and
$n\in[n_0,\infty)\cap\mathbb{Z}$, $m_0,n_0\in\mathbb{Z}$. If
\begin{equation}
u(m,n)\leq a(m,n)+\sum_{s=m_0}^{m-1}\sum_{t=n_0}^{n-1}
f(s,t)u(s,t),
\end{equation}
for all $(m,n)\in
[m_0,\infty)\cap\mathbb{Z}\times[n_0,\infty)\cap\mathbb{Z}$, then
\begin{equation}
u(m,n)\leq
a(m,n)\prod_{s=m_0}^{m-1}\left[1+\sum_{t=n_0}^{n-1}f(s,t)\right],
\end{equation}
for all $(m,n)\in
[m_0,\infty)\cap\mathbb{Z}\times[n_0,\infty)\cap\mathbb{Z}$.
\end{cor}

\begin{remark}
We note that, following the same steps of the proof of Theorem
\ref{teoo1}, it can be obtained other bound on the function $u$,
namely
\begin{equation}
\label{in6} u(t_1,t_2)\leq
a(t_1,t_2)e_{\int_{a_1}^{t_1}f(s_1,t_2)\Delta_1 s_1}(t_2,a_2).
\end{equation}
When $\mathbb{T}_1=\mathbb{T}_2=\mathbb{R}$, then the bounds in
(\ref{in6}) and (\ref{in020}) coincide (see Corollary
\ref{corolario1}). But if, for example, we let
$\mathbb{T}_1=\mathbb{T}_2=\mathbb{Z}$, the bounds obtained can be
different; moreover, at different points one can be sharper than
the other and vice-versa. Note up the following example:

Let $f(t_1,t_2)$ be a function defined by $f(0,0)=1/4$,
$f(1,0)=1/5$, $f(2,0)=1$, $f(0,1)=1/2$, $f(1,1)=0$ and $f(2,1)=5$.
Set $a_1=a_2=0$.

Then, from (\ref{in020}) we get
\begin{align}
u(2,1)&\leq a(2,1)\frac{3}{2},\nonumber\\
u(3,2)&\leq a(3,2)\frac{147}{10},\nonumber
\end{align}
while from (\ref{in6}) we get
\begin{align}
u(2,1)&\leq a(2,1)\frac{29}{20},\nonumber\\
u(3,2)&\leq a(3,2)\frac{637}{40}.\nonumber
\end{align}
\end{remark}

We now present the particular case of Theorem \ref{teoo1} when
$\mathbb{T}_1=\mathbb{Z}$ and $\mathbb{T}_2=\mathbb{R}$.

\begin{cor}[cf. \cite{Ferr6}]
Let $\mathbb{T}_1=\mathbb{Z}$ and $\mathbb{T}_2=\mathbb{R}$.
Assume that the functions $u(t,x),\ a(t,x)$ and $f(t,x)$ satisfy
the hypothesis of Theorem \ref{teoo1} for all
$(t,x)\in\tilde{\mathbb{T}}_1\times\tilde{\mathbb{T}}_2$, with
$a_1=a_2=0$. If
\begin{equation}
u(t,x)\leq a(t,x)+\sum_{k=0}^{t-1}\int_{0}^{x}
f(k,\tau)u(k,\tau)d\tau,
\end{equation}
for all $(t,x)\in\tilde{\mathbb{T}}_1\times\tilde{\mathbb{T}}_2$,
then
\begin{equation}
u(t,x)\leq
a(t,x)\prod_{k=0}^{t-1}\left[1+\int_{0}^{x}f(k,\tau)d\tau\right],
\end{equation}
for all $(t,x)\in\tilde{\mathbb{T}}_1\times\tilde{\mathbb{T}}_2$.
\end{cor}

We now generalize Theorem \ref{teoo1}. If in Theorem
\ref{teorema2} $f$ does not depend on the first two variables,
then we obtain Theorem \ref{teoo1}.

\begin{theorem}[cf. \cite{Ferr6}]\label{teorema2}
Let $u(t_1,t_2), a(t_1,t_2)\in
C(\tilde{\mathbb{T}}_1\times\tilde{\mathbb{T}}_2,\mathbb{R}_0^+)$,
with $a(t_1,t_2)$ nondecreasing in each of the variables and
$f(t_1,t_2,s_1,s_2)\in C(S,\mathbb{R}_0^+)$ be nondecreasing in
$t_1$ and $t_2$, where
$S=\{(t_1,t_2,s_1,s_2)\in\tilde{\mathbb{T}}_1\times\tilde{\mathbb{T}}_2\times\tilde{\mathbb{T}}_1\times\tilde{\mathbb{T}}_2:a_1\leq
s_1\leq t_1,a_2\leq s_2\leq t_2\}$. If
\begin{equation}
u(t_1,t_2)\leq a(t_1,t_2)+\int_{a_1}^{t_1}\int_{a_2}^{t_2}
f(t_1,t_2,s_1,s_2)u(s_1,s_2)\Delta_1 s_1\Delta_2 s_2,
\end{equation}
for $(t_1,t_2)\in \tilde{\mathbb{T}}_1\times\tilde{\mathbb{T}}_2$,
then
\begin{equation}
\label{in5} u(t_1,t_2)\leq
a(t_1,t_2)e_{\int_{a_2}^{t_2}f(t_1,t_2,t_1,s_2)\Delta_2
s_2}(t_1,a_1),\ (t_1,t_2)\in
\tilde{\mathbb{T}}_1\times\tilde{\mathbb{T}}_2.
\end{equation}
\end{theorem}

\begin{proof}
We start by fixing arbitrary numbers
$t_1^\ast\in\tilde{\mathbb{T}}_1$ and
$t_2^\ast\in\tilde{\mathbb{T}}_2$ and consider the function
defined on
$[a_1,t_1^\ast]\cap\tilde{\mathbb{T}}_1\times[a_2,t_2^\ast]\cap\tilde{\mathbb{T}}_2$
by
\begin{equation}
v(t_1,t_2)=a(t_1^\ast,t_2^\ast)+\varepsilon+\int_{a_1}^{t_1}\int_{a_2}^{t_2}
f(t_1^\ast,t_2^\ast,s_1,s_2)u(s_1,s_2)\Delta_1 s_1\Delta_2 s_2,
\end{equation}
for an arbitrary $\varepsilon>0$. From our hypothesis we see that
\begin{equation}
u(t_1,t_2)\leq v(t_1,t_2),\ \mbox{for all}\ (t_1,t_2)\in
[a_1,t_1^\ast]\cap\tilde{\mathbb{T}}_1\times[a_2,t_2^\ast]\cap\tilde{\mathbb{T}}_2.
\end{equation}
Moreover, $\Delta$-differentiating with respect to the first
variable and then with respect to the second, we obtain
\begin{align}
\frac{\partial}{\Delta_2 t_2}\left(\frac{\partial
v(t_1,t_2)}{\Delta_1
t_1}\right)&=f(t_1^\ast,t_2^\ast,t_1,t_2)u(t_1,t_2)\nonumber\\
&\leq f(t_1^\ast,t_2^\ast,t_1,t_2)v(t_1,t_2),\nonumber
\end{align}
for all $(t_1,t_2)\in
[a_1,t_1^\ast]^\kappa\cap\tilde{\mathbb{T}}_1\times[a_2,t_2^\ast]^\kappa\cap\tilde{\mathbb{T}}_2$.
From this last inequality, we can write
$$\frac{v(t_1,t_2)\frac{\partial}{\Delta_2 t_2}\left(\frac{\partial
v(t_1,t_2)}{\Delta_1
t_1}\right)}{v(t_1,t_2)v(t_1,\sigma_2(t_2))}\leq
f(t_1^\ast,t_2^\ast,t_1,t_2),$$ hence,
$$\frac{v(t_1,t_2)\frac{\partial}{\Delta_2 t_2}\left(\frac{\partial
v(t_1,t_2)}{\Delta_1
t_1}\right)}{v(t_1,t_2)v(t_1,\sigma_2(t_2))}\leq
f(t_1^\ast,t_2^\ast,t_1,t_2)+\frac{\frac{\partial
v(t_1,t_2)}{\Delta_1 t_1}\frac{\partial v(t_1,t_2)}{\Delta_2
t_2}}{v(t_1,t_2)v(t_1,\sigma_2(t_2))}.$$ The previous inequality
can be rewritten as
$$\frac{\partial}{\Delta_2 t_2}\left(\frac{\frac{\partial v(t_1,t_2)}{\Delta_1 t_1}}{v(t_1,t_2)}\right)\leq f(t_1^\ast,t_2^\ast,t_1,t_2).$$
$\Delta$-integrating with respect to the second variable from
$a_2$ to $t_2$ and noting that $\frac{\partial
v(t_1,t_2)}{\Delta_1 t_1}\mid_{(t_1,a_2)}=0$, we have
$$\frac{\frac{\partial v(t_1,t_2)}{\Delta_1 t_1}}{v(t_1,t_2)}\leq\int_{a_2}^{t_2}f(t_1^\ast,t_2^\ast,t_1,s_2)\Delta_2 s_2,$$
that is,
$$\frac{\partial v(t_1,t_2)}{\Delta_1 t_1}\leq\int_{a_2}^{t_2}f(t_1^\ast,t_2^\ast,t_1,s_2)\Delta_2 s_2 v(t_1,t_2).$$
Fix $t_2=t_2^\ast$ and put
$p(t_1):=\int_{a_2}^{t_2^\ast}f(t_1^\ast,t_2^\ast,t_1,s_2)\Delta_2
s_2\in \mathcal{R}^+$. By Theorem \ref{teo3}
$$v(t_1,t_2^\ast)\leq(a(t_1^\ast,t_2^\ast)+\varepsilon) e_p(t_1,a_1).$$
Letting $t_1=t_1^\ast$ in the above inequality, and remembering
that $t_1^\ast$, $t_2^\ast$ and $\varepsilon$ are arbitrary,
inequality (\ref{in5}) follows.
\end{proof}

\subsection{Nonlinear inequalities}

\begin{theorem}[cf. \cite{Ferr6}]
\label{teo2} Let $u(t_1,t_2), f(t_1,t_2)\in
C(\tilde{\mathbb{T}}_1\times\tilde{\mathbb{T}}_2,\mathbb{R}_0^+)$.
Moreover, let $a(t_1,t_2)\in
C(\tilde{\mathbb{T}}_1\times\tilde{\mathbb{T}}_2,\mathbb{R}^+)$
and be a nondecreasing function in each of the variables. If $p$
and $q$ are two positive real numbers such that $p\geq q$ and if
\begin{equation}
\label{in3} u^p(t_1,t_2)\leq
a(t_1,t_2)+\int_{a_1}^{t_1}\int_{a_2}^{t_2}
f(s_1,s_2)u^q(s_1,s_2)\Delta_1 s_1\Delta_2 s_2,
\end{equation}
for $(t_1,t_2)\in \tilde{\mathbb{T}}_1\times\tilde{\mathbb{T}}_2$,
then
\begin{equation}
\label{in4} u(t_1,t_2)\leq
a^{\frac{1}{p}}(t_1,t_2)\left[e_{\int_{a_2}^{t_2}f(t_1,s_2)a^{\frac{q}{p}-1}(t_1,s_2)\Delta_2
s_2}(t_1,a_1)\right]^{\frac{1}{p}},\quad (t_1,t_2)\in
\tilde{\mathbb{T}}_1\times\tilde{\mathbb{T}}_2.
\end{equation}
\end{theorem}

\begin{proof}
Since $a(t_1,t_2)$ is positive and nondecreasing on
$(t_1,t_2)\in\tilde{\mathbb{T}}_1\times\tilde{\mathbb{T}}_2$,
inequality (\ref{in3}) implies that,
$$u^p(t_1,t_2)\leq a(t_1,t_2)\left(1+\int_{a_1}^{t_1}\int_{a_2}^{t_2}
f(s_1,s_2)\frac{u^q(s_1,s_2)}{a(s_1,s_2)}\Delta_1 s_1\Delta_2
s_2\right).$$  Define $v(t_1,t_2)$ on
$\tilde{\mathbb{T}}_1\times\tilde{\mathbb{T}}_2$ by

$$v(t_1,t_2)=1+\int_{a_1}^{t_1}\int_{a_2}^{t_2}
f(s_1,s_2)\frac{u^q(s_1,s_2)}{a(s_1,s_2)}\Delta_1 s_1\Delta_2
s_2.$$ Then,
\begin{equation}
\frac{\partial}{\Delta_2 t_2}\left(\frac{\partial
v(t_1,t_2)}{\Delta_1
t_1}\right)=f(t_1,t_2)\frac{u^q(t_1,t_2)}{a(t_1,t_2)}\leq
f(t_1,t_2)a^{\frac{q}{p}-1}(t_1,t_2)v^{\frac{q}{p}}(t_1,t_2).\nonumber
\end{equation}
Now, note that $v^{\frac{q}{p}}(t_1,t_2)\leq v(t_1,t_2)$ and
therefore
$$\frac{\partial}{\Delta_2 t_2}\left(\frac{\partial
v(t_1,t_2)}{\Delta_1 t_1}\right)\leq
f(t_1,t_2)a^{\frac{q}{p}-1}(t_1,t_2)v(t_1,t_2).$$ From here, we
can follow the same procedure as in the proof of Theorem
\ref{teoo1} to obtain
$$v(t_1,t_2)\leq e_p(t_1,a_1),$$
where
$p(t_1)=\int_{a_2}^{t_2}f(t_1,s_2)a^{\frac{q}{p}-1}(t_1,s_2)\Delta_2
s_2$. Noting that $$u(t_1,t_2)\leq
a^{\frac{1}{p}}(t_1,t_2)v^{\frac{1}{p}}(t_1,t_2),$$ we obtain the
desired inequality (\ref{in4}).
\end{proof}

\begin{theorem}[cf. \cite{Ferr6}]
Let $u(t_1,t_2), a(t_1,t_2)\in
C(\tilde{\mathbb{T}}_1\times\tilde{\mathbb{T}}_2,\mathbb{R}_0^+)$,
with $a(t_1,t_2)$ nondecreasing in each of the variables and
$f(t_1,t_2,s_1,s_2)\in C(S,\mathbb{R}_0^+)$, where
$S=\{(t_1,t_2,s_1,s_2)\in\tilde{\mathbb{T}}_1\times\tilde{\mathbb{T}}_2\times\tilde{\mathbb{T}}_1\times\tilde{\mathbb{T}}_2:a_1\leq
s_1\leq t_1,a_2\leq s_2\leq t_2\}$. If $p$ and $q$ are two
positive real numbers such that $p\geq q$ and if
\begin{equation}\label{ineq1}
u^p(t_1,t_2)\leq a(t_1,t_2)+\int_{a_1}^{t_1}\int_{a_2}^{t_2}
f(t_1,t_2,s_1,s_2)u^q(s_1,s_2)\Delta_1 s_1\Delta_2 s_2,
\end{equation}
for all $(t_1,t_2)\in
\tilde{\mathbb{T}}_1\times\tilde{\mathbb{T}}_2$, then
\begin{equation}
u(t_1,t_2)\leq
a^{\frac{1}{p}}(t_1,t_2)\left[e_{\int_{a_2}^{t_2}a^{\frac{q}{p}-1}(t_1,s_2)f(t_1,t_2,t_1,s_2)\Delta_2
s_2}(t_1,a_1)\right]^{\frac{1}{p}},\nonumber
\end{equation}
for all $(t_1,t_2)\in
\tilde{\mathbb{T}}_1\times\tilde{\mathbb{T}}_2$.
\end{theorem}

\begin{proof}
Since $a(t_1,t_2)$ is positive and nondecreasing on
$(t_1,t_2)\in\tilde{\mathbb{T}}_1\times\tilde{\mathbb{T}}_2$,
inequality \eqref{ineq1} implies that,
$$u^p(t_1,t_2)\leq a(t_1,t_2)\left(1+\int_{a_1}^{t_1}\int_{a_2}^{t_2}
f(t_1,t_2,s_1,s_2)\frac{u^q(s_1,s_2)}{a(s_1,s_2)}\Delta_1
s_1\Delta_2 s_2\right).$$  Fix $t_1^\ast\in\tilde{\mathbb{T}}_1$
and $t_2^\ast\in\tilde{\mathbb{T}}_2$ arbitrarily and define a
function $v(t_1,t_2)$ on
$[a_1,t_1^\ast]\cap\tilde{\mathbb{T}}_1\times[a_2,t_2^\ast]\cap\tilde{\mathbb{T}}_2$
by
$$v(t_1,t_2)=1+\int_{a_1}^{t_1}\int_{a_2}^{t_2}
f(t_1^\ast,t_2^\ast,s_1,s_2)\frac{u^q(s_1,s_2)}{a(s_1,s_2)}\Delta_1
s_1\Delta_2 s_2.$$ Then,
\begin{align}
\frac{\partial}{\Delta_2 t_2}\left(\frac{\partial
v(t_1,t_2)}{\Delta_1
t_1}\right)&=f(t_1^\ast,t_2^\ast,t_1,t_2)\frac{u^q(t_1,t_2)}{a(t_1,t_2)}\nonumber\\
&\leq
f(t_1^\ast,t_2^\ast,t_1,t_2)a^{\frac{q}{p}-1}(t_1,t_2)v^{\frac{q}{p}}(t_1,t_2).\nonumber
\end{align}
Since $v^{\frac{q}{p}}(t_1,t_2)\leq v(t_1,t_2)$, we have that
$$\frac{\partial}{\Delta_2 t_2}\left(\frac{\partial
v(t_1,t_2)}{\Delta_1 t_1}\right)\leq
f(t_1^\ast,t_2^\ast,t_1,t_2)a^{\frac{q}{p}-1}(t_1,t_2)v(t_1,t_2).$$
We can follow the same steps as before to reach the inequality
$$\frac{\partial v(t_1,t_2)}{\Delta_1 t_1}\leq\int_{a_2}^{t_2}f(t_1^\ast,t_2^\ast)g(t_1^\ast,t_2^\ast,t_1,s_2)a^{\frac{q}{p}-1}(t_1,s_2)\Delta_2 s_2 v(t_1,t_2).$$
Fix $t_2=t_2^\ast$ and put
$p(t_1):=\int_{a_2}^{t_2^\ast}f(t_1^\ast,t_2^\ast)g(t_1^\ast,t_2^\ast,t_1,s_2)a^{\frac{q}{p}-1}(t_1,s_2)\Delta_2
s_2\in \mathcal{R}^+$. Again, an application of Theorem \ref{teo3}
gives
$$v(t_1,t_2^\ast)\leq e_p(t_1,a_1),$$
and putting $t_1=t_1^\ast$ we obtain the desired inequality.
\end{proof}

We consider now the following time scale: let
$\{\alpha_k\}_{k\in\mathbb{N}}$ be a sequence of positive numbers
and let $t_0^\alpha\in\mathbb{R}$. Let
$$t_k^\alpha=t_0^\alpha+\sum_{n=1}^k\alpha_n,\ k\in\mathbb{N},$$
and assume that $\lim_{k\rightarrow\infty}t_k^\alpha=\infty$. Then
we can define the following time scale
$\mathbb{T}^\alpha=\{t_k^\alpha:k\in\mathbb{N}_0\}$. Further, for
$p\in\mathcal{R}$ (cf. \cite[Example 4.6]{Agarwal}),
\begin{equation}
\label{e1} e_p(t_k^\alpha,t_0^\alpha)=\prod_{n=1}^k(1+\alpha_n
p(t_{n-1})),\quad \mbox{for all}\quad k\in\mathbb{N}_0.
\end{equation}
Now, for two sequences $\{\alpha_k,\beta_k\}_{k\in\mathbb{N}}$ and
two numbers $t_0^\alpha,t_0^\beta\in\mathbb{R}$ as above, we
define two time scales
$\mathbb{T}^\alpha=\{t_k^\alpha:k\in\mathbb{N}_0\}$ and
$\mathbb{T}^\beta=\{t_k^\beta:k\in\mathbb{N}_0\}$. We can
therefore state the following result:

\begin{cor}[cf. \cite{Ferr6}]
Let $u(t,s), a(t,s)$ be defined on
$\mathbb{T}^\alpha\times\mathbb{T}^\beta$, be nonnegative with $a$
nondecreasing in its variables. Further, let $f(t,s,\tau,\xi)$,
where
$(t,s,\tau,\xi)\in\mathbb{T}^\alpha\times\mathbb{T}^\beta\times\mathbb{T}^\alpha\times\mathbb{T}^\beta$
with $\tau\leq t$ and $\xi\leq s,$ be a nonnegative function and
nondecreasing in its first two variables. If $p$ and $q$ are two
positive real numbers such that $p\geq q$ and if
\begin{equation}
u^p(t,s)\leq
a(t,s)+\sum_{\tau\in[t_0^\alpha,t)_{\mathbb{T}^\alpha}}\sum_{\xi\in[t_0^\beta,s)_{\mathbb{T}^\beta}}
\mu^\alpha(\tau)\mu^\beta(\xi)f(t,s,\tau,\xi)u^q(\tau,\xi),\nonumber
\end{equation}
for all $(t,s)\in\mathbb{T}^\alpha\times\mathbb{T}^\beta$, where
$\mu^\alpha$ and $\mu^\beta$ are the graininess functions of
$\mathbb{T}^\alpha$ and $\mathbb{T}^\beta$, respectively, then
\begin{equation}
u(t,s)\leq
a^{\frac{1}{p}}(t,s)\left[e_{\int_{t_0^\beta}^{s}f(t,s)a^{\frac{q}{p}-1}(t,\xi)g(t,s,t,\xi)\Delta^\beta
\xi}(t,t_0^\alpha)\right]^{\frac{1}{p}},\nonumber
\end{equation}
for all $(t,s)\in\mathbb{T}^\alpha\times\mathbb{T}^\beta$, where
$e$ is given by (\ref{e1}).
\end{cor}

We end this section showing how Theorem \ref{teo2} can be applied
to estimate the solutions of the following partial delta dynamic
equation:

\begin{equation}\label{diff1}\frac{\partial}{\Delta_2 t_2}\left(\frac{\partial
u^2(t_1,t_2)}{\Delta_1
t_1}\right)=F(t_1,t_2,u(t_1,t_2)),
\end{equation}
with the given initial boundary conditions (we assume that
$a_1=a_2=0$)
\begin{equation}\label{IBVP0}
u^2(t_1,0)=g(t_1),\quad u^2(0,t_2)=h(t_2),\quad g(0)=0,\quad
h(0)=0,
\end{equation}
where $F\in
C(\tilde{\mathbb{T}}_1\times\tilde{\mathbb{T}}_2\times\mathbb{R}_0^+,\mathbb{R}_0^+)$,
$g\in C(\tilde{\mathbb{T}}_1,\mathbb{R}_0^+)$, $h\in
C(\tilde{\mathbb{T}}_2,\mathbb{R}^+_0)$, with $g$ and $h$
nondecreasing functions and positive on their domains except at
$a_1$ and $a_2$, respectively.

\begin{theorem}
Assume that on its domain, $F$ satisfies $$F(t_1,t_2,u)\leq
t_2u.$$ If $u(t_1,t_2)$ is a solution of the IBVP above for
$(t_1,t_2)\in\tilde{\mathbb{T}}_1\times\tilde{\mathbb{T}}_2$, then
\begin{equation}
\label{in7} u(t_1,t_2)\leq
\sqrt{(g(t_1)+h(t_2))}\left[e_{\int_{0}^{t_2}s_2
(g(t_1)+h(s_2))^{-\frac{1}{2}}\Delta_2
s_2}(t_1,0)\right]^{\frac{1}{2}},
\end{equation}
for $(t_1,t_2)\in\tilde{\mathbb{T}}_1\times\tilde{\mathbb{T}}_2$,
except at the point $(0,0)$.
\end{theorem}

\begin{proof}
Let $u(t_1,t_2)$ be a solution of the given IBVP
\eqref{diff1}--\eqref{IBVP0}. Then it satisfies the following
delta integral equation
$$u^2(t_1,t_2)=g(t_1)+h(t_2)+\int_0^{t_1}\int_0^{t_2}F(s_1,s_2,u(s_1,s_2))\Delta_1 s_1\Delta_2 s_2.$$
The hypothesis on $F$ imply that
$$u^2(t_1,t_2)\leq g(t_1)+h(t_2)+\int_0^{t_1}\int_0^{t_2}s_2 u(s_1,s_2)\Delta_1 s_1\Delta_2 s_2.$$
An application of Theorem \ref{teo2} with
$a(t_1,t_2)=g(t_1)+h(t_2)$ and $f(t_1,t_2)=t_2$ gives (\ref{in7}).
\end{proof}

\section{State of the Art}

The results of this chapter are published or accepted for
publication in \cite{rchid,Ferr8,Ferr4,Ferr6}. Moreover the three
papers \cite{Ferr0,Ferr1,Ferr5} are already available in the
literature and contain related results. Some partial results of
\cite{Ferr0} were presented in the \emph{Time Scales seminar} at
the Missouri University of Science and Technology.

The study of integral inequalities on time scales (with both
$\Delta$-integrals or $\Diamond_\alpha$-integrals) and its
applications is under strong current research. We provide here
some more (recent) references within this subject:
\cite{AT1,AT,BDum,dinu0,dinu,Liu}.

\clearpage{\thispagestyle{empty}\cleardoublepage}

\chapter{Discrete Fractional Calculus}\label{disfrac}

We introduce a discrete-time fractional calculus of
variations\index{Discrete Fractional Calculus} on the time scale
$h\mathbb{Z}$, $h
> 0$. First and second order necessary optimality conditions are
established. Examples illustrating the use of the new
Euler--Lagrange and Legendre type conditions are given. They show
that solutions to the considered fractional problems become the
classical discrete-time solutions when the fractional order of the
discrete-derivatives are integer values, and that they converge to
the fractional continuous-time solutions when $h$ tends to zero.
Our Legendre type condition is useful to eliminate false
candidates identified via the Euler--Lagrange fractional equation.

\section{Introduction}
\label{int}

The \emph{Fractional Calculus} is an important research field in
several different areas: physics (including classical and quantum
mechanics as well as thermodynamics), chemistry, biology,
economics, and control theory \cite{Miller1,TenreiroMachado}. It
has its origin more than 300 years ago when L'Hopital asked
Leibniz what should be the meaning of a derivative of non-integer
order. After that episode several more famous mathematicians
contributed to the development of Fractional Calculus: Abel,
Fourier, Liouville, Riemann, Riesz, just to mention a few names.

In \cite{Miller} Miller and Ross define a fractional sum of order
$\nu>0$ \emph{via} the solution of a linear difference equation.
They introduce it as
\begin{equation}
\label{naosei8}
\Delta^{-\nu}f(t)=\frac{1}{\Gamma(\nu)}\sum_{s=a}^{t-\nu}(t-\sigma(s))^{(\nu-1)}f(s).
\end{equation}
Definition \eqref{naosei8} is the discrete analogue of the
Riemann-Liouville fractional integral
$$
_a\mathbf{D}_x^{-\nu}f(x)=\frac{1}{\Gamma(\nu)}\int_{a}^{x}(x-s)^{\nu-1}f(s)ds
$$
of order $\nu>0$, which can be obtained \emph{via} the solution of
a linear differential equation \cite{Miller,Miller1}. Basic
properties of the operator $\Delta^{-\nu}$ in (\ref{naosei8}) were
obtained in \cite{Miller}. More recently, Atici and Eloe
introduced the fractional difference of order $\alpha>0$ by
$\Delta^\alpha f(t)=\Delta(\Delta^{-(1-\alpha)}f(t))$, and
developed some of its properties that allow to obtain solutions of
certain fractional difference equations \cite{Atici0,Atici}.

The fractional differential calculus has been widely developed in
the past few decades due mainly to its demonstrated applications
in various fields of science and engineering
\cite{book:Kilbas,Miller1,Podlubny}. The study of necessary
optimality conditions for fractional problems of the calculus of
variations and optimal control is a fairly recent issue attracting
an increasing attention -- see
\cite{agr0,RicDel,El-Nabulsi1,El-Nabulsi2,gastao:delfim,gasta1}
and references therein -- but available results address only the
continuous-time case. Therefore, it is pertinent to start a
fractional discrete-time theory of the calculus of variations,
namely, using the time scale $\mathbb{T}=h\mathbb{Z}$, $h > 0$.
Computer simulations show that this time scale is particularly
interesting because when $h$ tends to zero one recovers previous
fractional continuous-time results.

Our objective is two-fold. On one hand we proceed to develop the
theory of \emph{fractional difference calculus}, e.g., introducing
the concept of left and right fractional sum/difference
(\textrm{cf.} Definitions~\ref{def0} and \ref{def1}). On the other
hand, we believe that the results herein presented will potentiate
research not only in the fractional calculus of variations but
also in solving fractional difference equations, specifically,
fractional equations in which left and right fractional
differences appear [see the Euler--Lagrange Equation \eqref{EL}].

Main results so far obtained by us contain a fractional formula of
$h$-summation by parts (Theorem~\ref{teorr1}), and necessary
optimality conditions of first and second order
(Theorems~\ref{teorem0} and \ref{teorem1}, respectively) for the
proposed $h$-fractional problem of the calculus of variations
\eqref{naosei7}. We also provide some illustrative examples.

We remark that it remains a challenge how to generalize the
present results to an arbitrary time scale $\mathbb{T}$. This is a
difficult open problem since an explicit formula for the Cauchy
function (cf. Definition \ref{cauchy}) of the linear dynamic
equation $y^{\Delta^n}=0$ on an arbitrary time scale is not known.


\section{Preliminaries}
\label{sec0}

It is known that $y(t,s):=H_{n-1}(t,\sigma(s))$ is the Cauchy
function (cf. Definition \ref{cauchy}) for $y^{\Delta^n}=0$
\cite[Example~5.115]{livro}, where $H_{n-1}$ is the time scale
generalized polynomial $h_{n-1}$ of Definition \ref{polinomios}
(we use here $H$ instead of $h$ in order to avoid confuse with the
parameter $h$ of the time scale $\mathbb{T}=h\mathbb{Z}$).

From now on we restrict ourselves to the time scale $\mathbb{T} =
h\mathbb{Z}$, $h > 0$.

If we have a function $f$ of two variables, $f(t,s)$, its partial
forward $h$-differences will be denoted by $\Delta_{t,h}$ and
$\Delta_{s,h}$, respectively. We will make use of the standard
conventions $\sum_{t=c}^{c-1}f(t)=0$, $c\in\mathbb{Z}$, and
$\prod_{i=0}^{-1}f(i)=1$.

Before giving an explicit formula for the generalized polynomials
$H_{k}$ on $\mathbb{T}=h\mathbb{Z}$ we introduce the following
definition:
\begin{definition}[cf. \cite{BAFT}]
For arbitrary $x,y\in\mathbb{R}$ the $h$-factorial function is
defined by
\begin{equation*}
x_h^{(y)}:=h^y\frac{\Gamma(\frac{x}{h}+1)}{\Gamma(\frac{x}{h}+1-y)}\,
,
\end{equation*}
where $\Gamma$ is the well-known Euler gamma function, and we use
the convention that division at a pole yields zero.
\end{definition}

\begin{remark}
For $h = 1$, and in accordance with the previous literature [see
formula \eqref{naosei8}], we write $x^{(y)}$ to denote
$x_h^{(y)}$.
\end{remark}

\begin{prop}
\label{prop:d} For the time-scale $\mathbb{T}$ one has
\begin{equation}
\label{hn} H_{k}(t,s):=\frac{(t-s)_h^{(k)}}{k!}\quad\mbox{for
all}\quad s,t\in h\mathbb{Z} \text{ and } k\in \mathbb{N}_0 \, .
\end{equation}
\end{prop}

To prove \eqref{hn} we use the following technical lemma. We
remind the reader the basic property $\Gamma(x+1)=x\Gamma(x)$ of
the gamma function.

\begin{lemma}
\label{lem:tl} Let $s \in \mathbb{T}$. Then, for all $t \in
\mathbb{T}^\kappa$ one has
\begin{equation*}
\Delta_{t,h} \left\{\frac{(t-s)_h^{(k+1)}}{(k+1)!}\right\} =
\frac{(t-s)_h^{(k)}}{k!} \, .
\end{equation*}
\end{lemma}
\begin{proof}
The equality follows by direct computations:
\begin{equation*}
\begin{split}
\Delta_{t,h} &\left\{\frac{(t-s)_h^{(k+1)}}{(k+1)!}\right\}
=\frac{1}{h}\left\{\frac{(\sigma(t)-s)_h^{(k+1)}}{(k+1)!}-\frac{(t-s)_h^{(k+1)}}{(k+1)!}\right\}\\
&=\frac{h^{k+1}}{h(k+1)!}\left\{\frac{\Gamma((t+h-s)/h+1)}{\Gamma((t+h-s)/h+1-(k+1))}-\frac{\Gamma((t-s)/h+1)}{\Gamma((t-s)/h+1-(k+1))}\right\}\\
&=\frac{h^k}{(k+1)!}\left\{\frac{\Gamma((t-s)/h+2)}{\Gamma((t-s)/h+1-k)}-\frac{\Gamma((t-s)/h+1)}{\Gamma((t-s)/h-k)}\right\}\\
&=\frac{h^k}{(k+1)!}\left\{\frac{((t-s)/h+1)\Gamma((t-s)/h+1)}{((t-s)/h-k)\Gamma((t-s)/h-k)}-\frac{\Gamma((t-s)/h+1)}{\Gamma((t-s)/h-k)}\right\}\\
&=\frac{h^k}{(k+1)!}\left\{\frac{(k+1)\Gamma((t-s)/h+1)}{((t-s)/h-k)\Gamma((t-s)/h-k)}\right\}
=\frac{h^k}{k!}\left\{\frac{\Gamma((t-s)/h+1)}{\Gamma((t-s)/h+1-k)}\right\}\\
&=\frac{(t-s)_h^{(k)}}{k!} \, .
\end{split}
\end{equation*}
\end{proof}
\begin{proof}(of Proposition~\ref{prop:d})
We proceed by mathematical induction. For $k=0$
$$
H_0(t,s)=\frac{1}{0!}h^0\frac{\Gamma(\frac{t-s}{h}+1)}{\Gamma(\frac{t-s}{h}+1-0)}
=\frac{\Gamma(\frac{t-s}{h}+1)}{\Gamma(\frac{t-s}{h}+1)}=1 \, .
$$
Assume that \eqref{hn} holds for $k$ replaced by $m$. Then, by
Lemma~\ref{lem:tl}
\begin{eqnarray*}
H_{m+1}(t,s) &=& \int_s^t H_m(\tau,s)\Delta\tau = \int_s^t
\frac{(\tau-s)_h^{(m)}}{m!} \Delta\tau =
\frac{(t-s)_h^{(m+1)}}{(m+1)!},
\end{eqnarray*}
which is \eqref{hn} with $k$ replaced by $m+1$.
\end{proof}

Let $y_1(t),\ldots,y_n(t)$ be $n$ linearly independent solutions
of the linear homogeneous dynamic equation $y^{\Delta^n}=0$. From
Theorem~\ref{eqsol} we know that the solution of \eqref{IVP} (with
$L=\Delta^n$ and $t_0=a$) is
\begin{equation*}
y(t) = \Delta^{-n} f(t)=\int_a^t
\frac{(t-\sigma(s))_h^{(n-1)}}{\Gamma(n)}f(s)\Delta s\\
=\frac{1}{\Gamma(n)}\sum_{k=a/h}^{t/h-1} (t-\sigma(kh))_h^{(n-1)}
f(kh) h \, .
\end{equation*}
Since $y^{\Delta_i}(a)=0$, $i = 0,\ldots,n-1$, then we can write
\begin{equation}
\label{eq:derDh:int}
\begin{split}
\Delta^{-n} f(t) &= \frac{1}{\Gamma(n)}\sum_{k=a/h}^{t/h-n}
(t-\sigma(kh))_h^{(n-1)} f(kh) h \\
&=
\frac{1}{\Gamma(n)}\int_a^{\sigma(t-nh)}(t-\sigma(s))_h^{(n-1)}f(s)
\Delta s \, .
\end{split}
\end{equation}
Note that function $t \rightarrow (\Delta^{-n} f)(t)$ is defined
for $t=a+n h \mbox{ mod}(h)$ while function $t \rightarrow f(t)$
is defined for $t=a \mbox{ mod}(h)$. Extending
\eqref{eq:derDh:int} to any positive real value $\nu$, and having
as an analogy the continuous left and right fractional derivatives
\cite{Miller1}, we define the left fractional $h$-sum and the
right fractional $h$-sum as follows. We denote by
$\mathcal{F}_\mathbb{T}$ the set of all real valued functions
defined on the time scale $\mathbb{T}$.

\begin{definition}\label{def0}
Let $f\in\mathcal{F}_\mathbb{T}$. The left and right
fractional\index{Fractional sum operators} $h$-sum of order
$\nu>0$ are, respectively, the operators $_a\Delta_h^{-\nu} :
\mathcal{F}_\mathbb{T} \rightarrow
\mathcal{F}_{\tilde{\mathbb{T}}_\nu^+}$ and $_h\Delta_b^{-\nu} :
\mathcal{F}_\mathbb{T} \rightarrow
\mathcal{F}_{\tilde{\mathbb{T}}_\nu^-}$,
$\tilde{\mathbb{T}}_\nu^\pm = \{a \pm \nu h : a \in \mathbb{T}\}$,
defined by
\begin{equation*}
\begin{split}
_a\Delta_h^{-\nu}f(t) &=
\frac{1}{\Gamma(\nu)}\int_{a}^{\sigma(t-\nu
h)}(t-\sigma(s))_h^{(\nu-1)}f(s)\Delta s
=\frac{1}{\Gamma(\nu)}\sum_{k=\frac{a}{h}}^{\frac{t}{h}-\nu}(t-\sigma(kh))_h^{(\nu-1)}f(kh)h\\
_h\Delta_b^{-\nu}f(t) &=
\frac{1}{\Gamma(\nu)}\int_{t}^{\sigma(b)}(s-\sigma(t))_h^{(\nu-1)}f(s)\Delta
s
=\frac{1}{\Gamma(\nu)}\sum_{k=\frac{t}{h}+\nu}^{\frac{b}{h}}(kh-\sigma(t))_h^{(\nu-1)}f(kh)h.
\end{split}
\end{equation*}
\end{definition}

\begin{lemma}
Let $\nu>0$ be an arbitrary positive real number. For any $t \in
\mathbb{T}$ we have: (i) $\lim_{\nu\rightarrow
0}{_a}\Delta_h^{-\nu}f(t+\nu h)=f(t)$; (ii) $\lim_{\nu\rightarrow
0}{_h}\Delta_b^{-\nu}f(t-\nu h)=f(t)$.
\end{lemma}
\begin{proof}
Since
\begin{align*}
{_a}\Delta_h^{-\nu}f(t+\nu
h)&=\frac{1}{\Gamma(\nu)}\int_{a}^{\sigma(t)}(t+\nu
h-\sigma(s))_h^{(\nu-1)}f(s)\Delta s\\
&=\frac{1}{\Gamma(\nu)}\sum_{k=\frac{a}{h}}^{\frac{t}{h}}(t+\nu h-\sigma(kh))_h^{(\nu-1)}f(kh)h\\
&=h^{\nu}f(t)+\frac{\nu}{\Gamma(\nu+1)}\sum_{k=\frac{a}{h}}^{\frac{\rho(t)}{h}}(t+\nu
h-\sigma(kh))_h^{(\nu-1)}f(kh)h\, ,
\end{align*}
it follows that $\lim_{\nu\rightarrow 0}{_a}\Delta_h^{-\nu}f(t+\nu
h)=f(t)$. To prove (ii) we use a similar method:
\begin{align*}
{_h}\Delta_b^{-\nu}f(t-\nu h)&=\frac{1}{\Gamma(\nu)}\int_{t}^{\sigma(b)}(s+\nu h-\sigma(t))_h^{(\nu-1)}f(s)\Delta s\\
&=h^{\nu}f(t)+\frac{\nu}{\Gamma(\nu+1)}\sum_{k=\frac{\sigma(t)}{h}}^{\frac{b}{h}}(kh+\nu
h-\sigma(t))_h^{(\nu-1)}f(kh)h\, ,
\end{align*}
and therefore $\lim_{\nu\rightarrow 0}{_h}\Delta_b^{-\nu}f(t-\nu
h)=f(t)$.
\end{proof}

For any $t\in\mathbb{T}$ and for any $\nu\geq 0$ we define
$_a\Delta_h^{0}f(t) := {_h}\Delta_b^{0}f(t) := f(t)$ and write
\begin{equation}
\label{seila1}
\begin{gathered}
{_a}\Delta_h^{-\nu}f(t+\nu h) = h^\nu f(t)
+\frac{\nu}{\Gamma(\nu+1)}\int_{a}^{t}(t+\nu
h-\sigma(s))_h^{(\nu-1)}f(s)\Delta s\, , \\
{_h}\Delta_b^{-\nu}f(t)=h^\nu f(t-\nu h) +
\frac{\nu}{\Gamma(\nu+1)}\int_{\sigma(t)}^{\sigma(b)}(s+\nu
h-\sigma(t))_h^{(\nu-1)}f(s)\Delta s \, .
\end{gathered}
\end{equation}

\begin{theorem}\label{teorem2}
Let $f\in\mathcal{F}_\mathbb{T}$ and $\nu\geq0$. For all
$t\in\mathbb{T}^\kappa$ we have
\begin{equation}
\label{naosei1} {_a}\Delta_{h}^{-\nu} f^{\Delta}(t+\nu
h)=(_a\Delta_h^{-\nu}f(t+\nu h))^{\Delta} -\frac{\nu}{\Gamma(\nu +
1)}(t+\nu h-a)_h^{(\nu-1)}f(a) \, .
\end{equation}
\end{theorem}
To prove Theorem~\ref{teorem2} we make use of the following
\begin{lemma}
\label{lemma:tl} Let $t\in\mathbb{T}^\kappa$. The following
equality holds for all $s\in\mathbb{T}^\kappa$:
\begin{multline}
\label{proddiff}
\Delta_{s,h}\left((t+\nu h-s)_h^{(\nu-1)}f(s))\right)\\
=(t+\nu h-\sigma(s))_h^{(\nu-1)}f^{\Delta}(s) -(v-1)(t+\nu
h-\sigma(s))_h^{(\nu-2)}f(s) \, .
\end{multline}
\end{lemma}
\begin{proof}
Direct calculations give the intended result:
\begin{equation*}
\begin{split}
\Delta&_{s,h} \left((t+\nu h-s)_h^{(\nu-1)}f(s)\right)\\
&=\Delta_{s,h}\left((t+\nu h-s)_h^{(\nu-1)}\right)f(s)+\left(t+\nu
h
-\sigma(s)\right)_h^{(\nu-1)}f^{\Delta}(s)\\
&=\frac{\left(t+\nu h-\sigma(s)\right)_h^{(\nu-1)}-\left(t+\nu
h-s\right)_h^{(\nu-1)}}{h}f(s)+\left(t+\nu h -
\sigma(s)\right)_h^{(\nu-1)}\\
&=\frac{f(s)}{h}\left[h^{\nu-1}\frac{\Gamma\left(\frac{t+\nu
h-\sigma(s)}{h}+1\right)}{\Gamma\left(\frac{t+\nu
h-\sigma(s)}{h}+1-(\nu-1)\right)}-h^{\nu-1}\frac{\Gamma\left(\frac{t+\nu
h-s}{h}+1\right)}{\Gamma\left(\frac{t+\nu
h-s}{h}+1-(\nu-1)\right)}\right]\\
&\qquad +\left(t+\nu h -
\sigma(s)\right)_h^{(\nu-1)}f^{\Delta}(s)\\
&=f(s)\left[h^{\nu-2}\left[\frac{\Gamma(\frac{t+\nu
h-s}{h})}{\Gamma(\frac{t-s}{h}+1)}-\frac{\Gamma(\frac{t+\nu
h-s}{h}+1)}{\Gamma(\frac{t-s}{h}+2)}\right]\right]+\left(t+\nu h -
\sigma(s)\right)_h^{(\nu-1)}f^{\Delta}(s)\\
&=f(s)h^{\nu-2}\left[\frac{\Gamma(\frac{t+\nu
h-s}{h})}{\Gamma(\frac{t-s}{h}+2)}\left(\frac{t-s}{h}+1-\frac{t+\nu
h-s}{h}\right)\right] +\left(t+\nu h -
\sigma(s)\right)_h^{(\nu-1)}f^{\Delta}(s)\\
&=f(s)h^{\nu-2}\frac{\Gamma(\frac{t+\nu
h-s-h}{h}+1)}{\Gamma(\frac{t-s+\nu h-h}{h}+1-(\nu-2))}(-(\nu-1))+
\left(t+\nu h - \sigma(s)\right)_h^{(\nu-1)}f^{\Delta}(s)\\
&=-(\nu-1)(t+\nu h -\sigma(s))_h^{(\nu-2)}f(s)+\left(t+\nu h -
\sigma(s)\right)_h^{(\nu-1)}f^{\Delta}(s) \, ,
\end{split}
\end{equation*}
where the first equality follows directly from \eqref{produto}.
\end{proof}

\begin{remark}
Given an arbitrary $t\in\mathbb{T}^\kappa$ it is easy to prove, in
a similar way as in the proof of Lemma~\ref{lemma:tl}, the
following equality analogous to \eqref{proddiff}: for all
$s\in\mathbb{T}^\kappa$
\begin{multline}
\label{eq:semlhante}
\Delta_{s,h}\left((s+\nu h-\sigma(t))_h^{(\nu-1)}f(s))\right)\\
=(\nu-1)(s+\nu h-\sigma(t))_h^{(\nu-2)}f^\sigma(s) + (s+\nu
h-\sigma(t))_h^{(\nu-1)}f^{\Delta}(s) \, .
\end{multline}
\end{remark}
\begin{proof}(of Theorem~\ref{teorem2})
From Lemma~\ref{lemma:tl} we obtain that
\begin{equation}
\label{naosei}
\begin{split}
{_a}\Delta_{h}^{-\nu} & f^{\Delta}(t+\nu h) = h^\nu
f^\Delta(t)+\frac{\nu}{\Gamma(\nu+1)}\int_{a}^{t}(t+\nu
h-\sigma(s))_h^{(\nu-1)}f^{\Delta}(s)\Delta s\\
&=h^\nu f^\Delta(t)+\frac{\nu}{\Gamma(\nu+1)}\left[(t+\nu
h-s)_h^{(\nu-1)}f(s)\right]_{s=a}^{s=t}\\
&\qquad
+\frac{\nu}{\Gamma(\nu+1)}\int_{a}^{\sigma(t)}(\nu-1)(t+\nu
h-\sigma(s))_h^{(\nu-2)}
f(s)\Delta s\\
&=-\frac{\nu(t+\nu h-a)_h^{(\nu-1)}}{\Gamma(\nu+1)}f(a)
+h^{\nu}f^\Delta(t)+\nu h^{\nu-1}f(t)\\
&\qquad +\frac{\nu}{\Gamma(\nu+1)}\int_{a}^{t}(\nu-1)(t+\nu
h-\sigma(s))_h^{(\nu-2)} f(s)\Delta s.
\end{split}
\end{equation}
We now show that $(_a\Delta_h^{-\nu}f(t+\nu h))^{\Delta}$ equals
\eqref{naosei}:
\begin{equation*}
\begin{split}
(_a\Delta_h^{-\nu} & f(t+\nu h))^\Delta = \frac{1}{h}\left[h^\nu
f(\sigma(t))+\frac{\nu}{\Gamma(\nu+1)}\int_{a}^{\sigma(t)}(\sigma(t)+\nu
h-\sigma(s))_h^{(\nu-1)}
f(s)\Delta s\right.\\
&\qquad \left.-h^\nu f(t)-\frac{\nu}{\Gamma(\nu+1)}\int_{a}^{t}(t+\nu h-\sigma(s))_h^{(\nu-1)} f(s)\Delta s\right]\\
&=h^\nu
f^\Delta(t)+\frac{\nu}{h\Gamma(\nu+1)}\left[\int_{a}^{t}(\sigma(t)+\nu
h-\sigma(s))_h^{(\nu-1)}
f(s)\Delta s\right.\\
&\qquad\left.-\int_{a}^{t}(t+\nu h-\sigma(s))_h^{(\nu-1)}
f(s)\Delta s\right]+h^{\nu-1}\nu f(t)\\
&=h^\nu
f^\Delta(t)+\frac{\nu}{\Gamma(\nu+1)}\int_{a}^{t}\Delta_{t,h}\left((t+\nu
h -\sigma(s))_h^{(\nu-1)}
\right)f(s)\Delta s+h^{\nu-1}\nu f(t)\\
&=h^\nu
f^\Delta(t)+\frac{\nu}{\Gamma(\nu+1)}\int_{a}^{t}(\nu-1)(t+\nu
h-\sigma(s))_h^{(\nu-2)} f(s)\Delta s+\nu h^{\nu-1}f(t) \, .
\end{split}
\end{equation*}
\end{proof}

Follows the counterpart of Theorem~\ref{teorem2} for the right
fractional $h$-sum:
\begin{theorem}
\label{teorem3} Let $f\in\mathcal{F}_\mathbb{T}$ and $\nu\geq 0$.
For all $t\in\mathbb{T}^\kappa$ we have
\begin{equation}
\label{naosei12} {_h}\Delta_{\rho(b)}^{-\nu} f^{\Delta}(t-\nu
h)=\frac{\nu}{\Gamma(\nu+1)}(b+\nu
h-\sigma(t))_h^{(\nu-1)}f(b)+(_h\Delta_b^{-\nu}f(t-\nu
h))^{\Delta} \, .
\end{equation}
\end{theorem}
\begin{proof}
From \eqref{eq:semlhante} we obtain that
\begin{equation}
\label{naosei99}
\begin{split}
{_h}\Delta_{\rho(b)}^{-\nu} & f^{\Delta}(t-\nu h) =h^\nu
f^\Delta(t)+\frac{\nu}{\Gamma(\nu+1)}\int_{\sigma(t)}^{b}(s+\nu
h-\sigma(t))_h^{(\nu-1)}
f^{\Delta}(s)\Delta s\\
&=h^\nu f^\Delta(t)+\left[\frac{\nu(s+\nu h
-\sigma(t))_h^{(\nu-1)}}{\Gamma(\nu+1)}f(s)\right]_{s=\sigma(t)}^{s=b}\\
&\qquad -
\frac{\nu}{\Gamma(\nu+1)}\int_{\sigma(t)}^{b}(\nu-1)(s+\nu
h-\sigma(t))_h^{(\nu-2)}
f^\sigma(s)\Delta s\\
&=\frac{\nu(b+\nu h-\sigma(t))_h^{(\nu-1)}}{\Gamma(\nu+1)}f(b)
+ h^\nu f^\Delta(t) -\nu h^{\nu-1}f(\sigma(t))\\
&\qquad
-\frac{\nu}{\Gamma(\nu+1)}\int_{\sigma(t)}^{b}(\nu-1)(s+\nu
h-\sigma(t))_h^{(\nu-2)} f^\sigma(s)\Delta s.
\end{split}
\end{equation}
We show that $(_h\Delta_b^{-\nu}f(t-\nu h))^{\Delta}$ equals
\eqref{naosei99}:
\begin{equation*}
\begin{split}
(_h&\Delta_b^{-\nu} f(t-\nu h))^{\Delta}\\
&=\frac{1}{h}\left[h^{\nu} f(\sigma(t))-h^\nu
f(t)+\frac{\nu}{\Gamma(\nu+1)}\int_{\sigma^2(t)}^{\sigma(b)}(s+\nu
h-\sigma^2(t)))_h^{(\nu-1)}
f(s)\Delta s\right.\\
&\qquad
\left.-\frac{\nu}{\Gamma(\nu+1)}\int_{\sigma(t)}^{\sigma(b)}(s+\nu
h-\sigma(t))_h^{(\nu-1)} f(s)\Delta s\right]
\end{split}
\end{equation*}
\begin{equation*}
\begin{split}
&=h^{\nu}f^\Delta(t)+\frac{\nu}{h\Gamma(\nu+1)}\left[\int_{\sigma^2(t)}^{\sigma(b)}(s+\nu
h-\sigma^2(t)))_h^{(\nu-1)}
f(s)\Delta s\right.\\
&\qquad \left.-\int_{\sigma^2(t)}^{\sigma(b)}(s+\nu
h-\sigma(t))_h^{(\nu-1)}
f(s)\Delta s\right]-\nu h^{\nu-1} f(\sigma(t))\\
&=h^{\nu}f^\Delta(t)+\frac{\nu}{\Gamma(\nu+1)}\int_{\sigma^2(t)}^{\sigma(b)}\Delta_{t,h}\left((s+\nu
h-\sigma(t))_h^{(\nu-1)}\right)
f(s)\Delta s-\nu h^{\nu-1} f(\sigma(t))\\
&=h^{\nu}f^\Delta(t)-\frac{\nu}{\Gamma(\nu+1)}\int_{\sigma^2(t)}^{\sigma(b)}(\nu-1)(s+\nu
h-\sigma^2(t))_h^{(\nu-2)}
f(s)\Delta s-\nu h^{\nu-1} f(\sigma(t))\\
&=h^{\nu}f^\Delta(t)-\frac{\nu}{\Gamma(\nu+1)}\int_{\sigma(t)}^{b}(\nu-1)(s+\nu
h-\sigma(t))_h^{(\nu-2)} f(s)\Delta s-\nu h^{\nu-1} f(\sigma(t)).
\end{split}
\end{equation*}
\end{proof}

\begin{definition}
\label{def1} Let $0<\alpha\leq 1$ and set $\gamma := 1-\alpha$.
The \emph{left fractional difference}\index{Fractional difference
operators} $_a\Delta_h^\alpha f(t)$ and the \emph{right fractional
difference} $_h\Delta_b^\alpha f(t)$ of order $\alpha$ of a
function $f\in\mathcal{F}_\mathbb{T}$ are defined as
\begin{equation*}
_a\Delta_h^\alpha f(t) := (_a\Delta_h^{-\gamma}f(t+\gamma
h))^{\Delta}\ \text{ and } \ _h\Delta_b^\alpha
f(t):=-(_h\Delta_b^{-\gamma}f(t-\gamma h))^{\Delta},
\end{equation*}
for all $t\in\mathbb{T}^\kappa$.
\end{definition}


\section{Main results}
\label{sec1}

Our aim is to introduce the $h$-fractional problem of the calculus
of variations and to prove corresponding necessary optimality
conditions. In order to obtain an Euler--Lagrange type equation
(\textrm{cf.} Theorem~\ref{teorem0}) we first prove a fractional
formula of $h$-summation by parts.


\subsection{Fractional $h$-summation by parts}

A big challenge was to discover a fractional $h$-summation by
parts formula within this time scale setting. Indeed, there is no
clue of what such a formula should be. We found it eventually,
making use of the following lemma.

\begin{lemma}
\label{lemaa1} Let $f$ and $k$ be two functions defined on
$\mathbb{T}^\kappa$ and $\mathbb{T}^{\kappa^2}$, respectively, and
$g$ a function defined on
$\mathbb{T}^\kappa\times\mathbb{T}^{\kappa^2}$. Then, the
following equality holds:
\begin{equation*}
\int_{a}^{b}f(t)\left[\int_{a}^{t}g(t,s)k(s)\Delta s\right]\Delta
t=\int_{a}^{\rho(b)}k(t)\left[\int_{\sigma(t)}^{b}g(s,t)f(s)\Delta
s\right]\Delta t \, .
\end{equation*}
\end{lemma}
\begin{proof}
Consider the matrices $R = \left[ f(a+h), f(a+2h), \cdots, f(b-h)
\right]$,
\begin{equation*}
C_1 = \left[
\begin{array}{c}
g(a+h,a)k(a) \\
g(a+2h,a)k(a)+g(a+2h,a+h)k(a+h) \\
\vdots \\
g(b-h,a)k(a)+g(b-h,a+h)k(a+h)+\cdots+ g(b-h,b-2h)k(b-2h)
\end{array}
\right]
\end{equation*}
\begin{gather*}
C_2 = \left[
\begin{array}{c}
g(a+h,a) \\
g(a+2h,a) \\
\vdots \\
g(b-h,a)  \end{array} \right], \ \  C_3 = \left[
\begin{array}{c}
0 \\
g(a+2h,a+h) \\
\vdots \\
g(b-h,a+h)  \end{array} \right], \ \ C_4 = \left[
\begin{array}{c}
0 \\
0 \\
\vdots \\
g(b-h,b-2h)
\end{array}
\right] .
\end{gather*}
Direct calculations show that
\begin{equation*}
\begin{split}
\int_{a}^{b}&f(t)\left[\int_{a}^{t}g(t,s)k(s)\Delta s\right]\Delta
t
=h^2\sum_{i=a/h}^{b/h-1} f(ih)\sum_{j=a/h}^{i-1}g(ih,jh)k(jh) = h^2 R \cdot C_1\\
&=h^2 R \cdot \left[k(a) C_2 + k(a+h)C_3 +\cdots +k(b-2h) C_4 \right]\\
&=h^2\left[k(a)\sum_{j=a/h+1}^{b/h-1}g(jh,a)f(jh)+k(a+h)\sum_{j=a/h+2}^{b/h-1}g(jh,a+h)f(jh)\right.\\
&\left.\qquad \qquad +\cdots+k(b-2h)\sum_{j=b/h-1}^{b/h-1}g(jh,b-2h)f(jh)\right]\\
&=\sum_{i=a/h}^{b/h-2}k(ih)h\sum_{j=\sigma(ih)/h}^{b/h-1}g(jh,ih)f(jh)
h =\int_a^{\rho(b)}k(t)\left[\int_{\sigma(t)}^b g(s,t)f(s)\Delta
s\right]\Delta t.
\end{split}
\end{equation*}
\end{proof}

\begin{theorem}[fractional $h$-summation by parts]\label{teorr1}
Let $f$ and $g$ be real valued functions defined on
$\mathbb{T}^\kappa$ and $\mathbb{T}$, respectively. Fix
$0<\alpha\leq 1$ and put $\gamma := 1-\alpha$. Then,
\begin{multline}
\label{delf:sumPart} \int_{a}^{b}f(t)_a\Delta_h^\alpha g(t)\Delta
t=h^\gamma f(\rho(b))g(b)-h^\gamma
f(a)g(a)+\int_{a}^{\rho(b)}{_h\Delta_{\rho(b)}^\alpha
f(t)g^\sigma(t)}\Delta t\\
+\frac{\gamma}{\Gamma(\gamma+1)}g(a)\left(\int_{a}^{b}(t+\gamma
h-a)_h^{(\gamma-1)}f(t)\Delta t -\int_{\sigma(a)}^{b}(t+\gamma
h-\sigma(a))_h^{(\gamma-1)}f(t)\Delta t\right).
\end{multline}
\end{theorem}
\begin{proof}
By \eqref{naosei1} we can write
\begin{equation}
\label{rui0}
\begin{split}
\int_{a}^{b} &f(t)_a\Delta_h^\alpha g(t)\Delta t
=\int_{a}^{b}f(t)(_a\Delta_h^{-\gamma} g(t+\gamma h))^{\Delta}\Delta t\\
&=\int_{a}^{b}f(t)\left[_a\Delta_h^{-\gamma}
g^{\Delta}(t+\gamma h)+\frac{\gamma}{\Gamma(\gamma+1)}(t+\gamma h-a)_h^{(\gamma-1)}g(a)\right]\Delta t\\
&=\int_{a}^{b}f(t)_a\Delta_h^{-\gamma}g^{\Delta}(t+\gamma h)\Delta
t +\int_{a}^{b}\frac{\gamma}{\Gamma(\gamma+1)}(t+\gamma
h-a)_h^{(\gamma-1)}f(t)g(a)\Delta t.
\end{split}
\end{equation}
Using \eqref{seila1} we get
\begin{equation*}
\begin{split}
\int_{a}^{b} &f(t)_a\Delta_h^{-\gamma} g^{\Delta}(t+\gamma h) \Delta t\\
&=\int_{a}^{b}f(t)\left[h^\gamma g^{\Delta}(t) +
\frac{\gamma}{\Gamma(\gamma+1)}\int_{a}^{t}(t+\gamma h
-\sigma(s))_h^{(\gamma-1)} g^{\Delta}(s)\Delta s\right]\Delta t\\
&=h^\gamma\int_{a}^{b}f(t)g^{\Delta}(t)\Delta
t+\frac{\gamma}{\Gamma(\gamma+1)}\int_{a}^{b}f(t)\int_{a}^{t}(t+\gamma
h-\sigma(s))_h^{(\gamma-1)}
g^{\Delta}(s)\Delta s \Delta t\\
&=h^\gamma\int_{a}^{b}f(t)g^{\Delta}(t)\Delta
t+\frac{\gamma}{\Gamma(\gamma+1)}\int_{a}^{\rho(b)}
g^{\Delta}(t)\int_{\sigma(t)}^{b}(s+\gamma h-\sigma(t))_h^{(\gamma-1)}f(s)\Delta s \Delta t\\
&=h^\gamma f(\rho(b))[g(b)-g(\rho(b))]+\int_{a}^{\rho(b)}
g^{\Delta}(t)_h\Delta_{\rho(b)}^{-\gamma} f(t-\gamma h)\Delta t,
\end{split}
\end{equation*}
where the third equality follows by Lemma~\ref{lemaa1}. We proceed
to develop the right-hand side of the last equality as follows:
\begin{equation*}
\begin{split}
h^\gamma & f(\rho(b))[g(b)-g(\rho(b))]+\int_{a}^{\rho(b)}
g^{\Delta}(t)_h\Delta_{\rho(b)}^{-\gamma} f(t-\gamma h)\Delta t\\
&=h^\gamma f(\rho(b))[g(b)-g(\rho(b))]
+\left[g(t)_h\Delta_{\rho(b)}^{-\gamma}
f(t-\gamma h)\right]_{t=a}^{t=\rho(b)}\\
&\quad -\int_{a}^{\rho(b)} g^\sigma(t)(_h\Delta_{\rho(b)}^{-\gamma} f(t-\gamma h))^{\Delta}\Delta t\\
&=h^\gamma f(\rho(b))g(b)-h^\gamma f(a)g(a)\\
&\quad
-\frac{\gamma}{\Gamma(\gamma+1)}g(a)\int_{\sigma(a)}^{b}(s+\gamma
h-\sigma(a))_h^{(\gamma-1)}f(s)\Delta s
+\int_{a}^{\rho(b)}{\left(_h\Delta_{\rho(b)}^\alpha
f(t)\right)g^\sigma(t)}\Delta t,
\end{split}
\end{equation*}
where the first equality follows from \eqref{partes2}. Putting
this into (\ref{rui0}) we get \eqref{delf:sumPart}.
\end{proof}

Let us prove that Theorem \ref{teorr1} indeed generalizes the
usual summation by parts formula
\begin{cor}
Suppose that $h=\alpha=1$ in Theorem \ref{teorr1}. Then, the next
formula holds:
\begin{equation}\label{sumpartes}
\int_a^b f(t)\Delta g(t)\Delta
t=f(\rho(b))g(b)-f(a)g(a)-\int_a^{\rho(b)}\Delta
f(t)g^\sigma(t)\Delta t.
\end{equation}
\end{cor}
\begin{remark}
Since $h=1$, then \eqref{sumpartes} becomes
$$\sum_{t=a}^{b-1} f(t)[g(t+1)-g(t)]=f(b-1)g(b)-f(a)g(a)-\sum_{t=a}^{b-2}[f(t+1)-f(t)]g(t+1).$$
If we allow $f$ to be defined at $t=b$, then we have
\begin{align}
\sum_{t=a}^{b-1}
f(t)[g(t+1)-g(t)]&=f(b-1)g(b)-f(a)g(a)-\sum_{t=a}^{b-2}[f(t+1)-f(t)]g(t+1)\nonumber\\
&=f(b)g(b)-f(a)g(a)-[f(b)-f(b-1)]g(b)\nonumber\\
&\hspace{3cm}-\sum_{t=a}^{b-2}[f(t+1)-f(t)]g(t+1)\nonumber\\
&=f(b)g(b)-f(a)g(a)-\sum_{t=a}^{b-1}[f(t+1)-f(t)]g(t+1)\nonumber.
\end{align}
This is the usual summation by parts formula.
\end{remark}


\subsection{Necessary optimality conditions}

We begin to fix two arbitrary real numbers $\alpha$ and $\beta$
such that $\alpha,\beta\in(0,1]$. Further, we put $\gamma :=
1-\alpha$ and $\nu :=1-\beta$.

Let a function
$L(t,u,v,w):\mathbb{T}^\kappa\times\mathbb{R}\times\mathbb{R}\times\mathbb{R}\rightarrow\mathbb{R}$
be given. We consider the problem of minimizing (or maximizing) a
functional
$\mathcal{L}:\mathcal{F}_\mathbb{T}\rightarrow\mathbb{R}$ subject
to given boundary conditions:
\begin{equation}
\label{naosei7}
\mathcal{L}[y(\cdot)]=\int_{a}^{b}L(t,y^{\sigma}(t),{_a}\Delta_h^\alpha
y(t),{_h}\Delta_b^\beta y(t))\Delta t \longrightarrow \min, \
y(a)=A, \ y(b)=B \, .
\end{equation}
Our main aim is to derive necessary optimality conditions for
problem \eqref{naosei7}.
\begin{definition}
For $f\in\mathcal{F}_\mathbb{T}$ we define the norm
$$\|f\|=\max_{t\in\mathbb{T}^\kappa}|f^\sigma(t)|+\max_{t\in\mathbb{T}^\kappa}|_a\Delta_h^\alpha
f(t)|+\max_{t\in\mathbb{T}^\kappa}|_h\Delta_b^\beta f(t)|.$$ A
function $\hat{y}\in\mathcal{F}_\mathbb{T}$ with $\hat{y}(a)=A$
and $\hat{y}(b)=B$ is called a local minimum for problem
\eqref{naosei7} provided there exists $\delta>0$ such that
$\mathcal{L}(\hat{y})\leq\mathcal{L}(y)$ for all
$y\in\mathcal{F}_\mathbb{T}$ with $y(a)=A$ and $y(b)=B$ and
$\|y-\hat{y}\|<\delta$.
\end{definition}

\begin{definition}
A function $\eta\in\mathcal{F}_\mathbb{T}$ is called an admissible
variation provided $\eta \neq 0$ and $\eta(a)=\eta(b)=0$.
\end{definition}

From now on we assume that the second-order partial derivatives
$L_{uu}$, $L_{uv}$, $L_{uw}$, $L_{vw}$, $L_{vv}$, and $L_{ww}$
exist and are continuous.


\subsubsection{First order optimality condition}

The next theorem gives a first order necessary condition for
problem \eqref{naosei7}, \textrm{i.e.}, an Euler--Lagrange type
equation for the fractional $h$-difference
setting\index{Euler--Lagrange equation!discrete fractional case}.
\begin{theorem}[The $h$-fractional Euler-Lagrange equation for problem \eqref{naosei7}]
\label{teorem0} If $\hat{y}\in\mathcal{F}_\mathbb{T}$ is a local
minimum for problem \eqref{naosei7}, then the equality
\begin{equation}
\label{EL} L_u[\hat{y}](t) +{_h}\Delta_{\rho(b)}^\alpha
L_v[\hat{y}](t)+{_a}\Delta_h^\beta L_w[\hat{y}](t)=0
\end{equation}
holds for all $t\in\mathbb{T}^{\kappa^2}$ with operator $[\cdot]$
defined by $[y](s) =(s,y^{\sigma}(s),{_a}\Delta_s^\alpha
y(s),{_s}\Delta_b^\beta y(s))$.
\end{theorem}
\begin{proof}
Suppose that $\hat{y}(\cdot)$ is a local minimum of
$\mathcal{L}[\cdot]$. Let $\eta(\cdot)$ be an arbitrarily fixed
admissible variation and define a function
$\Phi:\left(-\frac{\delta}{\|\eta(\cdot)\|},\frac{\delta}{\|\eta(\cdot)\|}\right)\rightarrow\mathbb{R}$
by
\begin{equation}
\label{fi}
\Phi(\varepsilon)=\mathcal{L}[\hat{y}(\cdot)+\varepsilon\eta(\cdot)].
\end{equation}
This function has a minimum at $\varepsilon=0$, so we must have
$\Phi'(0)=0$, i.e.,
$$\int_{a}^{b}\left[L_u[\hat{y}](t)\eta^\sigma(t)
+L_v[\hat{y}](t){_a}\Delta_h^\alpha\eta(t)
+L_w[\hat{y}](t){_h}\Delta_b^\beta\eta(t)\right]\Delta t=0,$$
which we may write equivalently as
\begin{multline}
\label{rui33} h
L_u[\hat{y}](t)\eta^\sigma(t)|_{t=\rho(b)}+\int_{a}^{\rho(b)}L_u[\hat{y}](t)\eta^\sigma(t)\Delta
t +\int_{a}^{b}L_v[\hat{y}](t){_a}\Delta_h^\alpha\eta(t)\Delta
t\\+\int_{a}^{b}L_w[\hat{y}](t){_h}\Delta_b^\beta\eta(t)\Delta
t=0.
\end{multline}
Using Theorem~\ref{teorr1} and the fact that $\eta(a)=\eta(b)=0$,
we get
\begin{equation}
\label{naosei5}
\int_{a}^{b}L_v[\hat{y}](t){_a}\Delta_h^\alpha\eta(t)\Delta
t=\int_{a}^{\rho(b)}\left({_h}\Delta_{\rho(b)}^\alpha
\left(L_v[\hat{y}]\right)(t)\right)\eta^\sigma(t)\Delta t
\end{equation}
for the third term in \eqref{rui33}. Using \eqref{naosei12} it
follows that
\begin{equation}
\label{naosei4}
\begin{split}
\int_{a}^{b} & L_w[\hat{y}](t){_h}\Delta_b^\beta\eta(t)\Delta t\\=&-\int_{a}^{b}L_w[\hat{y}](t)({_h}\Delta_b^{-\nu}\eta(t-\nu h))^{\Delta}\Delta t\\
=&-\int_{a}^{b}L_w[\hat{y}](t)\left[{_h}\Delta_{\rho(b)}^{-\nu}
\eta^{\Delta}(t-\nu h)-\frac{\nu}{\Gamma(\nu+1)}(b+\nu h-\sigma(t))_h^{(\nu-1)}\eta(b)\right]\Delta t\\
=&-\int_{a}^{b}L_w[\hat{y}](t){_h}\Delta_{\rho(b)}^{-\nu}
\eta^{\Delta}(t-\nu h)\Delta t
+\frac{\nu\eta(b)}{\Gamma(\nu+1)}\int_{a}^{b}(b+\nu
h-\sigma(t))_h^{(\nu-1)}L_w[\hat{y}](t)\Delta t .
\end{split}
\end{equation}
We now use Lemma~\ref{lemaa1} to get
\begin{equation}
\label{naosei2}
\begin{split}
\int_{a}^{b} &L_w[\hat{y}](t){_h}\Delta_{\rho(b)}^{-\nu} \eta^{\Delta}(t-\nu h)\Delta t\\
&=\int_{a}^{b}L_w[\hat{y}](t)\left[h^\nu\eta^{\Delta}(t)+\frac{\nu}{\Gamma(\nu+1)}\int_{\sigma(t)}^{b}(s+\nu
h-\sigma(t))_h^{(\nu-1)} \eta^{\Delta}(s)\Delta s\right]\Delta t\\
&=\int_{a}^{b}h^\nu L_w[\hat{y}](t)\eta^{\Delta}(t)\Delta t\\
&\qquad
+\frac{\nu}{\Gamma(\nu+1)}\int_{a}^{\rho(b)}\left[L_w[\hat{y}](t)\int_{\sigma(t)}^{b}(s+\nu
h-\sigma(t))_h^{(\nu-1)} \eta^{\Delta}(s)\Delta s\right]\Delta t\\
&=\int_{a}^{b}h^\nu L_w[\hat{y}](t)\eta^{\Delta}(t)\Delta t\\
&\qquad
+\frac{\nu}{\Gamma(\nu+1)}\int_{a}^{b}\left[\eta^{\Delta}(t)\int_{a}^{t}(t+\nu
h
-\sigma(s))_h^{(\nu-1)}L_w[\hat{y}](s)\Delta s\right]\Delta t\\
&=\int_{a}^{b}\eta^{\Delta}(t){_a}\Delta^{-\nu}_h
\left(L_w[\hat{y}]\right)(t+\nu h)\Delta t.
\end{split}
\end{equation}
We apply again the time scale integration by parts formula
\eqref{partes2}, this time to \eqref{naosei2}, to obtain,
\begin{equation}
\label{naosei3}
\begin{split}
\int_{a}^{b} & \eta^{\Delta}(t){_a}\Delta^{-\nu}_h
\left(L_w[\hat{y}]\right)(t+\nu h)\Delta t\\
&=\int_{a}^{\rho(b)}\eta^{\Delta}(t){_a}\Delta^{-\nu}_h
\left(L_w[\hat{y}]\right)(t+\nu h)\Delta t\\
&\qquad +(\eta(b)-\eta(\rho(b))){_a}\Delta^{-\nu}_h
\left(L_w[\hat{y}]\right)(t+\nu h)|_{t=\rho(b)}\\
&=\left[\eta(t){_a}\Delta^{-\nu}_h \left(L_w[\hat{y}]\right)(t+\nu
h)\right]_{t=a}^{t=\rho(b)}
-\int_{a}^{\rho(b)}\eta^\sigma(t)({_a}\Delta^{-\nu}_h
\left(L_w[\hat{y}]\right)(t+\nu h))^\Delta \Delta t\\
&\qquad +\eta(b){_a}\Delta^{-\nu}_h
\left(L_w[\hat{y}]\right)(t+\nu
h)|_{t=\rho(b)}-\eta(\rho(b)){_a}\Delta^{-\nu}_h
\left(L_w[\hat{y}]\right)(t+\nu h)|_{t=\rho(b)}\\
&=\eta(b){_a}\Delta^{-\nu}_h \left(L_w[\hat{y}]\right)(t+\nu
h)|_{t=\rho(b)}-\eta(a){_a}\Delta^{-\nu}_h
\left(L_w[\hat{y}]\right)(t+\nu h)|_{t=a}\\
&\qquad -\int_{a}^{\rho(b)}\eta^\sigma(t){_a}\Delta^{\beta}_h
\left(L_w[\hat{y}]\right)(t)\Delta t.
\end{split}
\end{equation}
Since $\eta(a)=\eta(b)=0$ we obtain, from \eqref{naosei2} and
\eqref{naosei3}, that
$$\int_{a}^{b}L_w[\hat{y}](t){_h}\Delta_{\rho(b)}^{-\nu}
\eta^\Delta(t)\Delta t
=-\int_{a}^{\rho(b)}\eta^\sigma(t){_a}\Delta^{\beta}_h
\left(L_w[\hat{y}]\right)(t)\Delta t\, ,$$ and after inserting in
\eqref{naosei4}, that
\begin{equation}
\label{naosei6}
\int_{a}^{b}L_w[\hat{y}](t){_h}\Delta_b^\beta\eta(t)\Delta t
=\int_{a}^{\rho(b)}\eta^\sigma(t){_a}\Delta^{\beta}_h
\left(L_w[\hat{y}]\right)(t) \Delta t.
\end{equation}
By \eqref{naosei5} and \eqref{naosei6} we may write \eqref{rui33}
as
$$\int_{a}^{\rho(b)}\left[L_u[\hat{y}](t)
+{_h}\Delta_{\rho(b)}^\alpha
\left(L_v[\hat{y}]\right)(t)+{_a}\Delta_h^\beta
\left(L_w[\hat{y}]\right)(t)\right]\eta^\sigma(t) \Delta t =0\,
.$$ Since the values of $\eta^\sigma(t)$ are arbitrary for
$t\in\mathbb{T}^{\kappa^2}$, the Euler-Lagrange equation
\eqref{EL} holds along $\hat{y}$.
\end{proof}

The next result is a direct consequence of Theorem~\ref{teorem0}.

\begin{cor}[The $h$-Euler--Lagrange equation
-- \textrm{cf.}, \textrm{e.g.}, \cite{CD:Bohner:2004}]
\label{ELCor} Let $\mathbb{T}$ be the time scale $h \mathbb{Z}$,
$h > 0$, with the forward jump operator $\sigma$ and the delta
derivative $\Delta$. Assume $a, b \in \mathbb{T}$, $a < b$. If
$\hat{y}$ is a solution to the problem
\begin{equation*}
\mathcal{L}[y(\cdot)]=\int_{a}^{b}L(t,y^{\sigma}(t),y^\Delta(t))\Delta
t \longrightarrow \min, \  y(a)=A, \  y(b)=B\, ,
\end{equation*}
then the equality
$L_u(t,\hat{y}^{\sigma}(t),\hat{y}^\Delta(t))-\left(L_v(t,\hat{y}^{\sigma}(t),\hat{y}^\Delta(t))\right)^\Delta
=0$ holds for all $t\in\mathbb{T}^{\kappa^2}$.
\end{cor}
\begin{proof}
Choose $\alpha=1$ and a $L$ that does not depend on $w$ in
Theorem~\ref{teorem0}.
\end{proof}

\begin{remark}
If we take $h=1$ in Corollary~\ref{ELCor} we have that
$$L_u(t,\hat{y}^{\sigma}(t),\Delta\hat{y}(t))-\Delta L_v(t,\hat{y}^{\sigma}(t),\Delta\hat{y}(t)) =0,$$
holds for all $t\in\mathbb{T}^{\kappa^2}$ (see also Theorem
\ref{teoremadisc}).
\end{remark}


\subsubsection{Natural boundary conditions}

If the initial condition $y(a)=A$ is not present in problem
\eqref{naosei7} (\textrm{i.e.}, $y(a)$ is free), besides the
$h$-fractional Euler--Lagrange equation \eqref{EL} the following
supplementary condition must be fulfilled:
\begin{multline}\label{rui11}
-h^\gamma L_v[\hat{y}](a)+\frac{\gamma}{\Gamma(\gamma+1)}\left(
\int_{a}^{b}(t+\gamma h-a)_h^{(\gamma-1)}L_v[\hat{y}](t)\Delta t\right.\\
\left.-\int_{\sigma(a)}^{b}(t+\gamma
h-\sigma(a))_h^{(\gamma-1)}L_v[\hat{y}](t)\Delta t\right)+
L_w[\hat{y}](a)=0.
\end{multline}
Similarly, if $y(b)=B$ is not present in \eqref{naosei7} ($y(b)$
is free), the extra condition
\begin{multline}\label{rui22}
h L_u[\hat{y}](\rho(b))+h^\gamma L_v[\hat{y}](\rho(b))-h^\nu L_w[\hat{y}](\rho(b))\\
+\frac{\nu}{\Gamma(\nu+1)}\left(\int_{a}^{b}(b+\nu
h-\sigma(t))_h^{(\nu-1)}L_w[\hat{y}](t)\Delta t \right.\\ \left.
-\int_{a}^{\rho(b)}(\rho(b)+\nu
h-\sigma(t))_h^{(\nu-1)}L_w[\hat{y}](t)\Delta t\right)=0
\end{multline}
is added to Theorem~\ref{teorem0}. The proofs of these facts are
immediate and we don't write them here. We just note that the
first term in \eqref{rui22} arises from the first term of the
left-hand side of \eqref{rui33}. Equalities \eqref{rui11} and
\eqref{rui22} are called \emph{natural boundary
conditions}\index{Natural boundary conditions}.


\subsubsection{Second order optimality condition}

We now obtain a second order necessary condition for problem
\eqref{naosei7}, \textrm{i.e.}, we prove a Legendre optimality
type condition for the fractional $h$-difference
setting\index{Legendre's necessary condition!discrete fractional
case}.
\begin{theorem}[The $h$-fractional Legendre necessary condition]
\label{teorem1} If $\hat{y}\in\mathcal{F}_\mathbb{T}$ is a local
minimum for problem \eqref{naosei7}, then the inequality
\begin{equation}
\label{eq:LC}
\begin{split}
h^2
&L_{uu}[\hat{y}](t)+2h^{\gamma+1}L_{uv}[\hat{y}](t)+2h^{\nu+1}(\nu-1)L_{uw}[\hat{y}](t)
+h^{2\gamma}(\gamma -1)^2 L_{vv}[\hat{y}](\sigma(t))\\
&+2h^{\nu+\gamma}(\gamma-1)L_{vw}[\hat{y}](\sigma(t))+2h^{\nu+\gamma}(\nu-1)L_{vw}[\hat{y}](t)+h^{2\nu}(\nu-1)^2 L_{ww}[\hat{y}](t)\\
&+h^{2\nu}L_{ww}[\hat{y}](\sigma(t))
+\int_{a}^{t}h^3L_{ww}[\hat{y}](s)\left(\frac{\nu(1-\nu)}{\Gamma(\nu+1)}(t+\nu
h - \sigma(s))_h^{(\nu-2)}\right)^2\Delta s\\
&+h^{\gamma}L_{vv}[\hat{y}](t)
+\int_{\sigma(\sigma(t))}^{b}h^3L_{vv}[\hat{y}](s)\left(\frac{\gamma(\gamma-1)}{\Gamma(\gamma+1)}(s+\gamma
h -\sigma(\sigma(t)))_h^{(\gamma-2)}\right)^2\Delta s \geq 0
\end{split}
\end{equation}
holds for all $t\in\mathbb{T}^{\kappa^2}$, where
$[\hat{y}](t)=(t,\hat{y}^{\sigma}(t),{_a}\Delta_t^\alpha
\hat{y}(t),{_t}\Delta_b^\beta\hat{y}(t))$.
\end{theorem}
\begin{proof}
By the hypothesis of the theorem, and letting $\Phi$ be as in
\eqref{fi}, we have as necessary optimality condition that
$\Phi''(0)\geq 0$ for an arbitrary admissible variation
$\eta(\cdot)$. Inequality $\Phi''(0)\geq 0$ is equivalent to
\begin{multline}
\label{des1}
\int_{a}^{b}\left[L_{uu}[\hat{y}](t)(\eta^\sigma(t))^2
+2L_{uv}[\hat{y}](t)\eta^\sigma(t){_a}\Delta_h^\alpha\eta(t)
+2L_{uw}[\hat{y}](t)\eta^\sigma(t){_h}\Delta_b^\beta\eta(t)\right.\\
\left. +L_{vv}[\hat{y}](t)({_a}\Delta_h^\alpha\eta(t))^2
+2L_{vw}[\hat{y}](t){_a}\Delta_h^\alpha\eta(t){_h}\Delta_b^\beta\eta(t)
+L_{ww}(t)({_h}\Delta_b^\beta\eta(t))^2\right]\Delta t\geq 0.
\end{multline}
Let $\tau\in\mathbb{T}^{\kappa^2}$ be arbitrary, and choose
$\eta:\mathbb{T}\rightarrow\mathbb{R}$ given by
$$\eta(t) =
\left\{
\begin{array}{ll}
h & \mbox{if $t=\sigma(\tau)$};\\
0 & \mbox{otherwise}.\end{array} \right.$$
It follows that
$\eta(a)=\eta(b)=0$, \textrm{i.e.}, $\eta$ is an admissible
variation. Using \eqref{naosei1} we get
\begin{equation*}
\begin{split}
\int_{a}^{b}&\left[L_{uu}[\hat{y}](t)(\eta^\sigma(t))^2
+2L_{uv}[\hat{y}](t)\eta^\sigma(t){_a}\Delta_h^\alpha\eta(t)
+L_{vv}[\hat{y}](t)({_a}\Delta_h^\alpha\eta(t))^2\right]\Delta t\\
&=\int_{a}^{b}\Biggl[L_{uu}[\hat{y}](t)(\eta^\sigma(t))^2\\
&\qquad\quad +2L_{uv}[\hat{y}](t)\eta^\sigma(t)\left(h^\gamma
\eta^{\Delta}(t)+
\frac{\gamma}{\Gamma(\gamma+1)}\int_{a}^{t}(t+\gamma h
-\sigma(s))_h^{(\gamma-1)}\eta^{\Delta}(s)\Delta s\right)\\
&\qquad\quad +L_{vv}[\hat{y}](t)\left(h^\gamma \eta^{\Delta}(t)
+\frac{\gamma}{\Gamma(\gamma+1)}\int_{a}^{t}(t+\gamma h
-\sigma(s))_h^{(\gamma-1)}\eta^{\Delta}(s)\Delta s\right)^2\Biggr]\Delta t\\
&=h^3L_{uu}[\hat{y}](\tau)+2h^{\gamma+2}L_{uv}[\hat{y}](\tau)+h^{\gamma+1}L_{vv}[\hat{y}](\tau)\\
&\quad
+\int_{\sigma(\tau)}^{b}L_{vv}[\hat{y}](t)\left(h^\gamma\eta^{\Delta}(t)
+\frac{\gamma}{\Gamma(\gamma+1)}\int_{a}^{t}(t+\gamma
h-\sigma(s))_h^{(\gamma-1)}\eta^{\Delta}(s)\Delta s\right)^2\Delta
t.
\end{split}
\end{equation*}
Observe that
\begin{multline*}
h^{2\gamma+1}(\gamma -1)^2 L_{vv}[\hat{y}](\sigma(\tau))\\
+\int_{\sigma^2(\tau)}^{b}L_{vv}[\hat{y}](t)\left(\frac{\gamma}{\Gamma(\gamma+1)}\int_{a}^{t}(t+\gamma
h-\sigma(s))_h^{(\gamma-1)}\eta^{\Delta}(s)\Delta s\right)^2\Delta t\\
=\int_{\sigma(\tau)}^{b}L_{vv}[\hat{y}](t)\left(h^\gamma
\eta^\Delta(t)+\frac{\gamma}{\Gamma(\gamma+1)}\int_{a}^{t}(t+\gamma
h-\sigma(s))_h^{(\gamma-1)}\eta^{\Delta}(s)\Delta s\right)^2\Delta
t.
\end{multline*}
Let $t\in[\sigma^2(\tau),\rho(b)]\cap h\mathbb{Z}$. Since
\begin{equation}
\label{rui10}
\begin{split}
\frac{\gamma}{\Gamma(\gamma+1)}&\int_{a}^{t}(t+\gamma h -\sigma(s))_h^{(\gamma-1)}\eta^{\Delta}(s)\Delta s\\
&= \frac{\gamma}{\Gamma(\gamma+1)}\left[\int_{a}^{\sigma(\tau)}(t+\gamma h-\sigma(s))_h^{(\gamma-1)}\eta^{\Delta}(s)\Delta s\right.\\
&\qquad\qquad\qquad\qquad \left.+\int_{\sigma(\tau)}^{t}(t+\gamma h-\sigma(s))_h^{(\gamma-1)}\eta^{\Delta}(s)\Delta s\right]\\
&=h\frac{\gamma}{\Gamma(\gamma+1)}\left[(t+\gamma h-\sigma(\tau))_h^{(\gamma-1)}-(t+\gamma h-\sigma(\sigma(\tau)))_h^{(\gamma-1)}\right]\\
&=\frac{\gamma}{\Gamma(\gamma+1)}h\left[h^{\gamma-1}\frac{\Gamma\left(\frac{t-\tau}{h}+\gamma\right)}{\Gamma\left(\frac{t-\tau}{h}+1\right)}
-h^{\gamma-1}\frac{\Gamma\left(\frac{t-\tau}{h}+\gamma-1\right)}{\Gamma\left(\frac{t-\tau}{h}\right)}\right]\\
&=\frac{\gamma h^\gamma}{\Gamma(\gamma+1)}\left[
\frac{\left(\frac{t-\tau}{h}+\gamma-1\right)\Gamma\left(\frac{t-\tau}{h}+\gamma-1\right)
-\left(\frac{t-\tau}{h}\right)\Gamma\left(\frac{t-\tau}{h}+\gamma-1\right)}
{\left(\frac{t-\tau}{h}\right)\Gamma\left(\frac{t-\tau}{h}\right)}\right]\\
&=\frac{\gamma
h^\gamma}{\Gamma(\gamma+1)}\left[\frac{(\gamma-1)\Gamma\left(\frac{t-\tau}{h}
+\gamma-1\right)}{\Gamma\left(\frac{t-\tau}{h}+1\right)}\right]\\
&=h^{2}\frac{\gamma(\gamma-1)}{\Gamma(\gamma+1)}(t+\gamma h
-\sigma(\sigma(\tau)))_h^{(\gamma-2)},
\end{split}
\end{equation}
we conclude that
\begin{multline*}
\int_{\sigma^2(\tau)}^{b}L_{vv}[\hat{y}](t)\left(\frac{\gamma}{\Gamma(\gamma+1)}\int_{a}^{t}(t
+\gamma h-\sigma(s))_h^{(\gamma-1)}\eta^{\Delta}(s)\Delta s\right)^2\Delta t\\
=\int_{\sigma^2(\tau)}^{b}L_{vv}[\hat{y}](t)\left(h^2\frac{\gamma(\gamma-1)}{\Gamma(\gamma+1)}(t
+\gamma h-\sigma^2(\tau))_h^{(\gamma-2)}\right)^2\Delta t.
\end{multline*}
Note that we can write
${_t}\Delta_b^\beta\eta(t)=-{_h}\Delta_{\rho(b)}^{-\nu}
\eta^\Delta(t-\nu h)$ because $\eta(b)=0$. It is not difficult to
see that the following equality holds:
\begin{equation*}
\begin{split}
\int_{a}^{b}2L_{uw}[\hat{y}](t)\eta^\sigma(t){_h}\Delta_b^\beta\eta(t)\Delta
t
&=-\int_{a}^{b}2L_{uw}[\hat{y}](t)\eta^\sigma(t){_h}\Delta_{\rho(b)}^{-\nu}
\eta^\Delta(t-\nu h)\Delta t\\
&=2h^{2+\nu}L_{uw}[\hat{y}](\tau)(\nu-1) \, .
\end{split}
\end{equation*}
Moreover,
\begin{equation*}
\begin{split}
\int_{a}^{b} &2L_{vw}[\hat{y}](t){_a}\Delta_h^\alpha\eta(t){_h}\Delta_b^\beta\eta(t)\Delta t\\
&=-2\int_{a}^{b}L_{vw}[\hat{y}](t)\left\{\left(h^\gamma\eta^{\Delta}(t)+\frac{\gamma}{\Gamma(\gamma+1)}
\cdot\int_{a}^{t}(t+\gamma h-\sigma(s))_h^{(\gamma-1)}\eta^{\Delta}(s)\Delta s\right)\right.\\
&\qquad\qquad
\left.\cdot\left[h^\nu\eta^{\Delta}(t)+\frac{\nu}{\Gamma(\nu+1)}\int_{\sigma(t)}^{b}(s
+\nu h-\sigma(t))_h^{(\nu-1)}\eta^{\Delta}(s)\Delta s\right]\right\}\Delta t\\
&=2h^{\gamma+\nu+1}(\nu-1)L_{vw}[\hat{y}](\tau)+2h^{\gamma+\nu+1}(\gamma-1)L_{vw}[\hat{y}](\sigma(\tau)).
\end{split}
\end{equation*}
Finally, we have that
\begin{equation*}
\begin{split}
&\int_{a}^{b} L_{ww}[\hat{y}](t)({_h}\Delta_b^\beta\eta(t))^2\Delta t\\
&=\int_{a}^{b}L_{ww}[\hat{y}](t)\left[h^\nu\eta^{\Delta}(t)+\frac{\nu}{\Gamma(\nu+1)}\int_{\sigma(t)}^{b}(s+\nu
h
-\sigma(t))_h^{(\nu-1)}\eta^{\Delta}(s)\Delta s\right]^2\Delta t\\
&=\int_{a}^{\sigma(\sigma(\tau))}L_{ww}[\hat{y}](t)\left[h^\nu\eta^{\Delta}(t)+\frac{\nu}{\Gamma(\nu+1)}
\int_{\sigma(t)}^{b}(s+\nu h-\sigma(t))_h^{(\nu-1)}\eta^{\Delta}(s)\Delta s\right]^2\Delta t\\
&=\int_{a}^{\tau}L_{ww}[\hat{y}](t)\left[\frac{\nu}{\Gamma(\nu+1)}\int_{\sigma(t)}^{b}(s
+\nu h-\sigma(t))_h^{(\nu-1)}\eta^{\Delta}(s)\Delta s\right]^2\Delta t\\
&\qquad +hL_{ww}[\hat{y}](\tau)(h^\nu-\nu h^\nu)^2+h^{2\nu+1}L_{ww}[\hat{y}](\sigma(\tau))\\
&=\int_{a}^{\tau}L_{ww}[\hat{y}](t)\left[h\frac{\nu}{\Gamma(\nu+1)}\left\{(\tau+\nu
h
-\sigma(t))_h^{(\nu-1)}-(\sigma(\tau)+\nu h-\sigma(t))_h^{(\nu-1)}\right\}\right]^2\\
&\qquad + hL_{ww}[\hat{y}](\tau)(h^\nu-\nu
h^\nu)^2+h^{2\nu+1}L_{ww}[\hat{y}](\sigma(\tau)).
\end{split}
\end{equation*}
Similarly as we did in \eqref{rui10}, we can prove that
\begin{multline*}
h\frac{\nu}{\Gamma(\nu+1)}\left\{(\tau+\nu
h-\sigma(t))_h^{(\nu-1)}-(\sigma(\tau)+\nu
h-\sigma(t))_h^{(\nu-1)}\right\}\\
=h^{2}\frac{\nu(1-\nu)}{\Gamma(\nu+1)}(\tau+\nu
h-\sigma(t))_h^{(\nu-2)}.
\end{multline*}
Thus, we have that inequality \eqref{des1} is equivalent to
\begin{multline}
\label{des2}
h\Biggl\{h^2L_{uu}[\hat{y}](t)+2h^{\gamma+1}L_{uv}[\hat{y}](t)
+h^{\gamma}L_{vv}[\hat{y}](t)+L_{vv}(\sigma(t))(\gamma h^\gamma-h^\gamma)^2\\
+\int_{\sigma(\sigma(t))}^{b}h^3L_{vv}(s)\left(\frac{\gamma(\gamma-1)}{\Gamma(\gamma+1)}(s
+\gamma h -\sigma(\sigma(t)))_h^{(\gamma-2)}\right)^2\Delta s\\
+2h^{\nu+1}L_{uw}[\hat{y}](t)(\nu-1)+2h^{\gamma+\nu}(\nu-1)L_{vw}[\hat{y}](t)\\
+2h^{\gamma+\nu}(\gamma-1)L_{vw}(\sigma(t))+h^{2\nu}L_{ww}[\hat{y}](t)(1-\nu)^2+h^{2\nu}L_{ww}[\hat{y}](\sigma(t))\\
+\int_{a}^{t}h^3L_{ww}[\hat{y}](s)\left(\frac{\nu(1-\nu)}{\Gamma(\nu+1)}(t+\nu
h - \sigma(s))^{\nu-2}\right)^2\Delta s\Biggr\}\geq 0.
\end{multline}
Because $h>0$, \eqref{des2} is equivalent to \eqref{eq:LC}. The
theorem is proved.
\end{proof}

The next result is a simple corollary of Theorem~\ref{teorem1}.
\begin{cor}[The $h$-Legendre necessary condition -- \textrm{cf.} Result~1.3 of \cite{CD:Bohner:2004}]
\label{CorDis:Bohner} Let $\mathbb{T}$ be the time scale $h
\mathbb{Z}$, $h > 0$, with the forward jump operator $\sigma$ and
the delta derivative $\Delta$. Assume $a, b \in \mathbb{T}$, $a <
b$. If $\hat{y}$ is a solution to the problem
\begin{equation*}
\mathcal{L}[y(\cdot)]=\int_{a}^{b}L(t,y^{\sigma}(t),y^\Delta(t))\Delta
t \longrightarrow \min, \  y(a)=A, \ y(b)=B \, ,
\end{equation*}
then the inequality
\begin{equation}
\label{LNCBohner}
h^2L_{uu}[\hat{y}](t)+2hL_{uv}[\hat{y}](t)+L_{vv}[\hat{y}](t)+L_{vv}[\hat{y}](\sigma(t))
\geq 0
\end{equation}
holds for all $t\in\mathbb{T}^{\kappa^2}$, where
$[\hat{y}](t)=(t,\hat{y}^{\sigma}(t),\hat{y}^\Delta(t))$.
\end{cor}
\begin{proof}
Choose $\alpha=1$ and a Lagrangian $L$ that does not depend on
$w$. Then, $\gamma=0$ and the result follows immediately from
Theorem~\ref{teorem1}.
\end{proof}

\begin{remark}
When $h\rightarrow 0$ we get $\sigma(t) = t$ and inequality
\eqref{LNCBohner} coincides with Legendre's classical necessary
optimality condition $L_{vv}[\hat{y}](t) \ge 0$ (\textrm{cf.}
Theorem \ref{teor2}).
\end{remark}


\section{Examples}
\label{sec2}

In this section we present some illustrative examples.
Calculations were done using Maxima Software.

\begin{example}
\label{ex:2} Let us consider the following problem:
\begin{equation}
\label{eq:ex2} \mathcal{L}[y(\cdot)]=\frac{1}{2} \int_{0}^{1}
\left({_0}\Delta_h^{\frac{3}{4}} y(t)\right)^2\Delta t
\longrightarrow \min \, , \quad y(0)=0 \, , \quad y(1)=1 \, .
\end{equation}
We consider (\ref{eq:ex2}) with different values of $h$. Numerical
results show that when $h$ tends to zero the $h$-fractional
Euler--Lagrange extremal tends to the fractional continuous
extremal: when $h \rightarrow 0$ (\ref{eq:ex2}) tends to the
fractional continuous variational problem in the
Riemann--Liouville sense studied in \cite[Example~1]{agr0}, with
solution given by
\begin{equation}
\label{solEx2}
y(t)=\frac{1}{2}\int_0^t\frac{dx}{\left[(1-x)(t-x)\right]^{\frac{1}{4}}}
\, .
\end{equation}
This is illustrated in Figure~\ref{Fig:2}.
\begin{figure}[ht]
\begin{center}
\includegraphics[scale=0.45]{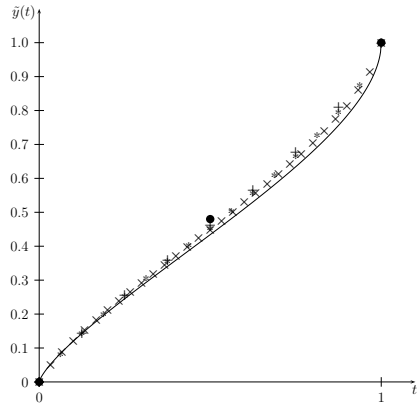}
  \caption{Extremal $\tilde{y}(t)$ for problem of Example~\ref{ex:2}
  with different values of $h$:
  $h=0.50$ ($\bullet$); $h=0.125$ ($+$);
  $h=0.0625$ ($\ast$); $h=1/30$ ($\times$).
  The continuous line represent function
  (\ref{solEx2}).}\label{Fig:2}
\end{center}
\end{figure}
In this example for each value of $h$ there is a unique
$h$-fractional Euler--Lagrange extremal, solution of \eqref{EL},
which always verifies the $h$-fractional Legendre necessary
condition \eqref{eq:LC}.
\end{example}

\begin{example}
\label{ex:1} Let us consider the following problem:
\begin{equation}
\label{eq:ex1} \mathcal{L}[y(\cdot)]=\int_{0}^{1}
\left[\frac{1}{2}\left({_0}\Delta_h^\alpha
y(t)\right)^2-y^{\sigma}(t)\right]\Delta t \longrightarrow \min \,
, \quad y(0) = 0 \, , \quad y(1) = 0 \, .
\end{equation}
We begin by considering problem (\ref{eq:ex1}) with a fixed value
for $\alpha$ and different values of $h$. The extremals
$\tilde{y}$ are obtained using our Euler--Lagrange equation
(\ref{EL}). As in Example~\ref{ex:2} the numerical results show
that when $h$ tends to zero the extremal of the problem tends to
the extremal of the corresponding continuous fractional problem of
the calculus of variations in the Riemann--Liouville sense. More
precisely, when $h$ approximates zero problem (\ref{eq:ex1}) tends
to the fractional continuous problem studied in
\cite[Example~2]{agr2}. For $\alpha=1$ and $h \rightarrow 0$ the
extremal of (\ref{eq:ex1}) is given by $y(t)=\frac{1}{2} t (1-t)$,
which coincides with the extremal of the classical problem of the
calculus of variations
\begin{equation*}
\mathcal{L}[y(\cdot)]=\int_{0}^{1} \left(\frac{1}{2}
y'(t)^2-y(t)\right) dt \longrightarrow \min \, , \quad y(0) = 0 \,
, \quad y(1) = 0 \, .
\end{equation*}
This is illustrated in Figure~\ref{Fig:0} for $h = \frac{1}{2^i}$,
$i = 1, 2, 3, 4$.
\begin{figure}[ht]
\begin{minipage}[b]{0.45\linewidth}
\begin{center}
\includegraphics[scale=0.45]{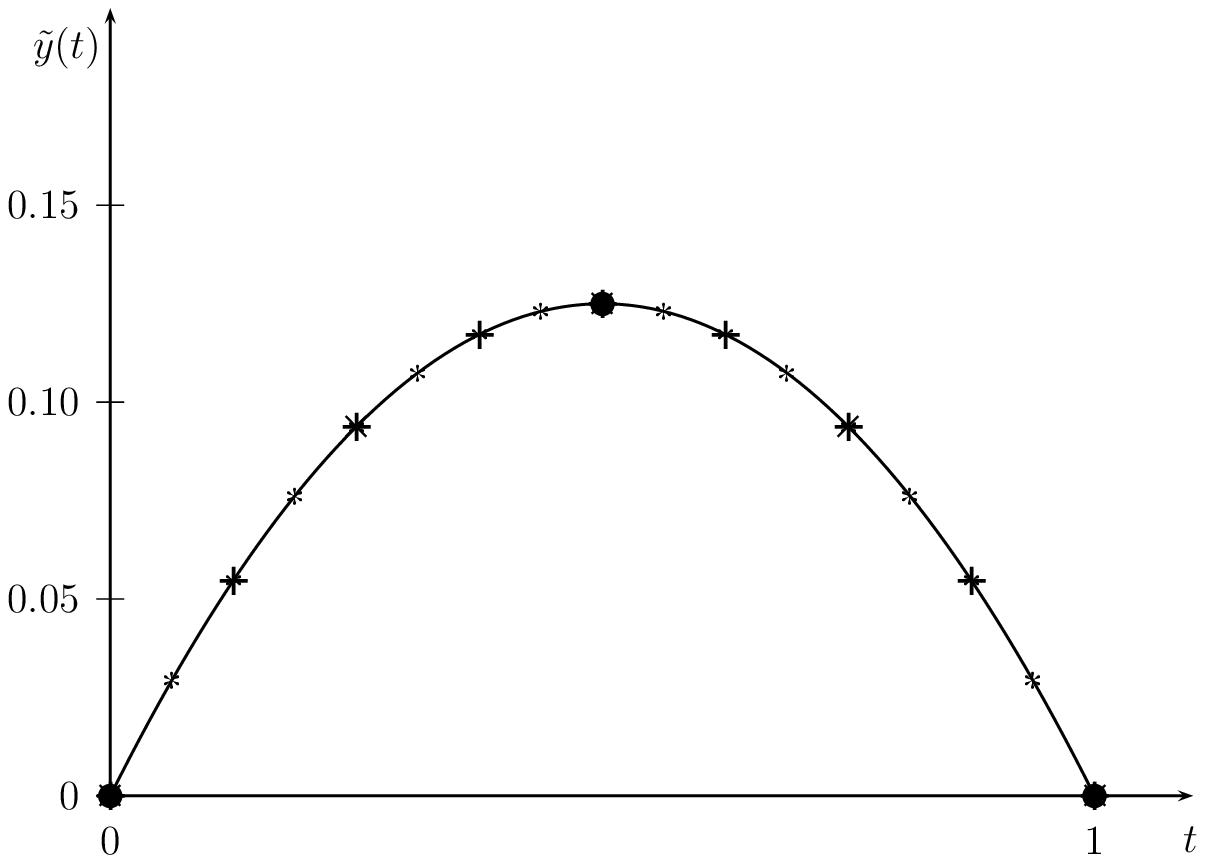}
  \caption{Extremal $\tilde{y}(t)$ for problem \eqref{eq:ex1}
  with $\alpha=1$ and different values of $h$:
  $h=0.5$ ($\bullet$); $h=0.25$ ($\times$);
  $h=0.125$ ($+$); $h=0.0625$ ($\ast$).}\label{Fig:0}
\end{center}
\end{minipage}
\hspace{0.05cm}
\begin{minipage}[b]{0.45\linewidth}
\begin{center}
\includegraphics[scale=0.45]{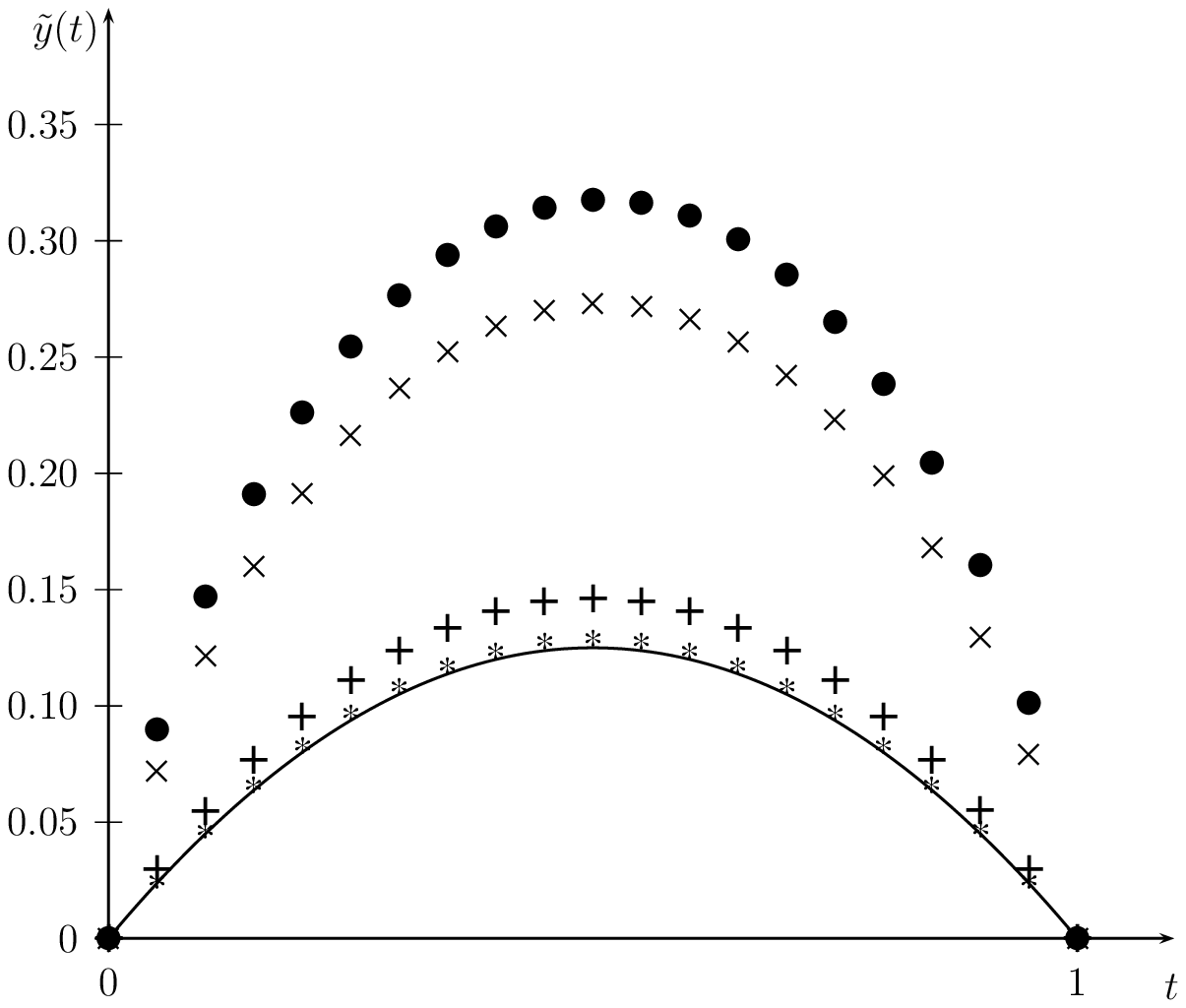}
  \caption{Extremal $\tilde{y}(t)$ for \eqref{eq:ex1}
with $h=0.05$ and different values of $\alpha$: $\alpha=0.70$
($\bullet$); $\alpha=0.75$ ($\times$); $\alpha=0.95$ ($+$);
$\alpha=0.99$ ($\ast$). The continuous line is $y(t)=\frac{1}{2} t
(1-t)$.}\label{Fig:1}
\end{center}
\end{minipage}
\end{figure}
In this example, for each value of $\alpha$ and $h$, we only have
one extremal (we only have one solution to (\ref{EL}) for each
$\alpha$ and $h$). Our Legendre condition \eqref{eq:LC} is always
verified along the extremals. Figure~\ref{Fig:1} shows the
extremals of problem \eqref{eq:ex1} for a fixed value of $h$
($h=1/20$) and different values of $\alpha$. The numerical results
show that when $\alpha$ tends to one the extremal tends to the
solution of the classical (integer order) discrete-time problem.
\end{example}

Our last example shows that the $h$-fractional Legendre necessary
optimality condition can be a very useful tool. In
Example~\ref{ex:3} we consider a problem for which the
$h$-fractional Euler--Lagrange equation gives several candidates
but just a few of them verify the Legendre condition
\eqref{eq:LC}.

\begin{example}
\label{ex:3} Let us consider the following problem:
\begin{equation}
\label{eq:ex3} \mathcal{L}[y(\cdot)]=\int_{a}^{b}
\left({_a}\Delta_h^\alpha
y(t)\right)^3+\theta\left({_h}\Delta_b^\alpha y(t)\right)^2\Delta
t \longrightarrow \min \, , \quad y(a)=0 \, , \quad y(b)=1 \, .
\end{equation}
For $\alpha=0.8$, $\beta=0.5$, $h=0.25$, $a=0$, $b=1$, and
$\theta=1$, problem (\ref{eq:ex3}) has eight different
Euler--Lagrange extremals. As we can see on
Table~\ref{candidates:ex3} only two of the candidates verify the
Legendre condition. To determine the best candidate we compare the
values of the functional $\mathcal{L}$ along the two good
candidates. The extremal we are looking for is given by the
candidate number five on Table~\ref{candidates:ex3}.
\begin{table}
\footnotesize \centering
\begin{tabular}{|c|c|c|c|c|c|}\hline
\# & $\tilde{y}\left(\frac{1}{4}\right)$ &
$\tilde{y}\left(\frac{1}{2}\right)$
& $\tilde{y}\left(\frac{3}{4}\right)$ & $\mathcal{L}(\tilde{y})$ & Legendre condition \eqref{eq:LC}\\
\hline
         1 & -0.5511786 &  0.0515282 &  0.5133134 &  9.3035911 &         Not verified \\
\hline
         2 &  0.2669091 &  0.4878808 &  0.7151924 &  2.0084203 &        Verified \\
\hline
         3 & -2.6745703 &  0.5599360 & -2.6730125 & 698.4443232 &         Not verified \\
\hline
         4 &  0.5789976 &  1.0701515 &  0.1840377 & 12.5174960 &         Not verified \\
\hline
         5 &  1.0306820 &  1.8920322 &  2.7429222 & -32.7189756 &        Verified \\
\hline
         6 &  0.5087946 & -0.1861431 &  0.4489196 & 10.6730959 &         Not verified \\
\hline
         7 &  4.0583690 & -1.0299054 & -5.0030989 & 2451.7637948 &         Not verified \\
\hline
         8 & -1.7436106 & -3.1898449 & -0.8850511 & 238.6120299 &         Not verified \\
\hline
  \end{tabular}
\smallskip
  \caption{There exist 8 Euler-Lagrange extremals for problem \eqref{eq:ex3}
  with $\alpha=0.8$, $\beta=0.5$, $h=0.25$, $a=0$, $b=1$, and $\theta=1$,
  but only 2 of them satisfy the fractional Legendre condition \eqref{eq:LC}.}
  \label{candidates:ex3}
\end{table}
\begin{table}
\footnotesize \centering
\begin{tabular}{|c|c|c|c|c|c|c|} \hline
\# & $\tilde{y}(0.1)$ & $\tilde{y}(0.2)$ & $\tilde{y}(0.3)$ &
$\tilde{y}(0.4)$ &  $\mathcal{L}(\tilde{y})$ & \eqref{eq:LC}\\
\hline

         1 & -0.305570704 & -0.428093486 & 0.223708338 & 0.480549114 & 12.25396166 &         No \\\hline

         2 & -0.427934654 & -0.599520948 & 0.313290997 & -0.661831134 & 156.2317667 &         No \\\hline

         3 & 0.284152257 & -0.227595659 & 0.318847274 & 0.531827387 & 8.669645848 &         No \\\hline

         4 & -0.277642565 & 0.222381632 & 0.386666793 & 0.555841555 & 6.993518478 &         No \\\hline

         5 & 0.387074742 & -0.310032839 & 0.434336603 & -0.482903047 & 110.7912605 &         No \\\hline

         6 & 0.259846344 & 0.364035314 & 0.463222456 & 0.597907505 & 5.104389191 &        Yes \\\hline

         7 & -0.375094681 & 0.300437245 & 0.522386246 & -0.419053781 & 93.95316858 &         No \\\hline

         8 & 0.343327771 & 0.480989769 & 0.61204299 & -0.280908953 & 69.23497954 &         No \\\hline

         9 & 0.297792192 & 0.417196073 & -0.218013689 & 0.460556635 & 14.12227593 &         No \\\hline

        10 & 0.41283304 & 0.578364133 & -0.302235104 & -0.649232892 & 157.8272685 &         No \\\hline

        11 & -0.321401682 & 0.257431098 & -0.360644857 & 0.400971272 & 19.87468886 &         No \\\hline

        12 & 0.330157414 & -0.264444122 & -0.459803086 & 0.368850105 & 24.84475504 &         No \\\hline

        13 & -0.459640837 & 0.368155651 & -0.515763025 & -0.860276767 & 224.9964788 &         No \\\hline

        14 & -0.359429958 & -0.50354835 & -0.640748011 & 0.294083676 & 34.43515839 &         No \\\hline

        15 & 0.477760586 & -0.382668914 & -0.66536683 & -0.956478654 & 263.3075289 &         No \\\hline

        16 & -0.541587541 & -0.758744525 & -0.965476394 & -1.246195157 & 392.9592508 &         No \\\hline
\end{tabular}
\smallskip
\caption{There exist 16 Euler-Lagrange extremals for problem
\eqref{eq:ex3}
  with $\alpha=0.3$, $h=0.1$, $a=0$, $b=0.5$, and $\theta=0$,
  but only 1 (candidate \#6) satisfy the fractional Legendre condition \eqref{eq:LC}.}\label{16dados}
\end{table}

For problem (\ref{eq:ex3}) with $\alpha=0.3$, $h=0.1$, $a=0$,
$b=0.5$, and $\theta=0$, we obtain the results of
Table~\ref{16dados}: there exist sixteen Euler--Lagrange extremals
but only one satisfy the fractional Legendre condition. The
extremal we are looking for is given by the candidate number six
on Table~\ref{16dados}.
\end{example}

\section{State of the Art}

Discrete fractional calculus is a theory that is in its infancy.
To the best of the author's knowledge there is no research paper
providing any kind of result for the time scale
$\mathbb{T}=h\mathbb{Z}$, $h>0$. Nevertheless, the three papers
\cite{Atici0,Atici,Miller} are already published and contain some
basic results for the time scale $\mathbb{T}=\mathbb{Z}$ and also
some methods to solve initial value problems with discrete
left-fractional derivatives. None of these papers contain results
involving any subject related with the calculus of variations. We
have already two research papers \cite{BAFT0,BAFT}
within this topic and further investigation is being undertaken by
Nuno R. O. Bastos in his PhD thesis under the supervision
of Delfim F. M. Torres.

The results of this chapter were presented at the ICDEA 2009,
International Conference on Difference Equations and Applications,
Lisbon, Portugal, October 19--23, 2009, in a contributed talk
entitled \emph{Calculus of Variations with Discrete Fractional
Derivatives}.

\clearpage{\thispagestyle{empty}\cleardoublepage}

\chapter{Conclusions and future work}\label{conclusao}

This thesis had two major objectives: the development of the
calculus of variations on a general time scale and the development
of a discrete fractional calculus of variations theory.

For the calculus of variations on time scales we contributed with
a necessary optimality condition for the problem depending on more
than one $\Delta$-derivative (cf. Theorem \ref{E-LHigher}) and
also for the isoperimetric problem (cf. Theorem \ref{T1}). Some
Sturm--Liouville eigenvalue problems and their relation with
isoperimetric problems were also considered (cf. Section
\ref{sub:sec:ep}). In a first step towards the development of
direct methods for the calculus of variations on a general time
scale we derived some useful integral inequalities\footnote{We
remark that the Gronwall--Bellman--Bihari type inequalities
presented in Section \ref{duasvar} for functions of two variables
can be very useful in studying partial delta dynamic equations.}
and solved some variational problems (see Section
\ref{sec:app:CV}). As an ``outsider" consequence of the herein
obtained integral inequalities we were able to prove existence of
solution to an integrodynamic equation (cf. Theorem \ref{teor}).
In fact, it is our purpose to study the existence of solutions to
boundary value problems of the form
\begin{equation}
\begin{gathered}
y^{\Delta^2}(t)=f(t,y^\sigma(t),y^\Delta(t)),\quad t\in[a,b]_\mathbb{T}^{\kappa^2}, \\
y(a)=A,  \ \  y(b)=B,
\end{gathered}
\end{equation}
and apply the possible obtained results to prove existence of
solution to some classes of Euler--Lagrange equations (remember
that the Euler--Lagrange equation is a second order dynamic
equation). To the best of our knowledge no such a study exists for
a general time scale (we don't even are aware of such a study for
discrete Euler--Lagrange equations).

\begin{remark}
While studying integral inequalities and its applications we were
able to derive some results which are not within the scope of the
calculus of variations and optimal control. This work is published
in the three research papers \cite{Ferr0,Ferr1,Ferr5}.
\end{remark}
We would also like to give answer to a deeper question: find the
Euler--Lagrange equation for the higher-order calculus of
variations problem without restriction (H) on the forward jump
operator. We have some work in progress towards this end, namely,
we obtain in \cite{FerrMalTorres}, using the weak maximum
principle on time scales \cite{zeidan2}, the desired
Euler--Lagrange equation but in integral form.

With respect to the discrete fractional calculus theory we proved
some properties for the fractional sum and difference operators in
Section \ref{sec0} and then, in Section \ref{sec1}, we proved the
Euler--Lagrange equation as well as Legendre's necessary
condition. We also indicate, in Section \ref{sec2}, that as the
parameter $h$ tends to zero, the $h$-fractional Euler--Lagrange
extremal seems to converge to the continuous fractional one. This
indicates that the discrete extremal can be used to approximate
the continuous one. One of the subjects that we are willing to
study within the discrete fractional calculus is the existence of
solutions to the Euler--Lagrange equation \eqref{EL}. Note that
this equation contains both the left and the right discrete
fractional derivatives. Since we obtained so far two main results,
there is plenty to do in the development of the theory of discrete
fractional calculus of variations. For example, we want to
discover what one gets with a general variable endpoint
variational problem. One can also think in obtaining some criteria
in order to establish a sufficient condition. Moreover, we want to
find if it is possible to make an analogue study to that done in
Chapter~\ref{chap2} in order to solve directly some discrete
fractional variational problems. We think that this is a fruitful
area with much to be done, in both of theoretical and practical
directions, and therefore we would like to make part of the
efforts to accomplish it.

To close this manuscript it follows a list of the author
publications during his PhD thesis:
\cite{rchid,Ferr0,Ferr1,Ferr8,Ferr2,Ferr3,Ferr4,Ferr5,Ferr6,Ferr7}.

\clearpage{\thispagestyle{empty}\cleardoublepage}


\clearpage{\thispagestyle{empty}\cleardoublepage}

\addcontentsline{toc}{chapter}{Index}

\printindex

\clearpage{\thispagestyle{empty}\cleardoublepage}


\begin{thebibliography}{99}\addcontentsline{toc}{chapter}{References}

\bibitem{Agarwal} R. Agarwal, M. Bohner, D. O'Regan, A. Peterson.
\emph{Dynamic equations on time scales: a survey}, J. Comput.
Appl. Math. {\bf 141} (2002), no.~1-2, 1--26.

\bibitem{inesurvey}
R. Agarwal, M. Bohner\ and\ A. Peterson, Inequalities on time
scales: a survey, Math. Inequal. Appl. {\bf 4} (2001), no.~4,
535--557.

\bibitem{SLP}
R. Agarwal, M. Bohner\ and\ P. J. Y. Wong, Sturm--Liouville
eigenvalue problems on time scales, Appl. Math. Comput. {\bf 99}
(1999), no.~2-3, 153--166.

\bibitem{agr0}
O. P. Agrawal, Formulation of Euler-Lagrange equations for
fractional variational problems, J. Math. Anal. Appl. {\bf 272}
(2002), no.~1, 368--379.

\bibitem{agr2}
O. P. Agrawal, A general finite element formulation for fractional
variational problems, J. Math. Anal. Appl. {\bf 337} (2008),
no.~1, 1--12.

\bibitem{BHAR}
C. D. Ahlbrandt\ and\ B. J. Harmsen, Discrete versions of
continuous isoperimetric problems, J. Differ. Equations Appl. {\bf
3} (1998), no.~5-6, 449--462.

\bibitem{discCV}
C. D. Ahlbrandt and A. C. Peterson. Discrete Hamiltonian Systems:
Difference Equations, Continued Fractions, and Riccati Equations,
volume 16 of Kluwer Texts in the Mathematical Sciences. Kluwer
Academic Publishers, Boston, 1996.

\bibitem{Pach}
E. Akin-Bohner, M. Bohner\ and\ F. Akin, Pachpatte inequalities on
time scales, JIPAM. J. Inequal. Pure Appl. Math. {\bf 6} (2005),
no.~1, Article 6, 23 pp. (electronic).

\bibitem{RicDel}
R. Almeida\ and\ D. F. M. Torres, 
Calculus of variations with fractional derivatives and fractional integrals, 
Appl. Math. Lett. {\bf 22} (2009), no.~12, 1816--1820.
{\tt arXiv:0907.1024}

\bibitem{AlmeidaTorres09}
R. Almeida and D. F. M. Torres, 
Isoperimetric problems on time scales with nabla derivatives, 
Journal of Vibration and Control, Vol. 15 (2009), no.~6, 951--958.
{\tt arXiv:0811.3650}

\bibitem{rchid}
M. R. Sidi Ammi, Rui A. C. Ferreira and D. F. M. Torres,
Diamond-$\alpha$ Jensen's Inequality on Time Scales, 
Journal of Inequalities and Applications, vol. 2008, 
Article ID 576876, 13 pages, 2008.
{\tt arXiv:0712.1680}

\bibitem{AT1}
M. R. Sidi Ammi and D. F. M. Torres, 
Combined dynamic Gr\"{u}ss inequalities on time scales, 
Journal of Mathematical Sciences, Vol. 161, No. 6, 2009, 792--802.
{\tt arXiv:0801.1865}

\bibitem{AT}
M. R. Sidi Ammi and D. F. M. Torres, 
H\"{o}lder's and Hardy's two Dimensional Diamond-alpha Inequalities on Time Scales,
An. Univ. Craiova Ser. Mat. Inform. 37 (2010), no.~1, 1--11.
{\tt arXiv:0910.4115}

\bibitem{econo}
F. M. Atici, D. C. Biles\ and\ A. Lebedinsky, An application of
time scales to economics, Math. Comput. Modelling {\bf 43} (2006),
no.~7-8, 718--726.

\bibitem{Atici0}
F. M. Atici\ and\ P. W. Eloe, A transform method in discrete
fractional calculus, Int. J. Difference Equ. {\bf 2} (2007),
no.~2, 165--176.

\bibitem{Atici}
F. M. Atici\ and\ P. W. Eloe, Initial value problems in discrete
fractional calculus, Proc. Amer. Math. Soc. {\bf 137} (2009),
no.~3, 981--989.

\bibitem{BAFT0}
N. R. O. Bastos, R. A. C. Ferreira and D. F. M. Torres, 
The Fractional Difference Calculus of Variations, 
Symposium on Fractional Signals and Systems Lisbon09, 
M. Ortigueira et. al. (eds.), Lisbon, Portugal, November 4-6, 2009, 15pp.

\bibitem{BAFT}
N. R. O. Bastos, R. A. C. Ferreira and D. F. M. Torres,
Discrete-time fractional variational problems, 
Signal Processing, 2010, doi: 10.1016/j.sigpro.2010.05.001.
{\tt arXiv:1005.0252}

\bibitem{Beesack}
P. R. Beesack, On some Gronwall-type integral inequalities in $n$
independent variables, J. Math. Anal. Appl. {\bf 100} (1984),
no.~2, 393--408.

\bibitem{bihari}
I. Bihari, A generalization of a lemma of Bellman and its
application to uniqueness problems of differential equations, Acta
Math. Acad. Sci. Hungar. {\bf 7} (1956), 81--94.

\bibitem{CD:Bohner:2004} M. Bohner.
\emph{Calculus of variations on time scales}. Dynam. Systems Appl.
{\bf 13} (2004), no.~3-4, pp.~339--349.

\bibitem{BFT}
M. Bohner, R. A. C. Ferreira and D. F. M. Torres, 
Integral Inequalities and their Applications to the Calculus of Variations
on Time Scales, Math. Inequal. Appl. 13 (2010), no. 3, 511--522.
{\tt arXiv:1001.3762}

\bibitem{dicovots}
M. Bohner\ and\ G. Sh. Guseinov, Double integral calculus of
variations on time scales, Comput. Math. Appl. {\bf 54} (2007),
no.~1, 45--57.

\bibitem{livro}
M. Bohner\ and\ A. Peterson, {\it Dynamic equations on time
scales}, Birkh\"auser Boston, Boston, MA, 2001.

\bibitem{livro2}
M. Bohner\ and\ A. Peterson, {\it Advances in dynamic equations on
time scales}, Birkh\"auser Boston, Boston, MA, 2003.

\bibitem{BDum}
M. Bohner and O. Duman. Opial type inequalities for diamond-alpha
derivatives and integrals on time scales. Differ. Equ. Dyn. Syst.,
2009, in press.

\bibitem{TSPORTU}
A. M. C. Brito da Cruz, H. S. Rodrigues and D. F. M. Torres,
Escalas Temporais e Mathematica, Bol. Soc. Port. Mat. No.~62 (2010), 1--18.
{\tt arXiv:0911.3187}

\bibitem{BR}
R. F. Brown, {\it A topological introduction to nonlinear
analysis}, Birkh\"auser Boston, Boston, MA, 1993.

\bibitem{caputo}
M. C. Caputo, Time scales: from nabla calculus to delta calculus
and viceversa via duality, {\tt arXiv:0910.0085}

\bibitem{wsc} W.-S. Cheung,
Integral inequalities and applications to the calculus of
variations, Int. J. Appl. Math. {\bf 9} (2002), no.~1, 85--108.

\bibitem{constantin}
A. Constantin, Solutions globales d'equations differentielles
perturb\'{e}es, C. R. Acad. Sci. Paris, \textbf{320} (1995),
1319–-1322.

\bibitem{contantin}
A. Constantin, Topological transversality: Application to an
integro-differential equation, \emph{J. Math. Anal. Appl.},
\textbf{197}(1) (1996), 855-863.

\bibitem{const1}
A. Constantin, Monotone iterative technique for a nonlinear
integral equation, J. Math. Anal. Appl. {\bf 205} (1997), no.~1,
280--283.

\bibitem{const2}
A. Constantin, Nonlinear alternative: application to an integral
equation, J. Appl. Anal. {\bf 5} (1999), no.~1, 119--123.

\bibitem{motivacao}
M. Denche\ and\ H. Khellaf, Integral inequalities similar to
Gronwall inequality, Electron. J. Differential Equations {\bf
2007}, No. 176, 14 pp. (electronic).

\bibitem{dinu0}
C. Dinu, A weighted Hermite Hadamard inequality for
Steffensen-Popoviciu and Hermite-Hadamard weights on time scales,
An. \c Stiin\c t. Univ. ``Ovidius'' Constan\c ta Ser. Mat. {\bf
17} (2009), no.~1, 77--90.

\bibitem{dinu}
C. Dinu, Hermite-Hadamard inequality on time scales, J. Inequal.
Appl. {\bf 2008}, Art. ID 287947, 24 pp.

\bibitem{otto} O. Dunkel,
Integral Inequalities With Applications to the Calculus of
Variations, Amer. Math. Monthly {\bf 31} (1924), no.~7, 326--337.

\bibitem{El-Nabulsi1}
R. A. El--Nabulsi\ and\ D. F. M. Torres, 
Necessary optimality conditions for fractional action-like integrals of variational
calculus with Riemann-Liouville derivatives of order
$(\alpha,\beta)$, Math. Methods Appl. Sci. {\bf 30} (2007),
no.~15, 1931--1939.
{\tt arXiv:math-ph/0702099}

\bibitem{El-Nabulsi2}
R. A. El--Nabulsi\ and\ D. F. M. Torres, 
Fractional actionlike variational problems, 
J. Math. Phys. {\bf 49} (2008), no.~5, 053521, 7 pp.
{\tt arXiv:0804.4500}

\bibitem{Ferr0}
R. A. C. Ferreira, Constantin integral inequalities on time
scales, Arch. Math. (Basel) {\bf 93} (2009), no.~2, 153--163.

\bibitem{Ferr1}
R. A. C. Ferreira, Existence of solutions to some retarded
integrodifferential equations via the topological transversality
theorem, Nonlinear Analysis \textbf{72} (2010), pp. 3062--3068.
doi: 10.1016/j.na.2009.11.045

\bibitem{Ferr8} R. A. C. Ferreira.
\emph{Existence of solution to an integrodynamic equation}, Dynam.
Systems Appl., to appear.

\bibitem{Ferr2} R. A. C. Ferreira and D. F. M. Torres, 
Remarks on the Calculus of Variations on time scales, 
Int. J. Ecol. Econ. Stat., Vol. 9, Nr. F07, 2007, pp. 65--73.
{\tt arXiv:0706.3152}

\bibitem{Ferr3} R. A. C. Ferreira\ and\ D. F. M. Torres,
Higher-order calculus of variations on time scales,
in {\it Mathematical control theory and finance}, 2008, 149--159, Springer, Berlin.
{\tt arXiv:0706.3141}

\bibitem{Ferr4} R. A. C. Ferreira\ and\ D. F. M. Torres,
Generalizations of Gronwall-Bihari inequalities on time scales,
J. Difference Equ. Appl. {\bf 15} (2009), no.~6, 529--539.
{\tt arXiv:0805.2673}

\bibitem{Ferr5} R. A. C. Ferreira\ and\ D. F. M. Torres,
Generalized retarded integral inequalities,
Appl. Math. Lett. {\bf 22} (2009), no.~6, 876--881.
{\tt arXiv:0803.2294}

\bibitem{Ferr6} R. A. C. Ferreira\ and\ D. F. M. Torres,
Some linear and nonlinear integral inequalities on time scales
in two independent variables, 
Nonlinear Dyn. Syst. Theory {\bf 9} (2009), no.~2, 161--169.
{\tt arXiv:0805.3765}

\bibitem{Ferr7}
R. A. C. Ferreira and D. F. M. Torres, 
Isoperimetric problems of the calculus of variations on time scales, 
Nonlinear Analysis and Optimization II. Optimization, pp.~123--131, 
Contemp. Math., Amer. Math. Soc., Providence, RI, 2010.
{\tt arXiv:0805.0278}

\bibitem{FerrMalTorres}
R. A. C. Ferreira, A. B. Malinowska and D. F. M. Torres,
Optimality conditions for the calculus of variations with
higher-order delta derivatives, submitted.

\bibitem{fol} G. B. Folland, {\it Real analysis},
Second edition, Wiley, New York, 1999.

\bibitem{gastao:delfim}
G. S. F. Frederico\ and\ D. F. M. Torres, 
A formulation of Noether's theorem for fractional problems 
of the calculus of variations, 
J. Math. Anal. Appl. {\bf 334} (2007), no.~2, 834--846.
{\tt arXiv:math/0701187}

\bibitem{gasta1}
G. S. F. Frederico\ and\ D. F. M. Torres, 
Fractional conservation laws in optimal control theory, 
Nonlinear Dynam. {\bf 53} (2008), no.~3, 215--222.
{\tt arXiv:0711.0609}

\bibitem{livroFPT}
A. Granas\ and\ J. Dugundji, {\it Fixed point theory}, Springer,
New York, 2003.

\bibitem{GronwallIne}
T. H. Gronwall, Note on the derivatives with respect to a
parameter of the solutions of a system of differential equations,
Ann. of Math. (2) {\bf 20} (1919), no.~4, 292--296.

\bibitem{H:L:1932} G. H. Hardy, J. E. Littlewood,
Some integral inequalities connected with the calculus of
variations, Q. J. Math., Oxf. Ser. {\bf 3} (1932), 241--252.

\bibitem{inequalities}
G. H. Hardy, J. E. Littlewood\ and\ G. P\'olya, {\it
Inequalities}, 2nd ed., Cambridge University Press, 1952.

\bibitem{Hilger90} S. Hilger,
Analysis on measure chains---a unified approach to continuous and
discrete calculus, Results Math. {\bf 18} (1990), no.~1-2, 18--56.

\bibitem{Hilger97} S. Hilger,
Differential and difference calculus---unified!, Nonlinear Anal.
{\bf 30} (1997), no.~5, 2683--2694.

\bibitem{HilgerThesis}S. Hilger. \emph{Ein Ma{\ss}kettenkalk$\ddot{\mbox{u}}$l
mit Anwendung auf Zentrumsmannigfaltigkeiten}. PhD thesis,
Universit$\ddot{\mbox{a}}$t W$\ddot{\mbox{u}}$rzburg, 1988.

\bibitem{zeidan} R. Hilscher, V. Zeidan.
\emph{Calculus of variations on time scales: weak local piecewise
$C\sp 1\sb {\rm rd}$ solutions with variable endpoints}. J. Math.
Anal. Appl. {\bf 289} (2004), no.~1, pp.~143--166.

\bibitem{zeidan2} R. Hilscher, V. Zeidan.
R. Hilscher\ and\ V. Zeidan, Weak maximum principle and accessory
problem for control problems on time scales, Nonlinear Anal. {\bf
70} (2009), no.~9, 3209--3226.

\bibitem{Tad}
T. Jankowski, The generalized quasilinearization for
integro-differential equations of Volterra type on time scales,
Rocky Mountain J. Math. {\bf 37} (2007), no.~3, 851--864.

\bibitem{book:DCV}
W. G. Kelley\ and\ A. C. Peterson, {\it Difference equations},
Academic Press, Boston, MA, 1991.

\bibitem{khellaf}
H. Khellaf, On integral inequalities for functions of several
independent variables, Electron. J. Differential Equations {\bf
2003}, No. 123, 12 pp. (electronic).

\bibitem{book:Kilbas}
A. A. Kilbas, H. M. Srivastava\ and\ J. J. Trujillo, {\it Theory
and applications of fractional differential equations}, Elsevier,
Amsterdam, 2006.

\bibitem{tisdell}
T. Kulik and C. C. Tisdell, Volterra integral
equations on time scales: Basic qualitative and quantitative
results with applications to initial value problems on unbounded
domains. Int. J. Difference Equ. 3 (2008), no. 1, 103--133.

\bibitem{Liu}
W. Liu, Qu$\tilde{\mbox{o}}$c Anh Ng\^{o}\ and\ W. Chen, A perturbed
Ostrowski-type inequality on time scales for $k$ points for
functions whose second derivatives are bounded, J. Inequal. Appl.
{\bf 2008}, Art. ID 597241, 12 pp.

\bibitem{MalinowTorres}
A. B. Malinowska and Delfim F. M. Torres, 
Necessary and sufficient conditions for local Pareto optimality on time scales, 
Journal of Mathematical Sciences, Vol.~161, No.~6, 2009, 803--810.
{\tt arXiv:0801.2123}

\bibitem{MalTorres}
A. B. Malinowska and D. F. M. Torres, 
On the diamond-alpha Riemann integral and mean value theorems on time scales, 
Dynamic Systems and Applications, Volume 18 (2009), no.~3-4, pp.~469--482.
{\tt arXiv:0804.4420}

\bibitem{MalTorres1}
A. B. Malinowska and D. F. M. Torres, 
Strong minimizers of the calculus of variations on time scales 
and the Weierstrass condition, 
Proceedings of the Estonian Academy of Sciences, Vol.~58, no.~4, 2009, 205--212.
{\tt arXiv:0905.1870}

\bibitem{MalTorres2}
A. B. Malinowska and D. F. M. Torres, 
The delta-nabla calculus of variations, 
Fasc. Math. 44 (2010), 75--83.
{\tt arXiv:0912.0494}

\bibitem{MalTorres3}
A. B. Malinowska and Delfim F. M. Torres, 
Leitmann's direct method of optimization for absolute extrema 
of certain problems of the calculus of variations on time scales, 
Applied Mathematics and Computation, 2010, in press. 
doi: 10.1016/j.amc.2010.01.015
{\tt arXiv:1001.1455}

\bibitem{natorres}
N. Martins, D. F. M. Torres, 
Calculus of variations on time scales with nabla derivatives, 
Nonlinear Analysis: Theory, Methods \& Applications 71 (2009), 
no.~12, pp.~e763--e773.
{\tt arXiv:0807.2596}

\bibitem{mike}
Mikeladze, Sh.E., De la r\'{e}solution num\'{e}rique des \'{e}quations
int\'{e}grales, Bull.Acad.Sci.URSS VII (1935), 255--257 (in Russian).

\bibitem{Miller}
K. S. Miller\ and\ B. Ross, Fractional difference calculus, in
{\it Univalent functions, fractional calculus, and their
applications (K\=oriyama, 1988)}, 139--152, Horwood, Chichester,
1989.

\bibitem{Miller1}
K. S. Miller\ and\ B. Ross, {\it An introduction to the fractional
calculus and fractional differential equations}, Wiley, New York,
1993.

\bibitem{mit} D. S. Mitrinovi\'c, {\it Analytic inequalities},
Springer, New York, 1970.

\bibitem{Okr0}
W. Mydlarczyk\ and\ W. Okrasi\'nski, Nonlinear Volterra integral
equations with convolution kernels, Bull. London Math. Soc. {\bf
35} (2003), no.~4, 484--490.

\bibitem{og}
J. A. Oguntuase, On an inequality of Gronwall, JIPAM. J. Inequal.
Pure Appl. Math. {\bf 2} (2001), no.~1, Article 9, 6 pp.
(electronic).

\bibitem{Okr1}
W. Okrasi\'nski, On the existence and uniqueness of nonnegative
solutions of a certain nonlinear convolution equation, Ann. Polon.
Math. {\bf 36} (1979), no.~1, 61--72.

\bibitem{Okr2}
W. Okrasi\'nski, On a nonlinear convolution equation occurring in
the theory of water percolation, Ann. Polon. Math. {\bf 37}
(1980), no.~3, 223--229.

\bibitem{Bihary}
S. A. \"Ozg\"un, A. Zafer\ and\ B. Kaymak\c calan, Gronwall and
Bihari type inequalities on time scales, in {\it Advances in
difference equations (Veszpr\'em, 1995)}, 481--490, Gordon and
Breach, Amsterdam, 1997.

\bibitem{rev1:r} U. M. \"Ozkan, M. Z. Sarikaya\ and\ H. Yildirim, 
Extensions of certain integral inequalities on time scales, 
Appl. Math. Lett. {\bf 21} (2008), no.~10, 993--1000.

\bibitem{Pch}
B. G. Pachpatte, \emph{Integral and finite difference inequalities
and applications}, Mathematics studies 205, Elsevier, 2006.

\bibitem{Podlubny}
I. Podlubny, {\it Fractional differential equations}, Academic
Press, San Diego, CA, 1999.

\bibitem{diam}
J. W. Rogers Jr. and Q. Sheng, Notes on the diamond-$\alpha$
dynamic derivative on time scales, J. Math. Anal. Appl.,
\textbf{326} (2007), no. 1, 228–-241.

\bibitem{saldisc}
Sh. Salem\ and\ K. R. Raslan, Some new discrete inequalities and
their applications, JIPAM. J. Inequal. Pure Appl. Math. {\bf 5}
(2004), no.~1, Article 2, 9 pp. (electronic).

\bibitem{Sbordone} C. Sbordone,
On some integral inequalities and their applications to the
calculus of variations, Boll. Un. Mat. Ital. C (6) {\bf 5} (1986),
no.~1, 73--94.

\bibitem{diam1}
Q. Sheng, M. Fadag, J. Henderson and J. M. Davis, An exploration
of combined dynamic derivatives on time scales and their
applications, Nonlinear Anal. Real World Appl. \textbf{7} (2006),
no. 3, 395–-413.

\bibitem{TenreiroMachado}
M. F. Silva, J. A. Tenreiro Machado\ and\ R. S. Barbosa, Using
fractional derivatives in joint control of hexapod robots, J. Vib.
Control {\bf 14} (2008), no.~9-10, 1473--1485.

\bibitem{Brunt}
B. van Brunt, {\it The calculus of variations}, Springer, New
York, 2004.

\bibitem{Gronwall}
F.-H. Wong, C.-C. Yeh\ and\ C.-H. Hong, Gronwall inequalities on
time scales, Math. Inequal. Appl. {\bf 9} (2006), no.~1, 75--86.

\bibitem{wong} F.-H. Wong, C.-C. Yeh\ and\ W.-C. Lian,
An extension of Jensen's inequality on time scales, Adv. Dyn.
Syst. Appl. {\bf 1} (2006), no.~1, 113--120.

\bibitem{Xing}
Y. Xing, M. Han\ and\ G. Zheng, Initial value problem for
first-order integro-differential equation of Volterra type on time
scales, Nonlinear Anal. {\bf 60} (2005), no.~3, 429--442.

\bibitem{Zhan}
Z. Zhan, W. Wei\ and\ H. Xu, Hamilton-Jacobi-Bellman equations on
time scales, Math. Comput. Modelling {\bf 49} (2009), no.~9-10,
2019--2028.

\end{thebibliography}
\end{document}